\documentclass[11pt]{article}

\usepackage[letterpaper, margin=1.0in]{geometry}
\usepackage{bm}
\usepackage{xcolor}
\usepackage{setspace}
\makeatletter

\usepackage[sort&compress,numbers]{natbib}
\usepackage{appendix}
\usepackage{enumerate}
\usepackage[ruled,linesnumbered]{algorithm2e}
    \SetKwInOut{Input}{Input}
    \SetKwInOut{Output}{Output}
\usepackage{multirow}
\usepackage{graphicx}
\usepackage{subcaption}
\usepackage{tabularx}
\usepackage{makecell}
\usepackage{booktabs}
\usepackage{cancel}
\usepackage{hyperref}

\usepackage{placeins}   

\usepackage{comment}
\usepackage[normalem]{ulem}

\usepackage{amsmath}
\usepackage{amstext}
\usepackage{amsthm}
\usepackage{amssymb}
\usepackage{bbm}
\usepackage{amsbsy}
\usepackage{mathtools}

    \newtheorem{prop}{Proposition}[section]
    \numberwithin{prop}{section} 
    \newtheorem{lem}{Lemma}[section]
    \numberwithin{lem}{section} 
    \newtheorem{rem}{Remark}[section]
    \numberwithin{rem}{section} 
    
    \numberwithin{cor}{section} 
    \newtheorem{thm}{Theorem}[section]
    \numberwithin{thm}{section} 
    \newtheorem{defn}{Definition}
    \numberwithin{defn}{section}
    
    \numberwithin{fact}{section}

    \theoremstyle{plain}
    \newtheorem*{thm*}{Theorem}
    \newtheorem*{rem*}{Remark}
    
    \newcommand{\beq}{\begin{equation}}
    \newcommand{\eeq}{\end{equation}}
    \newcommand{\beqnn}{\begin{equation*}}
    \newcommand{\eeqnn}{\end{equation*}}
        
    \def\ba {\bm{a}}
    \def\bb {\bm{b}}
    \def\bc {\bm{c}}
    \def\bd {\bm{d}}
    \def\be {\bm{e}}
    \def\bg {\bm{g}}
    \def\bh {\bm{h}} 
    \def\bn {\bm{n}}
    \def\bp {\bm{p}}
    \def\bq {\bm{q}}
    \def\br {\bm{r}}
    \def\bu {\bm{u}}
    \def\bs {\bm{s}}
    \def\bv {\bm{v}}
    \def\bw {\bm{w}}
    \def\bx {\bm{x}}
    \def\by {\bm{y}}
    \def\bz {\bm{z}}
    
    \def\bl {\boldsymbol{\lambda}}
    \def\btheta {\boldsymbol{\theta}}
    \def\bxi {\boldsymbol{\xi}}

    \DeclareRobustCommand{\matr}[1]{\bm{#1}}
    \newcommand{\normsq}[1]{\left\Vert{#1} \right\Vert_{2}^{2}}
    \newcommand{\normtwo}[1]{\left\Vert{#1} \right\Vert_{2}}
    
    \newcommand{\norminf}[1]{\left\Vert{#1} \right\Vert_{\infty}}
    \newcommand{\diff}{\mathop{}\!d}
    \newcommand{\image}[1]{\mathrm{Im}({#1})}

    \DeclareMathOperator*{\argmin}{arg\,min}

    \def\Atop{\matr{A}^{\top}}
    \def\cone{\mathrm{cone}}
    \def\dom {\mathrm{dom~}}        
    \def\gmRm {\Gamma_0(\Rm)}
    \def\gmRn {\Gamma_0(\Rn)}
    \def\interior {\mathrm{int }}   
    \def\proj {\mathrm{proj}}
    \def\ri {\mathrm{ri}}           

    \def\bxsol {\bx^{s}}
    \def\bpsol {\bp^{s}}
    \def\eqset {\mathcal{E}}
    \def\eqcset {\mathcal{E}^{\mathsf{C}}}
    
    \def\Jcbxo {J^{\mathsf{C}}(\bxo)}
    
    \def\N {\mathbb{N}}
    \def\R {\mathbb{R}}
    \def\Rn {\R^n}
    \def\Rm {\R^m}
    
    \def\bxo {\bx_{0}}
    
    \def\exo {\eqset(\bxo)}
    \def\ecxo {\eqcset(\bxo)}

\title{Continuous-time analysis and algorithms for linear and convex quadratic programs via exactly solvable differential inclusions}
\author{Gabriel P. Langlois$^1$ \and Akwum Onwunta$^{2}$ \and J\'er\^ome Darbon$^3$}
\date{%
    $^1$Department of Mathematics, University of Illinois Urbana-Champaign (\url{gp42@illinois.edu}). \textbf{Corresponding author}.\\%
    $^2$Industrial and Systems Engineering, Lehigh University (\url{ako221@lehigh.edu}).\\%
    $^3$Division of Applied Mathematics, Brown University (\url{jerome_darbon@brown.edu}).\\[2ex]%
    \today
}

\begin{document}
\maketitle

\begin{abstract}
Linear and convex quadratic programs are fundamental problems in optimization. While many methods have been developed for solving these programs in large-scale applications, to our knowledge none have attempted to solve their continuous-time gradient flow, a differential inclusion, in full generality without assumptions beyond feasibility. Recent work solved the gradient flows in two special cases, basis pursuit and the lasso, and obtained competitive algorithms. Motivated by these results, we solve the general case exactly. Using the minimal selection principle for monotone differential inclusions, we show that the trajectory of the gradient flow consists of finitely many pieces, each obtained from one nonnegative least-squares problem. Its final piece yields an optimal solution or a certificate of unboundedness. This analysis leads to two algorithms. The first algorithm is exact, terminates after finitely many steps, and applies to linear programs and projections onto polyhedra. The second algorithm, which we call minimal selection descent, discretizes the gradient flow for general convex quadratic programs. Its iterates converge to an optimal solution when one exists and otherwise certify unboundedness. We apply both algorithms to a positivity-preserving limiter for a discontinuous Galerkin scheme and to optimal control problems for the Poisson and heat equations. In the limiter, the first algorithm conserves the total mass more accurately than the state-of-the-art Douglas--Rachford splitting method. On the control problems, minimal selection descent reaches the optimal value to machine precision and is competitive in time with commercial and open-source solvers.
\end{abstract}
\tableofcontents

\section{Introduction}\label{sec:intro}
Linear and convex quadratic programs~\cite{markowitz1956optimization,wolfe1959simplex} are fundamental optimization problems used in many fields, including operations research, economics, robotics, machine learning, and scientific computing~\cite{charnes1957management,tassa2014control,tibshirani1996regression,chambolle2016introduction,pearson2017fast,liu2025optimization}. These problems can be expressed as
\begin{equation}\label{eq:intro_main_problem}
    \min_{\bx \in \Rn} \left\{\frac{1}{2}\langle\bx,\matr{Q}\bx\rangle + \langle\bc,\bx\rangle\right\}  \ \text{subject to} \ \matr{A}\bx \leqslant \bb,
\end{equation}
where $\matr{Q} \in \R^{n \times n}$ is symmetric positive semidefinite (SPSD), $\matr{A} \in \R^{m \times n}$, $\bc \in \Rn$, and $\bb \in \Rm$. Problem~\eqref{eq:intro_main_problem} is a linear program when $\matr{Q} = \matr{0}$ and a quadratic program otherwise. The study of algorithms for these problems began in the 1940s and 1950s, and much of modern optimization theory originates in it~\cite{dantzig1951maximization,cottle2010harry,todd2002many}. In current applications, the dimensions $m$ and $n$ are often large, and researchers have developed many new methods to solve problem~\eqref{eq:intro_main_problem} at that scale~\cite{wright2025optimization,dinh2025real}.

We consider here the continuous-time optimization of problem~\eqref{eq:intro_main_problem} via a \emph{differential inclusion}~\cite{aubin2012differential,brezis1973ope}. Let $C = \{\bx \in \Rn : \matr{A}\bx \leqslant \bb\}$ and assume that $C$ is nonempty. The (nonsmooth) gradient flow of problem~\eqref{eq:intro_main_problem} is the differential inclusion
\begin{equation}\label{eq:intro_di}
\dot{\bx}(\tau) \in -\matr{Q}\bx(\tau) - \bc - N_{C}(\bx(\tau)), \quad \bx(0) = \bxo \in C,
\end{equation}
where $N_{C}(\bx(\tau))$ denotes the normal cone of $C$ at $\bx(\tau)$ (see Table~\ref{tab:notation} in Section~\ref{sec:setup}). At each time $\tau$, the velocity $\dot{\bx}(\tau)$ is an element of the nonempty set on the right-hand side of~\eqref{eq:intro_di}. This differential inclusion has a unique solution, which converges to an optimal solution of problem~\eqref{eq:intro_main_problem} whenever one exists~\cite{aubin2012differential,brezis1973ope}.

Solutions of differential inclusions, however, have rarely been computed exactly, except in simple examples~\cite{hiriart2013convexI,goebel2024set} and for certain linear complementarity systems describing mechanical systems and electrical networks~\cite{heemels2000linear}. Most research has instead focused on the theory of differential inclusions and on the numerical approximation of their solutions in, e.g., nonsmooth analysis, control theory, and optimization~\cite{brezis1973ope,aubin2012differential,dupuis1991lipschitz,dupuis1993dynamical,hiriart2013convexII,attouch2014variational,dontchev1992difference}. Exact solutions have recently been computed for the differential inclusions of basis pursuit and the lasso~\cite{chen2001atomic,tibshirani1996regression} by~\citet{tendero2021algorithm} and~\citet{langlois2025fast}. Both papers used these solutions to analyze the dynamics and to derive algorithms that are exact (numerically up to machine precision) and competitive with established methods, such as the simplex method and coordinate descent.

Motivated by these results, we solve exactly the differential inclusion~\eqref{eq:intro_di} characterizing linear and convex quadratic programs. We use the \emph{minimal selection principle} (Section~\ref{subsec:msp}), which states that at each time, the solution moves along the element of minimal norm in the set on the right-hand side of~\eqref{eq:intro_di}. Our analysis yields two algorithms. Algorithm~\ref{alg:quad_prog} computes the solution exactly when $\matr{Q} = t\matr{I}_{n\times n}$ with $t \geqslant 0$, that is, for linear programs ($t = 0$) and projections onto polyhedra ($t > 0$), and terminates after finitely many iterations. Algorithm~\ref{alg:quad_prog_2}, which we call \emph{minimal selection descent} (MSD), discretizes~\eqref{eq:intro_di} by the explicit Euler method and applies to every SPSD $\matr{Q}$.

\begin{algorithm}[ht!]
  \caption{Computing an optimal solution of problem~\eqref{eq:intro_main_problem} with $\matr{Q} = t\matr{I}_{n\times n}$, or certifying unboundedness.}
  \label{alg:quad_prog}
  \Input{A matrix $\matr{A} \in \mathbb{R}^{m \times n}$, vectors $\bb \in \Rm$, $\bc \in \Rn$, a scalar $t \geqslant 0$, and a feasible initial point $\bx^{(0)} \in C$.}
  \Output{An optimal solution $\bxsol$ of~\eqref{eq:intro_main_problem} with dual variable $\bpsol$, or a direction $\bd^{s}$ certifying unboundedness ($\matr{A}\bd^{s} \leqslant \boldsymbol{0}$ and $\langle\bc,\bd^{s}\rangle < 0$).}
  \BlankLine
  \For{$k = 0, 1, 2, \dots$}{
    Compute $\eqset^{(k)} = \bigl\{i \in \{1,\dots,m\} : [\matr{A}\bx^{(k)}]_i = b_i\bigr\}$\;
    Compute $\bp^{(k)} \in \displaystyle\argmin_{\substack{p_i \geqslant 0,\ i \in \eqset^{(k)} \\[2pt] p_i = 0,\ i \notin \eqset^{(k)}}} \frac{1}{2}\normsq{\Atop\bp + t\bx^{(k)} + \bc}$ and set $\bd^{(k)} = -\Atop\bp^{(k)} - t\bx^{(k)} - \bc$\;
    Compute the stepsize $\Delta^{(k)} = \min_{i \notin \eqset^{(k)}} \frac{b_i - [\matr{A}\bx^{(k)}]_i}{\max(0,[\matr{A}\bd^{(k)}]_i)}$\;
    \BlankLine
    \uIf{$t > 0$ and $t\Delta^{(k)} \geqslant 1$}{
      Set $\bxsol = \bx^{(k)} + \dfrac{1}{t}\bd^{(k)}$ and $\bpsol = \bp^{(k)}$\;
      \Return $(\bxsol,\bpsol)$\tcp*{Optimal}
    }
    \uElseIf{$t = 0$ and $\Delta^{(k)} = +\infty$}{
      \lIf{$\normtwo{\bd^{(k)}} = 0$}{\Return $(\bx^{(k)},\bp^{(k)})$, an optimal pair}
      \lElse{\Return $\bd^{s} = \bd^{(k)}$, certifying unboundedness}
    }
    Update $\bx^{(k+1)} = \bx^{(k)} + \Delta^{(k)}\bd^{(k)}$\;
  }
\end{algorithm}

\begin{algorithm}[ht!]
  \caption{Minimal selection descent for problem~\eqref{eq:intro_main_problem} with SPSD matrix $\matr{Q}$.}
  \label{alg:quad_prog_2}
  \Input{A matrix $\matr{A} \in \mathbb{R}^{m \times n}$, an SPSD matrix $\matr{Q} \in \R^{n\times n}$, vectors $\bb \in \Rm$, $\bc \in \Rn$, a feasible initial point $\bx^{(0)} \in C$, and a tolerance $\varepsilon \geqslant 0$.}
  \Output{An approximate primal optimal solution $\bxsol$ of~\eqref{eq:intro_main_problem} with dual variable $\bpsol$, or a direction $\bd^{s}$ certifying unboundedness ($\matr{A}\bd^{s} \leqslant \boldsymbol{0}$, $\matr{Q}\bd^{s} = \boldsymbol{0}$, and $\langle\bc,\bd^{s}\rangle < 0$).}
  \BlankLine
  \For{$k = 0, 1, 2, \dots$}{
    Compute $\eqset^{(k)} = \bigl\{i \in \{1,\dots,m\} : [\matr{A}\bx^{(k)}]_i = b_i\bigr\}$\;
    Compute $\bd^{(k)} = -\bigl(\Atop\bp^{(k)} + \matr{Q}\bx^{(k)} + \bc\bigr)$, where $\bp^{(k)} \in \displaystyle\argmin_{\substack{p_i \geqslant 0,\ i \in \eqset^{(k)} \\[2pt] p_i = 0,\ i \notin \eqset^{(k)}}} \frac{1}{2}\normsq{\Atop\bp + \matr{Q}\bx^{(k)} + \bc}$\;
    \lIf{$\normtwo{\bd^{(k)}} \leqslant \varepsilon$}{\Return $(\bxsol,\bpsol) = (\bx^{(k)},\bp^{(k)})$\tcp*[h]{Converged}}
    Compute the stepsize $\Delta^{(k)} = \min\left(\min_{i \notin \eqset^{(k)}}\left\{
        \frac{b_i - [\matr{A}\bx^{(k)}]_i}{\max(0,[\matr{A}\bd^{(k)}]_i)}\right\},
	 \frac{\normsq{\bd^{(k)}}}{\langle\bd^{(k)},\matr{Q}\bd^{(k)}\rangle}\right)$\;
    \lIf{$\Delta^{(k)} = +\infty$}{\Return $\bd^{s} = \bd^{(k)}$\tcp*[h]{Unbounded}}
    Update $\bx^{(k+1)} = \bx^{(k)} + \Delta^{(k)}\bd^{(k)}$\;
  }
\end{algorithm}

\paragraph{Contributions of this paper:} \textbf{(i)} We compute the solution of the differential inclusion~\eqref{eq:intro_di} exactly. Each piece of the solution is an explicit function of time given by one nonnegative least-squares problem, and each breakpoint is the first zero of an explicit analytic function. \textbf{(ii)} We prove that the solution has finitely many pieces for every SPSD $\matr{Q}$, and its final piece yields either an optimal solution of problem~\eqref{eq:intro_main_problem} or a certificate of unboundedness. When $\matr{Q} = t\matr{I}_{n\times n}$ with $t \geqslant 0$, every piece lies on a straight line with a closed-form breakpoint and no active set occurs twice, so Algorithm~\ref{alg:quad_prog} terminates after finitely many iterations. \textbf{(iii)} For general SPSD $\matr{Q}$, each step of MSD preserves feasibility and solves one nonnegative least-squares problem over the active constraints. We prove that the iterates converge to an optimal solution when one exists and otherwise certify unboundedness or drive the objective to $-\infty$. When strict complementarity holds at the limit, the iterates also identify the optimal active set after finitely many steps and then converge at a linear rate. \textbf{(iv)} We apply our algorithms to a positivity-preserving limiter for a discontinuous Galerkin scheme and to optimal control problems constrained by the Poisson and heat equations. In the limiter, Algorithm~\ref{alg:quad_prog} conserves the total mass much more accurately than the splitting method of~\citet{liu2025optimization}. On the control problems, MSD reaches the optimal value to machine precision and is competitive in time with commercial and open-source solvers.

\paragraph{Related work:} Reviews of algorithms for linear and convex quadratic programs include~\cite{todd2002many,wright2025optimization,dinh2025real}. Closest to our approach are the steepest descent methods for linear programs, which move from one feasible point to another along steepest descent directions over the faces of the feasible polyhedron~\cite{brown1951computational,zoutendijk1960methods,rosen1960gradient,lemke1961constrained,chang1989steepest}. These methods either assumed that the data are nondegenerate or approximated the steepest descent direction. When $\matr{Q} = \matr{0}$, the pieces of the trajectory of~\eqref{eq:intro_di} are also rays along steepest descent directions. We compute these directions exactly, and our analysis extends to convex quadratic programs. Like primal active-set methods for quadratic programs~\cite[Section~16.5]{nocedal2006numerical}, our algorithms keep the iterates feasible and hold a set of constraints active on each step. Active-set methods, however, step toward the minimizer of the objective on the face of the current set, which requires solving a linear system involving $\matr{Q}$, and they change the set one constraint at a time. Our algorithms move along the minimal selection, compute it from one nonnegative least-squares problem problem in $\matr{A}$, and can add or remove several constraints at once.

By the minimal selection principle, the differential inclusion~\eqref{eq:intro_di} can be written equivalently as a projected dynamical system~\cite{dupuis1991lipschitz,dupuis1993dynamical}, whose velocity is the projection of $-\matr{Q}\bx - \bc$ onto the tangent cone of $C$ (Section~\ref{subsec:moreau}). It can also be written as a complementarity system, which couples a differential equation with complementarity conditions between the multipliers of the constraints and their slacks. The three formulations have the same unique solution~\cite{brogliato2006equivalence}. After a constant state is appended, the complementarity formulation is a linear complementarity system, a class that also models mechanical systems with unilateral constraints and electrical networks with ideal diodes~\cite{heemels2000linear}. \citet{heemels2000linear} construct the solutions of such systems piece by piece, but for~\eqref{eq:intro_di} their assumptions require all rows of $\matr{A}$ to be linearly independent, and they leave open whether the pieces can accumulate. Our analysis requires linear independence only for the active constraints and proves that the pieces are finitely many.

\paragraph{Organization of this paper:} We begin in Section~\ref{sec:setup} by detailing our setup and discussing the \emph{minimal selection principle} for differential inclusions, which casts the problem of computing an optimal solution to~\eqref{eq:intro_main_problem} (or certifying unboundedness) into that of computing the trajectory of an ordinary differential equation. In Section~\ref{sec:results_I}, we characterize the instantaneous and local dynamics of the solution to the differential inclusions associated with problem~\eqref{eq:intro_main_problem}. In Section~\ref{sec:results_II}, we analyze the global solution to the differential inclusions and investigate its asymptotic limit, showing in particular that the global solution and its asymptotic limit can be computed from finitely many pieces and nonnegative least-squares problems. Section~\ref{sec:computations_global} analyzes the explicit Euler discretization of the differential inclusion associated with problem~\eqref{eq:intro_main_problem} for a general SPSD $\matr{Q}$, and shows that the iterates converge to an optimal solution or certify unboundedness, and that they identify the optimal active set and converge at a linear rate when the multipliers at the limit are positive. We present in Section~\ref{sec:numerics} numerical experiments to test our algorithms on challenging scientific computing problems, including preserving positivity for discontinuous Galerkin methods and PDE-constrained optimal control problems.

\section{Preliminaries}
\label{sec:setup}
We review here the first-order optimality conditions of problem~\eqref{eq:intro_main_problem} and present our differential inclusions approach. Table~\ref{tab:notation} below summarizes some notation and definitions. We refer the reader to Appendix~\ref{app:A} for more mathematical background.
\begin{table}[ht]
    \caption{Some notation and definitions used in this work.~\label{tab:notation}}
    \begin{center}
    \begin{tabular}{l|l}
    \hline
        Notation & Meaning \\
        \hline
        $[m]$ & The indices $i \in \{1,2,\dots,m\}$; \\
        $\bp_{\eqset}$ & The $|\eqset|$-dimensional subvector of $\bp \in \Rm$ indexed by $\eqset \subset [m]$; \\
        $\left\langle \bx_{1},\bx_{2} \right\rangle$ & The Euclidean scalar product of two vectors $\bx_{1},\bx_{2} \in \Rn$;\\ 
        $\interior{~C}$ & Interior of a nonempty subset $C$; \\
        $\ri~C$ & Relative interior of a nonempty subset $C$; \\
        $\chi_{C}$ & The characteristic function of a set $C$: \\
        & \quad $\chi_{C}(\bx) \coloneqq \begin{cases} 0 \; &\text{if}\;\bx \in C, \\ +\infty \; &\text{otherwise}.\end{cases}$ \\
        $\dom f$ & The domain of a function $f$: $\dom f \coloneqq \left\{ \bx \in \Rn : f(\bx) < +\infty\right\}$; \\
        $N_{C}(\bx)$ & The normal cone of a closed convex set $C$ at $\bx \in C$: \\
        & $\quad N_{C}(\bx) \coloneqq \left\{\bs \in \Rn : \langle\bs,\by-\bx\rangle \leqslant 0 \,\text{for every}\, \by \in C \right\}$;\\
        $T_{C}(\bx)$ & The tangent cone of a closed convex set $C$ at $\bx \in C$: \\
        & $\quad T_{C}(\bx) \coloneqq  \left\{\bd \in \Rn : \langle\bs,\bd\rangle\leqslant 0 \ \text{for every $\bs \in N_{C}(\bx)$}\right\}$; \\
        $\proj_{C}(\bx)$ & Projection of $\bx \in \Rn$ onto a closed convex set $C \subset \Rn$: \\
        & \quad $\proj_{C}(\bx) \coloneqq \argmin_{\by\in C}\left\Vert \bx-\by\right\Vert _{2}^{2}$; \\
        \hline
    \end{tabular}   
    \end{center} 
\end{table}

\subsection{Setup}
\label{subsec:char_di}
We assume throughout that problem~\eqref{eq:intro_main_problem} is feasible, i.e., there is $\bx \in \Rn$ such that $\matr{A}\bx \leqslant \bb$. For convenience, we use the convention $0/0 = 0$ and, for $a>0$, $a/0 = +\infty$. Let $C = \left\{\bx \in \Rn : \matr{A}\bx \leqslant \bb \right\}$, let $\chi_{C}$ denote the characteristic function of $C$, let $f(\bx) = \tfrac{1}{2}\langle\bx,\matr{Q}\bx\rangle + \langle\bc,\bx\rangle$, and let $V(\bx) = f(\bx) + \chi_{C}(\bx)$ denote the objective function in~\eqref{eq:intro_main_problem}.

The subdifferential of $V$ at a point $\bxo \in C$ is described in terms of the normal cone $N_{C}(\bxo)$ (Table~\ref{tab:notation}). Subdifferential calculus \cite[Exercise~10.22(a),(b)]{rockafellar2009variational} gives
\[
\partial_{\bx} V(\bxo) = \nabla_{\bx}f(\bxo) + N_{C}(\bxo).
\]
Since $C$ is a polyhedron, the rows of $\matr{A}$ corresponding to the constraints active at $\bxo$ generate $N_{C}(\bxo)$~\cite[Example 5.2.6b]{hiriart2013convexI}. We define correspondingly the \emph{equicorrelation set} of $\bxo$ as
\[
\exo \coloneqq \left\{i \in [m] : [\matr{A}\bxo]_{i} = b_{i}\right\}.
\]
Hence, explicitly,
\begin{equation}\label{eq:subdiff_V}
\partial_{\bx}V(\bxo) = \left\{\nabla_{\bx}f(\bxo) + \Atop\bp : \bp_{\eqset(\bxo)} \geqslant \boldsymbol{0}, \ \bp_{\eqcset(\bxo)} = \boldsymbol{0} \right\}.
\end{equation}

A point $\bxsol \in C$ is an optimal solution of problem~\eqref{eq:intro_main_problem} if and only if $\boldsymbol{0} \in \partial_{\bx}V(\bxsol)$~\cite[Theorem 10.1]{rockafellar2009variational}. Equivalently, there is $\bpsol \in \Rm$ with $\bp^{s}_{\eqset(\bxsol)} \geqslant \boldsymbol{0}$ and $\bp^{s}_{\eqcset(\bxsol)} = \boldsymbol{0}$ such that
\[
\nabla_{\bx}f(\bxsol) + \Atop\bpsol = \boldsymbol{0}.
\]
If $\matr{Q}$ is SPD, then $V$ is strongly convex and problem~\eqref{eq:intro_main_problem} has a unique optimal solution. If $\matr{Q}$ is merely SPSD, problem~\eqref{eq:intro_main_problem} may have no optimal solution, or several.

\subsection{Differential inclusions, the minimal selection principle, and the slow system}
\label{subsec:msp}
To certify if problem~\eqref{eq:intro_main_problem} has an optimal solution, and output one when it does, we propose to compute the trajectory of the differential inclusion
\begin{equation}
\dot{\bx}(\tau) \in -\partial_{\bx} V(\bx(\tau)), \quad \bx(0) = \bxo \in C.
\label{eq:main_grad_incl}
\end{equation}
A differential inclusion replaces the right-hand side of an ordinary differential equation with an appropriate set-valued mapping. When the mapping is the subdifferential of a proper, lower semicontinuous, convex function (as in~\eqref{eq:main_grad_incl}), the inclusion is called a \emph{gradient inclusion}. The next theorem states that the solution of~\eqref{eq:main_grad_incl} exists, is unique, and converges to a minimizer of $V$ when one exists. The theorem also states the \emph{minimal selection principle}, namely that at every $\tau \geqslant 0$, the solution evolves along the element of minimal norm in $-\partial_{\bx} V(\bx(\tau))$.

\begin{thm}\label{thm:MSP}
The differential inclusion~\eqref{eq:main_grad_incl} has a unique global solution $\bx \colon [0,+\infty) \to C$ that is globally Lipschitz continuous and satisfies the following properties:
\begin{itemize}
    \item[(i)] The right derivative $\tau \mapsto \dot{\bx}(\tau)$ is right-continuous and $\tau \mapsto \normtwo{\dot{\bx}(\tau)}$ is nonincreasing.
    
    \item[(ii)] If $\bx \mapsto V(\bx)$ achieves its minimum at some point, then $\tau \mapsto \bx(\tau)$ converges to such a point:
	\[
	\lim_{\tau \to +\infty} V(\bx(\tau)) = \min_{\bx \in \Rn} V(\bx) \ \text{and} \ \lim_{\tau \to +\infty}\bx(\tau) \in \argmin_{\bx \in \Rn} V(\bx).
	\]

	\item[(iii)] (The minimal selection principle.) The function $\tau \mapsto \bx(\tau)$ solves the initial value problem
	\begin{equation}\label{eq:slow_ode}
		\dot{\bx}(\tau) = -\proj_{\partial_{\bx}V(\bx(\tau))}(\boldsymbol{0}), \quad \bx(0) = \bxo \in C
 	\end{equation}
 	at every $\tau \geqslant 0$, the derivative on the left being the right derivative of part~(i).
\end{itemize}
\end{thm}
\begin{proof}
See~\cite[Theorem 1 (p.~147); Theorem 1 (p.~159)]{aubin2012differential} and~\cite[Theorem~3.1, p.~54]{brezis1973ope}.
\end{proof}

\subsection{Geometric interpretation of the minimal selection principle}
\label{subsec:moreau}
Following~\cite{aubin2012differential}, we call~\eqref{eq:slow_ode} the \emph{slow system} and its solution the \emph{slow solution}. At any $\bx \in C$, the normal cone $N_{C}(\bx)$ and the tangent cone $T_{C}(\bx)$ (Table~\ref{tab:notation}) are mutually polar closed convex cones, so by Moreau's decomposition theorem~\cite{moreau1962decomposition}, every $\bz \in \Rn$ decomposes as
 \begin{equation}\label{eq:moreau_decomp}
 \bz = \proj_{N_{C}(\bx)}(\bz) + \proj_{T_{C}(\bx)}(\bz), \quad \langle \proj_{N_{C}(\bx)}(\bz), \proj_{T_{C}(\bx)}(\bz)\rangle = 0.
 \end{equation}
 Since $\partial_{\bx}V(\bx) = \nabla_{\bx}f(\bx) + N_{C}(\bx)$, we have $\proj_{\partial_{\bx}V(\bx)}(\boldsymbol{0}) = \nabla_{\bx}f(\bx) + \proj_{N_{C}(\bx)}(-\nabla_{\bx}f(\bx))$, and applying~\eqref{eq:moreau_decomp} with $\bz = -\nabla_{\bx}f(\bx)$ gives $-\proj_{\partial_{\bx}V(\bx)}(\boldsymbol{0}) = \proj_{T_{C}(\bx)}(-\nabla_{\bx}f(\bx))$. The slow system~\eqref{eq:slow_ode} then becomes the \emph{projected dynamical system}~\cite{dupuis1991lipschitz}
 \begin{equation*}
 \dot{\bx}(\tau) = \bd(\bx(\tau)), \quad \bx(0) = \bxo \in C, \quad \text{where} \quad \bd(\bx) \coloneqq \proj_{T_{C}(\bx)}(-\nabla_{\bx}f(\bx)).
 \end{equation*}
 We call $\bd(\bx)$ the \emph{direction} at $\bx \in C$. The slow solution therefore evolves along the projection of the negative gradient $-\nabla_{\bx}f$ onto the tangent cone of $C$.

\subsection{Computing the direction via nonnegative least squares}
\label{subsec:computing_nnls}
Integrating the slow system~\eqref{eq:slow_ode} requires evaluating the direction $\bd(\bx(\tau))$ for all $\tau \geqslant 0$. At a single point $\bxo \in C$, we evaluate $\bd(\bxo)$ via nonnegative least squares (NNLS). Write $\bn(\bxo) \coloneqq \proj_{N_{C}(\bxo)}(-\nabla_{\bx}f(\bxo))$. By formula~\eqref{eq:subdiff_V}, we have $\bn(\bxo) = \Atop \bp(\bxo)$, where
\begin{equation}\label{eq:nnls_problem}
\bp(\bxo) \in \argmin_{\bp \in \Rm} \left\{\frac{1}{2}\normsq{\Atop\bp + \nabla_{\bx}f(\bxo)} : \bp_{\exo}\geqslant \boldsymbol{0},\, \bp_{\ecxo} = \boldsymbol{0}\right\}.
\end{equation}
Moreau's decomposition theorem~\eqref{eq:moreau_decomp} then gives
\begin{equation}\label{eq:d_from_p}
\bd(\bxo) = -\Atop\bp(\bxo) - \nabla_{\bx}f(\bxo).
\end{equation}
The vector $\bp(\bxo)$ may not always be unique, but the projections $\bd(\bxo)$ and $\bn(\bxo)$ are.

The optimal NNLS solution $\bp(\bxo)$ admits an explicit form in terms of the \emph{active set}
\begin{equation}
    J(\bp(\bxo)) \coloneqq \left\{i \in [m] : [\bp(\bxo)]_{i} > 0\right\}.
\label{eq:active_set}
\end{equation}
We write $J(\bxo) \equiv J(\bp(\bxo))$ when $\bp(\bxo)$ is unambiguous. By Carath\'eodory's theorem~\cite[Corollary~17.1.2]{rockafellar1997convex}, we can choose $\bp(\bxo)$ so that
\begin{equation}\label{eq:LI}\tag{LI}
\text{the rows of $\matr{A}_{J(\bxo)}$ are linearly independent.}
\end{equation}
Since $\bp_{J(\bxo)}(\bxo)$ is strictly positive, it is an unconstrained minimizer of $\bq \mapsto \tfrac{1}{2}\normsq{\matr{A}_{J(\bxo)}^{\top}\bq + \nabla_{\bx}f(\bxo)}$ by convexity. The first-order optimality conditions give
\begin{equation}\label{eq:nnls_explicit}
\bp_{J(\bxo)}(\bxo) = - \left(\matr{A}_{J(\bxo)}\matr{A}_{J(\bxo)}^{\top}\right)^{-1}\matr{A}_{J(\bxo)}\nabla_{\bx}f(\bxo).
\end{equation}
Let
\begin{equation}\label{eq:nnls_pperp}
\matr{P}^{\perp}_{J(\bxo)} \coloneqq \matr{I}_{n \times n} - \matr{A}_{J(\bxo)}^{\top}\left(\matr{A}_{J(\bxo)}\matr{A}^{\top}_{J(\bxo)}\right)^{-1}\matr{A}_{J(\bxo)}
\end{equation}
denote the orthogonal projection onto $\ker(\matr{A}_{J(\bxo)})$. Since $\bp(\bxo)$ vanishes outside $J(\bxo)$, substituting~\eqref{eq:nnls_explicit} into~\eqref{eq:d_from_p} gives
\begin{equation}\label{eq:steepest_dd}
\bd(\bxo) = -\matr{P}^{\perp}_{J(\bxo)}\nabla_{\bx}f(\bxo).
\end{equation}
When $J(\bxo) = \emptyset$, we read~\eqref{eq:nnls_explicit} as $\bp(\bxo) = \boldsymbol{0}$ and~\eqref{eq:nnls_pperp} as $\matr{P}^{\perp}_{J(\bxo)} = \matr{I}_{n \times n}$.

\begin{rem}\label{rem:computing_nnls}
Specialized algorithms for solving NNLS problems include active-set algorithms terminating in finite time, such as the classical Lawson--Hanson algorithm~\cite{lawson1995solving} and its generalization, Meyer's algorithm~\cite{meyer2013simple}. See also~\cite{dimiccoli2016fundamentals} for a review of other cone projection algorithms.
\end{rem}

\subsection{Karush--Kuhn--Tucker and strict complementarity conditions}
\label{subsec:properties_nnls}
The solution $\bp(\bxo)$ of the NNLS problem~\eqref{eq:nnls_problem} and the direction $\bd(\bxo)$ satisfy the Karush--Kuhn--Tucker (KKT) conditions of that problem and two of their consequences:
\begin{subequations}\label{eq:nnls_oc}
\begin{align}
&[\matr{A}\bd(\bxo)]_{\exo} \leqslant \boldsymbol{0}, \label{eq:nnls_oc_1} \\
&[\bp(\bxo)]_{i}[\matr{A}\bd(\bxo)]_{i} = 0 \ \text{for $i=1,\dots,m$}, \label{eq:nnls_oc_2} \\
&[\matr{A}\bd(\bxo)]_{i} = 0 \ \text{for every $i \in J(\bxo)$}, \label{eq:nnls_oc_2_5} \\
&\normsq{\bd(\bxo)} + \langle \nabla_{\bx}f(\bxo), \bd(\bxo)\rangle = 0. \label{eq:nnls_oc_3}
\end{align}
\end{subequations}
Conditions~\eqref{eq:nnls_oc_1}--\eqref{eq:nnls_oc_3} all follow from Moreau's decomposition~\eqref{eq:moreau_decomp}. Indeed, condition~\eqref{eq:nnls_oc_1} states that $\bd(\bxo) \in T_{C}(\bxo)$. For condition~\eqref{eq:nnls_oc_2}, the orthogonality in~\eqref{eq:moreau_decomp} gives
\[
\langle\matr{A}\bd(\bxo),\bp(\bxo)\rangle = \langle\bd(\bxo),\Atop\bp(\bxo)\rangle = \langle\bd(\bxo),\bn(\bxo)\rangle = 0,
\]
and every term of $\langle\matr{A}\bd(\bxo),\bp(\bxo)\rangle$ vanishes because $\bp_{\exo}(\bxo) \geqslant \boldsymbol{0}$, $\bp_{\ecxo}(\bxo) = \boldsymbol{0}$, and~\eqref{eq:nnls_oc_1} holds. Condition~\eqref{eq:nnls_oc_2_5} follows from~\eqref{eq:nnls_oc_2} because $[\bp(\bxo)]_{i} > 0$ for every $i \in J(\bxo)$. Condition~\eqref{eq:nnls_oc_3} follows from~\eqref{eq:d_from_p}, which gives $\langle \nabla_{\bx}f(\bxo),\bd(\bxo)\rangle = \langle -\bd(\bxo) - \bn(\bxo),\bd(\bxo)\rangle = -\normsq{\bd(\bxo)}$. Problem~\eqref{eq:nnls_problem} is convex with linear constraints, so conditions~\eqref{eq:nnls_oc_1}--\eqref{eq:nnls_oc_2} are also sufficient. Every feasible $\bp$ whose residual $-\Atop\bp - \nabla_{\bx}f(\bxo)$ satisfies~\eqref{eq:nnls_oc_1}--\eqref{eq:nnls_oc_2} in place of $\bd(\bxo)$ solves~\eqref{eq:nnls_problem}.

\begin{rem}\label{rem:ssc}
The KKT complementarity condition~\eqref{eq:nnls_oc_2} over the equicorrelation set $\exo$ need not be strict. To ease the exposition, the results of Sections~\ref{subsec:instant_dynamics}--\ref{subsec:computations_local}, except Proposition~\ref{prop:descent_direction}, assume the \emph{strict complementarity condition}
\begin{equation}\label{eq:SCC}\tag{SCC}
	[\matr{A}\bd(\bxo)]_{i} < 0 \ \text{for every $i \in \exo \setminus J(\bxo)$}.
\end{equation}
Section~\ref{subsec:removing_scc_local} extends these results to the case
  where~\eqref{eq:SCC} fails.
\end{rem}

\section{Local solution of the slow system}
\label{sec:results_I}
In this section, we compute the local solution $\tau \mapsto \bx(\tau)$ of the slow system~\eqref{eq:slow_ode} started at $\bxo \in C$ and show that $\tau \mapsto \bd(\bx(\tau))$ may jump at a breakpoint time $\tau_{*}(\bxo) > 0$. Section~\ref{subsec:instant_dynamics} analyzes the instantaneous dynamics, Section~\ref{subsec:local_sol} integrates the system on $[0,\tau_{*}(\bxo))$, and Section~\ref{subsec:continuation} extends the solution to $\tau_{*}(\bxo)$ whenever it is finite. Section~\ref{subsec:computations_local} computes the breakpoint time in practice, and Section~\ref{subsec:removing_scc_local} extends these results to the case where~\eqref{eq:SCC} fails.

\subsection{Instantaneous dynamics of the slow system}
\label{subsec:instant_dynamics}
\subsubsection{Descent properties}
\label{subsubsec:descent_prop}
How does the slow system~\eqref{eq:slow_ode} evolve instantaneously? We first show that the objective function $V$ decreases along the ray $\Delta \mapsto \bxo + \Delta\bd(\bxo)$ for small $\Delta > 0$. The decrease requires that $\bxo + \Delta\bd(\bxo)$ stay feasible, that is, $\matr{A}(\bxo + \Delta\bd(\bxo)) \leqslant \bb$. Since $[\matr{A}\bd(\bxo)]_{\eqset(\bxo)} \leqslant \boldsymbol{0}$ by~\eqref{eq:nnls_oc_1} and $[\matr{A}\bxo]_{\eqcset(\bxo)} < \bb_{\eqcset(\bxo)}$ by definition, the largest $\Delta$ for which this inequality holds is
\begin{equation}\label{eq:delta_feas}
\Delta_{\mathrm{feas}}(\bxo) \coloneqq \min_{i \in \eqcset(\bxo)} \frac{b_i-[\matr{A}\bxo]_{i}}{\max(0,[\matr{A}\bd(\bxo)]_{i})},
\end{equation}
which equals $+\infty$ if $\eqcset(\bxo) = \emptyset$ or $[\matr{A}\bd(\bxo)]_{\eqcset(\bxo)} \leqslant \boldsymbol{0}$. In any case, $\bxo + \Delta\bd(\bxo) \in C$ for every $\Delta \in [0,\Delta_{\mathrm{feas}}(\bxo)]$, and condition~\eqref{eq:nnls_oc_3} gives
\begin{equation}\label{eq:obj_fun_after_descent}
  V(\bxo + \Delta\bd(\bxo)) 
  = V(\bxo) - \Delta\normsq{\bd(\bxo)} + \frac{\Delta^2}{2}\langle\bd(\bxo),\matr{Q}\bd(\bxo)\rangle.
\end{equation}
Formula~\eqref{eq:obj_fun_after_descent} shows that $\bd(\bxo)$ is a descent direction.

\begin{prop}\label{prop:descent_direction}
Let $\bxo \in C$. If $\bd(\bxo) \neq \boldsymbol{0}$, then $V(\bxo + \Delta\bd(\bxo)) < V(\bxo)$ for every $\Delta \in (0,\Delta_{+}(\bxo))$, where
\begin{equation}\label{eq:delta_plus}
\Delta_{+}(\bxo) \coloneqq \min\left(\Delta_{\mathrm{feas}}(\bxo),\, \frac{2\normsq{\bd(\bxo)}}{\langle\bd(\bxo),\matr{Q}\bd(\bxo)\rangle}\right).
\end{equation}
\end{prop}
\begin{proof}
Suppose $\bd(\bxo) \neq \boldsymbol{0}$. If $\langle\bd(\bxo),\matr{Q}\bd(\bxo)\rangle = 0$, then the right-hand side of~\eqref{eq:obj_fun_after_descent} decreases with $\Delta$ until $\Delta = \Delta_{\mathrm{feas}}(\bxo)$. If $\langle\bd(\bxo),\matr{Q}\bd(\bxo)\rangle > 0$, the difference $V(\bxo + \Delta\bd(\bxo)) - V(\bxo)$ stays negative for $\Delta$ between $0$ and the smaller of $\Delta_{\mathrm{feas}}(\bxo)$ and $2\normsq{\bd(\bxo)}/\langle\bd(\bxo),\matr{Q}\bd(\bxo)\rangle$. In either case, $V(\bxo + \Delta \bd(\bxo)) < V(\bxo)$ whenever $\Delta \in (0,\Delta_{+}(\bxo))$.
\end{proof}

\subsubsection{The evolution rule}
\label{subsubsec:evolution_rule}
Proposition~\ref{prop:descent_direction} shows that $V$ decreases along the ray $\Delta \mapsto \bxo + \Delta\bd(\bxo)$. The ray stays in $C$ for every $\Delta \in [0,\Delta_{\mathrm{feas}}(\bxo)]$, so the direction $\bd(\bxo + \Delta\bd(\bxo))$ is well-defined there. We prove next that this direction is a linear function of $\bd(\bxo)$ for small $\Delta > 0$, a result we call the \emph{evolution rule}. We first show that the active set does not change along the ray, that is, $J(\bxo + \Delta\bd(\bxo)) = J(\bxo)$, as long as $\Delta$ stays below $\Delta_{\mathrm{feas}}(\bxo)$ and
\begin{equation}\label{eq:delta_J}
\Delta_{J}(\bxo) \coloneqq \min_{i \in J(\bxo)} \ [\bp(\bxo)]_{i}/\max(0,[(\matr{A}_{J(\bxo)}\matr{A}_{J(\bxo)}^{\top})^{-1}\matr{A}_{J(\bxo)}\matr{Q}\bd(\bxo)]_{i}).
\end{equation}
We set $\Delta_{J}(\bxo) = +\infty$ when $J(\bxo) = \emptyset$.

\begin{lem}\label{lem:active_set}
Let $\bxo \in C$, choose $\bp(\bxo)$ so that~\eqref{eq:LI} holds, and assume that the strict complementarity condition~\eqref{eq:SCC} holds. Then $J(\bxo + \Delta\bd(\bxo)) = J(\bxo)$ for every $\Delta \in [0,\Delta_{*}(\bxo))$, where $\Delta_{*}(\bxo) \coloneqq \min(\Delta_{\mathrm{feas}}(\bxo),\Delta_{J}(\bxo))$.
\end{lem}

\begin{proof}
Let $\Delta \in (0,\Delta_{*}(\bxo))$. Since $\Delta < \Delta_{\mathrm{feas}}(\bxo)$, the point $\bxo + \Delta\bd(\bxo)$ lies in $C$ and the direction $\bd(\bxo + \Delta\bd(\bxo))$ is well-defined. We first show that $\eqset(\bxo + \Delta\bd(\bxo)) = J(\bxo)$. On the one hand, $J(\bxo) \subset \eqset(\bxo)$ because $\bp_{\eqcset(\bxo)}(\bxo) = \boldsymbol{0}$, so condition~\eqref{eq:nnls_oc_2_5} gives
\[
[\matr{A}(\bxo + \Delta\bd(\bxo))]_{J(\bxo)} = [\matr{A}\bxo]_{J(\bxo)} + \Delta[\matr{A}\bd(\bxo)]_{J(\bxo)} = \bb_{J(\bxo)},
\]
that is, $J(\bxo) \subset \eqset(\bxo + \Delta\bd(\bxo))$. On the other hand,~\eqref{eq:SCC} implies
\[
[\matr{A}(\bxo + \Delta\bd(\bxo))]_{\eqset(\bxo) \setminus J(\bxo)} = [\matr{A}\bxo]_{\eqset(\bxo) \setminus J(\bxo)} + \Delta[\matr{A}\bd(\bxo)]_{\eqset(\bxo) \setminus J(\bxo)} < \bb_{\eqset(\bxo)\setminus J(\bxo)},
\]
while every $i \in \eqcset(\bxo)$ satisfies $[\matr{A}(\bxo + \Delta\bd(\bxo))]_{i} < b_{i}$ by the definition~\eqref{eq:delta_feas} of $\Delta_{\mathrm{feas}}(\bxo)$, that is, $\eqset(\bxo + \Delta\bd(\bxo)) \subset J(\bxo)$. Combining, we find $\eqset(\bxo + \Delta\bd(\bxo)) = J(\bxo)$.

Now consider the NNLS problem~\eqref{eq:nnls_problem} at $\bxo + \Delta\bd(\bxo)$,
\[
\min_{\substack{\bp_{\eqset(\bxo + \Delta\bd(\bxo))} \geqslant \boldsymbol{0} \\ \bp_{\eqcset(\bxo + \Delta\bd(\bxo))} = \boldsymbol{0}}} \frac{1}{2}\normsq{\Atop\bp + \matr{Q}(\bxo + \Delta\bd(\bxo)) + \bc},
\]
whose solutions give the direction $\bd(\bxo + \Delta\bd(\bxo))$ via formula~\eqref{eq:d_from_p}. Let $\bp^{\mathrm{LS}}(\bxo + \Delta\bd(\bxo))$ minimize the same objective over the vectors supported on $J(\bxo)$ without the sign constraints. This minimizer is unique because the rows of $\matr{A}_{J(\bxo)}$ are linearly independent. Moreover,
\begin{equation*}
\begin{alignedat}{1}
\bp^{\mathrm{LS}}_{J(\bxo)}(\bxo + \Delta\bd(\bxo)) &= -(\matr{A}_{J(\bxo)}\matr{A}_{J(\bxo)}^{\top})^{-1}\matr{A}_{J(\bxo)}\bigl(\matr{Q}(\bxo + \Delta\bd(\bxo)) + \bc\bigr) \\
&= -(\matr{A}_{J(\bxo)}\matr{A}_{J(\bxo)}^{\top})^{-1}\matr{A}_{J(\bxo)}(\matr{Q}\bxo + \bc) - \Delta\,(\matr{A}_{J(\bxo)}\matr{A}_{J(\bxo)}^{\top})^{-1}\matr{A}_{J(\bxo)}\matr{Q}\bd(\bxo) \\
&= \bp_{J(\bxo)}(\bxo) - \Delta\,(\matr{A}_{J(\bxo)}\matr{A}_{J(\bxo)}^{\top})^{-1}\matr{A}_{J(\bxo)}\matr{Q}\bd(\bxo),
\end{alignedat}
\end{equation*}
where we used formula~\eqref{eq:nnls_explicit} to obtain the last equality. Each entry of $\bp^{\mathrm{LS}}_{J(\bxo)}(\bxo + \Delta\bd(\bxo))$ is affine in $\Delta$ and positive at $\Delta = 0$, and it stays positive for every $\Delta < \Delta_{J}(\bxo)$ by~\eqref{eq:delta_J}. Since $\eqset(\bxo + \Delta\bd(\bxo)) = J(\bxo)$, the vector $\bp^{\mathrm{LS}}(\bxo + \Delta\bd(\bxo))$ is therefore feasible for the NNLS problem at $\bxo + \Delta\bd(\bxo)$. It minimizes the objective without the sign constraints, so it solves that problem, and it is the only solution because the rows of $\matr{A}_{J(\bxo)}$ are linearly independent. We conclude that $J(\bxo + \Delta\bd(\bxo)) = J(\bxo)$ for every $\Delta \in (0,\Delta_{*}(\bxo))$, and trivially at $\Delta = 0$.
\end{proof}

By Lemma~\ref{lem:active_set}, the active set stays equal to $J(\bxo)$ for every $\Delta \in (0,\Delta_{*}(\bxo))$, so the point $\bxo + \Delta\bd(\bxo)$ stays in the relative interior of the face of $C$ identified by $J(\bxo)$. Formula~\eqref{eq:steepest_dd} therefore holds along the ray with the same projection $\matr{P}^{\perp}_{J(\bxo)}$, and the evolution rule follows.

\begin{prop}[Evolution rule]\label{prop:evolution_rule}
Under the hypotheses of Lemma~\ref{lem:active_set}, we have
\begin{equation}\label{eq:evolution_rule}
\bd(\bxo + \Delta\bd(\bxo)) = \left(\matr{I}_{n\times n} - \Delta\matr{P}_{J(\bxo)}^{\perp}\matr{Q}\right)\bd(\bxo) \quad \text{for every $\Delta \in [0,\Delta_{*}(\bxo))$}.
\end{equation}
\end{prop}
\begin{proof}
By Lemma~\ref{lem:active_set}, formula~\eqref{eq:steepest_dd} applied at $\bxo + \Delta\bd(\bxo)$ yields
\[
\begin{alignedat}{1}
\bd(\bxo + \Delta\bd(\bxo)) &= -\matr{P}^{\perp}_{J(\bxo + \Delta\bd(\bxo))}(\matr{Q}(\bxo + \Delta\bd(\bxo)) + \bc) \\
&= -\matr{P}^{\perp}_{J(\bxo)}(\matr{Q}(\bxo + \Delta\bd(\bxo)) + \bc) \\
&= -\matr{P}^{\perp}_{J(\bxo)}(\matr{Q}\bxo + \bc) - \Delta\matr{P}^{\perp}_{J(\bxo)}\matr{Q}\bd(\bxo) \\
&= \bd(\bxo) - \Delta\matr{P}^{\perp}_{J(\bxo)}\matr{Q}\bd(\bxo).
\end{alignedat}
\]
Rearranging yields~\eqref{eq:evolution_rule}.
\end{proof}

\subsection{Local dynamics of the slow system}
\label{subsec:local_sol}
Proposition~\ref{prop:evolution_rule} describes how the direction $\bd(\bxo + \Delta\bd(\bxo))$ changes with $\Delta$. In particular, taking the limit $\Delta \to 0^{+}$ in~\eqref{eq:evolution_rule} yields
\[
\lim_{\Delta \to 0^{+}} \frac{\bd(\bxo + \Delta\bd(\bxo)) - \bd(\bxo)}{\Delta} = -\matr{P}^{\perp}_{J(\bxo)}\matr{Q}\bd(\bxo).
\]
This limit suggests that $\bd(\bx(\tau)) = e^{-\matr{P}^{\perp}_{J(\bxo)}\matr{Q}\tau}\bd(\bxo)$ for small $\tau > 0$. We now verify this formula under the strict complementarity condition~\eqref{eq:SCC}. Since $\dot{\bx}(\tau) = \bd(\bx(\tau))$, integrating this formula gives the candidate solution $\bxo + \matr{M}(\tau)\bd(\bxo)$, where
\begin{equation}\label{eq:int_mexp}
\matr{M}(\tau) \coloneqq \int_{0}^{\tau} e^{-\matr{P}^{\perp}_{J(\bxo)}\matr{Q}s}\diff s.
\end{equation}
The constraints in $J(\bxo)$ stay active along the candidate solution. By the identity $\matr{A}_{J(\bxo)}\matr{P}^{\perp}_{J(\bxo)} = \matr{0}$, every term of order $k \geqslant 1$ in the series of $\matr{A}_{J(\bxo)}e^{-\matr{P}^{\perp}_{J(\bxo)}\matr{Q}s}$ vanishes, so condition~\eqref{eq:nnls_oc_2_5} gives
\begin{equation}\label{eq:active_rows_invariant}
\matr{A}_{J(\bxo)}\matr{M}(\tau)\bd(\bxo) = \tau\matr{A}_{J(\bxo)}\bd(\bxo) = \boldsymbol{0} \quad \text{for every $\tau \geqslant 0$}.
\end{equation}
Formula~\eqref{eq:nnls_explicit}, evaluated at the candidate solution with the active set $J(\bxo)$, gives the multipliers
\begin{equation}\label{eq:w_mexp}
\bw(\tau) \coloneqq \bp_{J(\bxo)}(\bxo) - \left(\matr{A}_{J(\bxo)}\matr{A}_{J(\bxo)}^{\top}\right)^{-1}\matr{A}_{J(\bxo)}\matr{Q}\matr{M}(\tau)\bd(\bxo).
\end{equation}
The candidate solution stays in $C$, with equicorrelation set $J(\bxo)$ and nonnegative multipliers $\bw(\tau)$, until a constraint outside $J(\bxo)$ becomes active or a multiplier $w_{i}(\tau)$ turns negative. We therefore define the \emph{feasibility time}, \emph{active-set time}, and \emph{breakpoint time}
\begin{subequations}
\begin{align}
\tau_{\mathrm{feas}}(\bxo) &\coloneqq \inf\Bigl\{\tau > 0 : \bigl[\matr{A}\bigl(\bxo + \matr{M}(\tau)\bd(\bxo)\bigr)\bigr]_{i} = b_{i} \text{ for some } i \in \Jcbxo \Bigr\}, \label{eq:tau_feas} \\
\tau_{J}(\bxo) &\coloneqq \inf\Bigl\{\tau > 0 : w_{i}(\tau) < 0 \text{ for some } i \in J(\bxo)\Bigr\}, \label{eq:tau_J} \\
\tau_{*}(\bxo) &\coloneqq \min(\tau_{\mathrm{feas}}(\bxo), \tau_{J}(\bxo)). \label{eq:tau_star}
\end{align}
\end{subequations}
The next theorem shows that $\tau_{*}(\bxo) > 0$ and that, before this time, the candidate solution is the slow solution, with the constraints outside $J(\bxo)$ inactive and the multipliers $\bw(\tau)$ nonnegative.

\begin{thm}\label{thm:local_solution}
Let $\bxo \in C$, choose $\bp(\bxo)$ so that~\eqref{eq:LI} holds, and assume that~\eqref{eq:SCC} holds. Then $\tau_{*}(\bxo) > 0$ and the slow solution solves the initial value problem
\begin{equation}\label{eq:local_simple_ode}
	\dot{\bx}(\tau) = e^{-\matr{P}_{J(\bxo)}^{\perp}\matr{Q}\tau}\bd(\bxo), \quad \bx(0) = \bxo
\end{equation}
for $\tau \in [0,\tau_{*}(\bxo))$. In particular,
\begin{equation}\label{eq:local_solution}
	\bx(\tau) = \bxo + \matr{M}(\tau)\bd(\bxo) \quad \text{for $\tau \in [0,\tau_{*}(\bxo))$.}
\end{equation}
Moreover, for every $\tau \in (0,\tau_{*}(\bxo))$, we have $\eqset(\bx(\tau)) = J(\bxo)$ and the Moreau decomposition
\begin{equation}\label{eq:local_multipliers}
-\matr{Q}\bx(\tau) - \bc = \bd(\bx(\tau)) + \matr{A}_{J(\bxo)}^{\top}\bw(\tau), \quad \text{with } \bw(\tau) \geqslant \boldsymbol{0},
\end{equation}
where $\bd(\bx(\tau)) \in T_{C}(\bx(\tau))$ and $\matr{A}_{J(\bxo)}^{\top}\bw(\tau) \in N_{C}(\bx(\tau))$.
\end{thm}

\begin{proof}
Let $\tilde{\bx}(\tau) \coloneqq \bxo + \matr{M}(\tau)\bd(\bxo)$. Then $\tilde{\bx}(0) = \bxo$. By the fundamental theorem of calculus,
\begin{equation*}
\dot{\tilde{\bx}}(\tau) = \dot{\matr{M}}(\tau)\bd(\bxo) = e^{-\matr{P}^{\perp}_{J(\bxo)}\matr{Q}\tau}\bd(\bxo).
\end{equation*}
In the next three steps, we prove that $\tau_{*}(\bxo) > 0$, then that $\tilde{\bx}(\tau) \in C$, $\eqset(\tilde{\bx}(\tau)) = J(\bxo)$, and $\bd(\tilde{\bx}(\tau)) = -\matr{P}^{\perp}_{J(\bxo)}\bigl(\matr{Q}\tilde{\bx}(\tau) + \bc\bigr)$ for $\tau \in (0,\tau_{*}(\bxo))$, and finally that $\tau \mapsto \tilde{\bx}(\tau)$ is the slow solution.

\paragraph{Step~1.} Let $\bh(\tau) \coloneqq \matr{A}\tilde{\bx}(\tau) - \bb$. Then 
$\bh(0) = \matr{A}\bxo - \bb$ and $\dot{\bh}(0) = \matr{A}\bd(\bxo)$. By definition of $\eqset(\bxo)$ and by~\eqref{eq:SCC},
\begin{equation*}
\bh_{\eqcset(\bxo)}(0) < \boldsymbol{0}, \ \bh_{\eqset(\bxo)\setminus J(\bxo)}(0) = \boldsymbol{0}, \ \text{and} \ \dot{\bh}_{\eqset(\bxo)\setminus J(\bxo)}(0) = [\matr{A}\bd(\bxo)]_{\eqset(\bxo) \setminus J(\bxo)} < \boldsymbol{0}.
\end{equation*}
Thus $\bh_{\Jcbxo}(\tau) < \boldsymbol{0}$ for sufficiently small $\tau > 0$, giving $\tau_{\mathrm{feas}}(\bxo) > 0$. Next, the function $\tau \mapsto \bw(\tau)$ defined in~\eqref{eq:w_mexp} is continuous with $\bw(0) = \bp_{J(\bxo)}(\bxo) > \boldsymbol{0}$. Hence $\bw(\tau) > \boldsymbol{0}$ for sufficiently small $\tau > 0$, giving $\tau_{J}(\bxo) > 0$. We thus find $\tau_{*}(\bxo) > 0$.

\paragraph{Step~2a.} We first show that $[\matr{A}\tilde{\bx}(\tau)]_{J(\bxo)} = \bb_{J(\bxo)}$ for every $\tau \geqslant 0$. Since $J(\bxo) \subset \eqset(\bxo)$, we have $\matr{A}_{J(\bxo)}\bxo = \bb_{J(\bxo)}$, and~\eqref{eq:active_rows_invariant} gives
\[
\matr{A}_{J(\bxo)}\tilde{\bx}(\tau) - \bb_{J(\bxo)} = \matr{A}_{J(\bxo)}\bigl(\bxo + \matr{M}(\tau)\bd(\bxo)\bigr) - \bb_{J(\bxo)} = \boldsymbol{0}, \qquad \tau \geqslant 0.
\]
This identity shows $J(\bxo) \subset \eqset(\tilde{\bx}(\tau))$ whenever $\tilde{\bx}(\tau) \in C$.

\paragraph{Step~2b.} Conversely, $\bh_{\Jcbxo}$ is continuous and negative near zero by Step~1, and none of its components vanishes on $(0,\tau_{\mathrm{feas}}(\bxo))$ by definition~\eqref{eq:tau_feas}. By the intermediate value theorem, each of its components is negative on $(0,\tau_{\mathrm{feas}}(\bxo))$. Now let $\tau \in (0,\tau_{*}(\bxo))$. Since $\tau_{*}(\bxo) \leqslant \tau_{\mathrm{feas}}(\bxo)$, we have $\bh_{\Jcbxo}(\tau) < \boldsymbol{0}$, so Step~2a shows that $\tilde{\bx}(\tau) \in C$ and $\eqset(\tilde{\bx}(\tau)) \subset J(\bxo)$. We conclude that $\eqset(\tilde{\bx}(\tau)) = J(\bxo)$ for all $\tau \in (0,\tau_{*}(\bxo))$.

\paragraph{Step~2c.} We show that $\bd(\tilde{\bx}(\tau)) = -\matr{P}^{\perp}_{J(\bxo)}\bigl(\matr{Q}\tilde{\bx}(\tau) + \bc\bigr)$ for all $\tau \in [0,\tau_{*}(\bxo))$. At $\tau = 0$, this is formula~\eqref{eq:steepest_dd}, so let $\tau \in (0,\tau_{*}(\bxo))$. Splitting $-\matr{Q}\tilde{\bx}(\tau) - \bc$ along $\matr{I}_{n \times n} = \matr{P}^{\perp}_{J(\bxo)} + (\matr{I}_{n \times n} - \matr{P}^{\perp}_{J(\bxo)})$ and expanding $\tilde{\bx}(\tau) = \bxo + \matr{M}(\tau)\bd(\bxo)$ with formulas~\eqref{eq:nnls_explicit} and~\eqref{eq:w_mexp} gives the decomposition
\begin{equation*}
-\matr{Q}\tilde{\bx}(\tau) - \bc = -\matr{P}^{\perp}_{J(\bxo)}\bigl(\matr{Q}\tilde{\bx}(\tau) + \bc\bigr) + \matr{A}_{J(\bxo)}^{\top}\bw(\tau).
\end{equation*}
Since $\eqset(\tilde{\bx}(\tau)) = J(\bxo)$ and $\bw(\tau) \geqslant \boldsymbol{0}$ by definition~\eqref{eq:tau_J} of $\tau_{J}(\bxo)$, the second term on the right-hand side is a nonnegative combination of the rows active at $\tilde{\bx}(\tau)$, so it belongs to $N_{C}(\tilde{\bx}(\tau))$. The first term lies in $\ker\matr{A}_{J(\bxo)}$ because $\matr{A}_{J(\bxo)}\matr{P}^{\perp}_{J(\bxo)} = \matr{0}$, so it belongs to $T_{C}(\tilde{\bx}(\tau))$ and is orthogonal to the second term. By Moreau's decomposition theorem~\eqref{eq:moreau_decomp}, the first term is therefore $\proj_{T_{C}(\tilde{\bx}(\tau))}(-\matr{Q}\tilde{\bx}(\tau) - \bc) = \bd(\tilde{\bx}(\tau))$, as claimed.

\paragraph{Step~3.} We now verify that $\tau \mapsto \tilde{\bx}(\tau)$ solves the slow system~\eqref{eq:slow_ode}. Step~2c and formula~\eqref{eq:steepest_dd} give, for all $\tau \in [0,\tau_{*}(\bxo))$,
\begin{equation*}
\begin{alignedat}{1}
\bd(\tilde{\bx}(\tau)) &= -\matr{P}^{\perp}_{J(\bxo)}(\matr{Q}\bxo + \bc) - \matr{P}^{\perp}_{J(\bxo)}\matr{Q}\matr{M}(\tau)\bd(\bxo) \\
&= \bigl(\matr{I}_{n \times n} - \matr{P}^{\perp}_{J(\bxo)}\matr{Q}\matr{M}(\tau)\bigr)\bd(\bxo) = e^{-\matr{P}^{\perp}_{J(\bxo)}\matr{Q}\tau}\bd(\bxo) = \dot{\tilde{\bx}}(\tau),
\end{alignedat}
\end{equation*}
where the third equality follows from the fundamental theorem of calculus. Hence the candidate solution $\tau \mapsto \tilde{\bx}(\tau)$ solves the slow system on $[0,\tau_{*}(\bxo))$. Both $\tau \mapsto \bx(\tau)$ and $\tau \mapsto \tilde{\bx}(\tau)$ therefore solve the differential inclusion~\eqref{eq:main_grad_incl} on $[0,\tau_{*}(\bxo))$ and agree at $\tau = 0$. The locally Lipschitz function $\tau \mapsto \normsq{\bx(\tau) - \tilde{\bx}(\tau)}$ is differentiable at almost every $\tau$, with derivative $2\langle\dot{\bx}(\tau) - \dot{\tilde{\bx}}(\tau),\bx(\tau) - \tilde{\bx}(\tau)\rangle \leqslant 0$ because $\partial_{\bx}V$ is monotone. This nonnegative function vanishes at $\tau = 0$, so it vanishes identically because it is nonincreasing, and $\bx(\tau) = \tilde{\bx}(\tau) = \bxo + \matr{M}(\tau)\bd(\bxo)$ on $[0,\tau_{*}(\bxo))$. The identity $\eqset(\bx(\tau)) = J(\bxo)$ and formula~\eqref{eq:local_multipliers} then follow from Steps~2b and~2c.
\end{proof}

\subsection{Continuation of the local dynamics}
\label{subsec:continuation}
The feasibility and active-set times $\tau_{\mathrm{feas}}(\bxo)$ and $\tau_{J}(\bxo)$, when finite, indicate when the equicorrelation set grows or the active set may shrink. Since $\tau \mapsto \bx(\tau)$ is continuous by Theorem~\ref{thm:MSP}, formula~\eqref{eq:local_solution} extends to $\tau_{*}(\bxo)$ whenever $\tau_{*}(\bxo) < +\infty$. The right derivative $\tau \mapsto \dot{\bx}(\tau) = \bd(\bx(\tau))$, however, need not be continuous at $\tau_{*}(\bxo)$, as the next result demonstrates. 

\begin{prop}\label{prop:continuation}
Let $\bxo \in C$, choose $\bp(\bxo)$ so that~\eqref{eq:LI} holds, assume that~\eqref{eq:SCC} holds, and suppose $\tau_{*}(\bxo) < +\infty$. If $[\matr{A}e^{-\matr{P}^{\perp}_{J(\bxo)}\matr{Q}\tau_{*}(\bxo)}\bd(\bxo)]_{i} > 0$ for some $i \in \eqset(\bx(\tau_{*}(\bxo)))$, then the right derivative $\tau \mapsto \dot{\bx}(\tau) = \bd(\bx(\tau))$ is discontinuous at $\tau_{*}(\bxo)$, and
\begin{equation}
\normtwo{\bd(\bx(\tau_{*}(\bxo)))} < \normtwo{e^{-\matr{P}_{J(\bxo)}^{\perp}\matr{Q}\tau_{*}(\bxo)}\bd(\bxo)},
\label{eq:discontinuity}
\end{equation}
where the right-hand side is the norm of the left limit of $\bd(\bx(\tau))$ as $\tau \to \tau_{*}(\bxo)^{-}$. Otherwise, the right derivative is continuous.
\end{prop}
\begin{proof}
Let $\tau_{*} = \tau_{*}(\bxo)$, $\bx_{*} = \bx(\tau_{*})$, and $\eqset_{*} = \eqset(\bx_{*})$, for short.

\paragraph{Step~1.} By Theorem~\ref{thm:local_solution}, $\eqset(\bx(\tau)) = J(\bxo)$ for all $\tau \in (0,\tau_{*})$. Let
\[
I^{+} \coloneqq \left\{i \in \Jcbxo : [\matr{A}\bx_{*}]_{i} = b_{i}\right\},
\]
so that $\eqset_{*} = J(\bxo) \cup I^{+}$. Indeed, $\bx_{*} = \bxo + \matr{M}(\tau_{*})\bd(\bxo)$ by continuity, so $\matr{A}_{J(\bxo)}\bxo = \bb_{J(\bxo)}$ and~\eqref{eq:active_rows_invariant} give $\matr{A}_{J(\bxo)}\bx_{*} = \bb_{J(\bxo)}$. Define $\hat{\bp}$ by $\hat{\bp}_{J(\bxo)} \coloneqq \bw(\tau_{*})$ and $\hat{\bp}_{\Jcbxo} \coloneqq \boldsymbol{0}$. Since $\bw(\tau) \geqslant \boldsymbol{0}$ for $\tau \in (0,\tau_{*})$ by~\eqref{eq:local_multipliers} and $\bw$ is continuous, $\hat{\bp} \geqslant \boldsymbol{0}$. Because $\hat{\bp}$ is supported on $J(\bxo) \subset \eqset_{*}$, it is a feasible solution of the NNLS problem~\eqref{eq:nnls_problem} at $\bx_{*}$. Letting $\tau \to \tau_{*}^{-}$ in~\eqref{eq:local_multipliers} and~\eqref{eq:local_simple_ode} gives
\begin{equation}\label{eq:bphat_left}
	-\Atop\hat{\bp} - \matr{Q}\bx_{*} - \bc = \left(\matr{I}_{n\times n} - \matr{P}^{\perp}_{J(\bxo)}\matr{Q}\matr{M}(\tau_{*})\right)\bd(\bxo) = e^{-\matr{P}^{\perp}_{J(\bxo)}\matr{Q}\tau_{*}}\bd(\bxo).
\end{equation}

\paragraph{Step~2.} We now determine when $\hat{\bp}$ is an optimal solution of the NNLS problem~\eqref{eq:nnls_problem} at $\bx_{*}$. By Section~\ref{subsec:properties_nnls}, $\hat{\bp}$ is optimal if and only if the KKT conditions~\eqref{eq:nnls_oc_1}--\eqref{eq:nnls_oc_2} hold for its residual $-\Atop\hat{\bp} - \matr{Q}\bx_{*} - \bc$. On the one hand, the sign condition~\eqref{eq:nnls_oc_1} reads
\begin{equation}\label{eq:kkt_bpstar_1}
	\matr{A}_{\eqset_{*}}\left(-\Atop\hat{\bp} - \matr{Q}\bx_{*} - \bc\right) = \matr{A}_{\eqset_{*}}e^{-\matr{P}_{J(\bxo)}^{\perp}\matr{Q}\tau_{*}}\bd(\bxo) \leqslant \boldsymbol{0},
\end{equation}
with equality over $J(\bxo) \subset \eqset_{*}$ because $\matr{A}_{J(\bxo)}e^{-\matr{P}^{\perp}_{J(\bxo)}\matr{Q}s}\bd(\bxo) = \boldsymbol{0}$ for all $s \geqslant 0$, as differentiating~\eqref{eq:active_rows_invariant} shows. Since $\hat{\bp}$ is supported on $J(\bxo)$, this equality settles the complementarity condition~\eqref{eq:nnls_oc_2}, so $\hat{\bp}$ is optimal if and only if~\eqref{eq:kkt_bpstar_1} holds. On the other hand, for every $i \in I^{+}$, the function $\tau \mapsto [\matr{A}(\bxo + \matr{M}(\tau)\bd(\bxo)) - \bb]_{i}$ is negative on $(0,\tau_{*})$ by Step~1 and vanishes at $\tau_{*}$. Thus the derivative of the function at $\tau_{*}$ is nonnegative, yielding
\[
\frac{\diff}{\diff\tau}\bigg|_{\tau = \tau_{*}}\left([\matr{A}(\bxo + \matr{M}(\tau)\bd(\bxo))]_{i}\right) = [\matr{A}e^{-\matr{P}_{J(\bxo)}^{\perp}\matr{Q}\tau_{*}}\bd(\bxo)]_{i} \geqslant 0.
\]
Since $[\matr{A}e^{-\matr{P}^{\perp}_{J(\bxo)}\matr{Q}\tau_{*}}\bd(\bxo)]_{J(\bxo)} = \boldsymbol{0}$ and $\eqset_{*} = J(\bxo) \cup I^{+}$, condition~\eqref{eq:kkt_bpstar_1} holds if and only if $[\matr{A}e^{-\matr{P}^{\perp}_{J(\bxo)}\matr{Q}\tau_{*}}\bd(\bxo)]_{I^{+}} = \boldsymbol{0}$. Therefore, $\hat{\bp}$ is an optimal NNLS solution at $\bx_{*}$ if and only if the hypothesis of the proposition fails.

\paragraph{Step~3a.} If the hypothesis of the proposition fails, then Step~2 shows $\hat{\bp}$ is an optimal NNLS solution. In that case,
\[
\dot{\bx}(\tau_{*}) = \bd(\bx_{*}) = -\left(\Atop\hat{\bp} + \matr{Q}\bx_{*} + \bc\right) = e^{-\matr{P}^{\perp}_{J(\bxo)}\matr{Q}\tau_{*}}\bd(\bxo),
\]
where the last equality is~\eqref{eq:bphat_left}. By~\eqref{eq:local_simple_ode}, the right-hand side is also the left limit of $\dot{\bx}(\tau)$ as $\tau \to \tau_{*}^{-}$. Hence $\tau \mapsto \dot{\bx}(\tau)$ is continuous at $\tau_{*}$.

\paragraph{Step~3b.} If the hypothesis of the proposition holds, then Step~2 shows that $\hat{\bp}$ is not an optimal NNLS solution. Since $\hat{\bp}$ is feasible but not optimal, its objective value exceeds the optimal value $\frac{1}{2}\normsq{\bd(\bx_{*})}$, and~\eqref{eq:bphat_left} yields $\normtwo{\bd(\bx_{*})} < \normtwo{\Atop\hat{\bp} + \matr{Q}\bx_{*} + \bc} = \normtwo{e^{-\matr{P}^{\perp}_{J(\bxo)}\matr{Q}\tau_{*}}\bd(\bxo)}$. This gives~\eqref{eq:discontinuity}. By~\eqref{eq:local_simple_ode}, the right-hand side of~\eqref{eq:discontinuity} is the norm of the left limit of $\dot{\bx}(\tau)$ as $\tau \to \tau_{*}^{-}$, so this left limit differs from $\dot{\bx}(\tau_{*}) = \bd(\bx_{*})$ and $\tau \mapsto \dot{\bx}(\tau)$ is discontinuous at $\tau_{*}$.
\end{proof}

\subsection{Computing the breakpoint time}
\label{subsec:computations_local}
Theorem~\ref{thm:local_solution} and Proposition~\ref{prop:continuation} describe the slow solution up to and at the breakpoint time $\tau_{*}(\bxo)$. Computing this time requires finding the first zero of the function
\[
\tau \mapsto \left[\matr{A}(\bxo + \matr{M}(\tau)\bd(\bxo)) - \bb\right]_{i}, \quad i \in \Jcbxo,
\]
and the first time that a component of $\bw(\tau)$ in~\eqref{eq:w_mexp} turns negative, if any exist. Condition~\eqref{eq:nnls_oc_2_5} gives $\bd(\bxo) \in \ker(\matr{A}_{J})$ (where we write $J \coloneqq J(\bxo)$ for short). This kernel is the range of $\matr{P}^{\perp}_{J}$ and is invariant under $\matr{P}^{\perp}_{J}\matr{Q}$, which agrees on it with the SPSD matrix $\matr{P}^{\perp}_{J}\matr{Q}\matr{P}^{\perp}_{J}$. Expanding $\bd(\bxo)$ in an orthonormal basis of $\ker(\matr{A}_{J})$ made of eigenvectors of $\matr{P}^{\perp}_{J}\matr{Q}\matr{P}^{\perp}_{J}$, with eigenvalues $\lambda_{k} \geqslant 0$, and integrating as in~\eqref{eq:int_mexp} shows that each component of $\matr{M}(\tau)\bd(\bxo)$ is a linear combination of terms of the form
\[
\left(1-e^{-\lambda_{k}\tau}\right)/\lambda_{k}\ \text{(for $\lambda_{k} > 0$) and $\tau$ (for $\lambda_{k} = 0$)}.
\]
Thus the equations for $\tau_{\mathrm{feas}}(\bxo)$ and $\tau_{J}(\bxo)$ generally have no closed-form solution.

These equations, however, have closed-form solutions when $\bd(\bxo)$ is an eigenvector of $\matr{P}^{\perp}_{J}\matr{Q}$, since each then involves a single term. The direction $\bd(\bxo)$ is such an eigenvector when $\matr{Q} = t\matr{I}_{n\times n}$ with $t \geqslant 0$, because $\bd(\bxo) \in \ker(\matr{A}_{J})$ gives $\matr{P}^{\perp}_{J}\matr{Q}\bd(\bxo) = t\bd(\bxo)$. Thus $e^{-\matr{P}_{J}^{\perp}\matr{Q}\tau}\bd(\bxo) = e^{-t\tau}\bd(\bxo)$ and the trajectory of the slow solution on $[0,\tau_{*}(\bxo))$ is
\begin{equation}\label{eq:x_Q_equals_t_id}
\bx(\tau) = \bxo + \Delta(\tau)\bd(\bxo),\quad \text{with} \ \Delta(\tau) \coloneqq \begin{cases}
 \tau &\text{if $t=0$}, \\
 (1-e^{-t\tau})/t &\text{if $t>0$}.
 \end{cases}
\end{equation} 
For the active-set time $\tau_{J}(\bxo)$, we use $\matr{M}(\tau)\bd(\bxo) = \Delta(\tau)\bd(\bxo)$ and $\bd(\bxo) \in \ker(\matr{A}_{J})$ to get
\[
\bw(\tau) = \bp_{J}(\bxo) - t\Delta(\tau)\left(\matr{A}_{J}\matr{A}_{J}^{\top}\right)^{-1}\matr{A}_{J}\bd(\bxo) = \bp_{J}(\bxo) > \boldsymbol{0}
\]
for all $\tau > 0$. Hence we find $\tau_{J}(\bxo) = +\infty$. For the feasibility time $\tau_{\mathrm{feas}}(\bxo)$, each constraint $i \in \eqset(\bxo) \setminus J$ satisfies $[\matr{A}(\bxo + \Delta(\tau)\bd(\bxo)) - \bb]_{i} = \Delta(\tau)[\matr{A}\bd(\bxo)]_{i} < 0$ for $\tau > 0$ by~\eqref{eq:SCC}. By~\eqref{eq:delta_feas} and~\eqref{eq:tau_feas}, $\tau_{\mathrm{feas}}(\bxo)$ is therefore the solution of $\Delta(\tau) = \Delta_{\mathrm{feas}}(\bxo)$, if one exists. Since $\tau \mapsto \Delta(\tau)$ is strictly increasing with $\Delta(0) = 0$ and, for $t>0$, $\Delta(\tau) \to 1/t$ as $\tau \to +\infty$, the feasibility time is finite if and only if $\Delta_{\mathrm{feas}}(\bxo) < 1/t$. When this holds,
\begin{equation}\label{eq:exact_bp_time_qid}
\tau_{\mathrm{feas}}(\bxo) = \begin{cases}
 \Delta_{\mathrm{feas}}(\bxo) &\text{if $t=0$}, \\
 -\log(1-t\Delta_{\mathrm{feas}}(\bxo))/t &\text{if $t>0$},
 \end{cases}
\end{equation}
and $\tau_{*}(\bxo) = \tau_{\mathrm{feas}}(\bxo)$. If $\Delta_{\mathrm{feas}}(\bxo) \geqslant 1/t$, then $\tau_{\mathrm{feas}}(\bxo) = +\infty$, so $\tau_{*}(\bxo) = +\infty$ and~\eqref{eq:x_Q_equals_t_id} holds for all $\tau \geqslant 0$. For $t > 0$, $\bx(\tau) \to \bxo + \bd(\bxo)/t$ as $\tau \to +\infty$. For $t = 0$ and $\bd(\bxo) \neq \boldsymbol{0}$, formula~\eqref{eq:obj_fun_after_descent} gives $V(\bx(\tau)) = V(\bxo) - \tau\normsq{\bd(\bxo)} \to -\infty$.

\subsection{Removing strict complementarity}
\label{subsec:removing_scc_local}
All results so far except Proposition~\ref{prop:descent_direction} assume the strict complementarity condition~\eqref{eq:SCC}. When~\eqref{eq:SCC} fails, the \emph{grazing set}
\begin{equation}\label{eq:grazing_set}
G(\bxo) \coloneqq \left\{i \in \eqset(\bxo) \setminus J(\bxo) : [\matr{A}\bd(\bxo)]_{i} = 0\right\}
\end{equation}
is nonempty. A \emph{grazing constraint} stays active on the ray $\bxo + \Delta\bd(\bxo)$. Along the slow solution, it may join the active set, become inactive, or stay active with a zero multiplier. The direction $\bd(\bxo)$ on its own does not determine which of these occurs. The rate at which the direction $\bd(\bx(\tau))$ changes at $\tau = 0$ is the projection of $-\matr{Q}\bd(\bxo)$ onto the cone $\{\bs \in \Rn : [\matr{A}\bs]_{J(\bxo)} = \boldsymbol{0},\ [\matr{A}\bs]_{G(\bxo)} \leqslant \boldsymbol{0}\}$. A grazing constraint joins the active set if some multiplier of this projection is positive on it, and becomes inactive if the projection strictly decreases its value. Otherwise, the higher derivatives of $\bd(\bx(\tau))$ at $\tau = 0$ determine which case occurs.

The next theorem, whose proof is highly technical and relegated to Appendix~\ref{app:thm:scc_local}, extends the results of Sections~\ref{subsec:local_sol}--\ref{subsec:continuation} to the case where~\eqref{eq:SCC} fails. Complementarity~\eqref{eq:nnls_oc_2} holds for every optimal NNLS solution, so the set $J(\bxo) \cup G(\bxo) = \{i \in \eqset(\bxo) : [\matr{A}\bd(\bxo)]_{i} = 0\}$ does not depend on the choice of $\bp(\bxo)$. For an index set $K \subseteq [m]$, let $\matr{P}^{\perp}_{K}$ denote the orthogonal projection onto $\ker\matr{A}_{K}$, given by~\eqref{eq:nnls_pperp} with $K$ in place of $J(\bxo)$ when the rows of $\matr{A}_{K}$ are linearly independent. For a set $K$ with $J(\bxo) \subseteq K \subseteq J(\bxo) \cup G(\bxo)$, we write $\matr{M}_{K}$ and $\tau^{K}_{\mathrm{feas}}(\bxo)$ for~\eqref{eq:int_mexp} and~\eqref{eq:tau_feas} with $K$ in place of $J(\bxo)$. Here, $\tau^{K}_{\mathrm{feas}}(\bxo)$ ignores every constraint whose value vanishes identically along the curve $\bxo + \matr{M}_{K}(\tau)\bd(\bxo)$. Let
\begin{equation}\label{eq:tau_J_cone}
\begin{aligned}
\tau^{K}_{J}(\bxo) \coloneqq \inf\Bigl\{\tau > 0 :\ & \by \in C \ \text{and} \ -\matr{Q}\by - \bc - e^{-\matr{P}^{\perp}_{K}\matr{Q}\tau}\bd(\bxo) \notin N_{C}(\by), \\
& \text{where } \by = \bxo + \matr{M}_{K}(\tau)\bd(\bxo)\Bigr\},
\end{aligned}
\end{equation}
be the first time at which the normal part along the curve leaves the normal cone, and let $\tau^{K}_{*}(\bxo) \coloneqq \min(\tau^{K}_{\mathrm{feas}}(\bxo),\tau^{K}_{J}(\bxo))$. Unlike~\eqref{eq:tau_J}, definition~\eqref{eq:tau_J_cone} does not have multipliers because the rows of $\matr{A}_{K}$ can be linearly dependent.

\begin{thm}
\label{thm:scc_local}
Let $\bxo \in C$ and choose $\bp(\bxo)$ so that~\eqref{eq:LI} holds. There exists a set $K$ with $J(\bxo) \subseteq K \subseteq J(\bxo) \cup G(\bxo)$ such that the following hold.
\begin{itemize}
\item[(i)] The breakpoint time $\tau^{K}_{*}(\bxo)$ is positive, and the slow solution satisfies
\begin{equation}\label{eq:grazing_solution}
\bx(\tau) = \bxo + \matr{M}_{K}(\tau)\bd(\bxo),
\qquad
\dot{\bx}(\tau) = e^{-\matr{P}^{\perp}_{K}\matr{Q}\tau}\bd(\bxo)
\qquad \text{on } [0,\tau^{K}_{*}(\bxo)).
\end{equation}
Moreover, for all sufficiently small $\tau > 0$, there exists $\bw \in \R^{|K|}$ with positive entries such that
\begin{equation}\label{eq:grazing_multipliers}
-\matr{Q}\bx(\tau) - \bc = \bd(\bx(\tau)) + \matr{A}_{K}^{\top}\bw.
\end{equation}

\item[(ii)] Suppose $\tau^{K}_{*}(\bxo) < +\infty$. If $[\matr{A}e^{-\matr{P}^{\perp}_{K}\matr{Q}\tau^{K}_{*}(\bxo)}\bd(\bxo)]_{i} > 0$ for some $i \in \eqset(\bx(\tau^{K}_{*}(\bxo)))$, then the right derivative $\tau \mapsto \dot{\bx}(\tau)$ is discontinuous at $\tau^{K}_{*}(\bxo)$, and
\begin{equation}\label{eq:scc_discontinuity}
\normtwo{\bd(\bx(\tau^{K}_{*}(\bxo)))} < \normtwo{e^{-\matr{P}^{\perp}_{K}\matr{Q}\tau^{K}_{*}(\bxo)}\bd(\bxo)},
\end{equation}
where the right-hand side is the norm of the left limit of $\dot{\bx}(\tau)$ as $\tau \to \tau^{K}_{*}(\bxo)^{-}$. Otherwise, the right derivative is continuous.

\item[(iii)] If the rows of $\matr{A}_{J(\bxo) \cup G(\bxo)}$ are linearly independent, then $J(\bx(\tau)) = K$ for all sufficiently small $\tau > 0$, and $K$ is the only set with $J(\bxo) \subseteq K \subseteq J(\bxo) \cup G(\bxo)$ for which part~(i) holds.
\end{itemize}
\end{thm}

We denote by $J^{+}_{*}(\bxo)$ the set $K$ constructed in Appendix~\ref{app:thm:scc_local}, and write $\matr{M}_{*}$, $\tau^{+}_{\mathrm{feas}}(\bxo)$, $\tau^{+}_{J}(\bxo)$, and $\tau^{+}_{*}(\bxo)$ for the objects built from it. By part~(iii), $K$ is unique when the rows of $\matr{A}_{J(\bxo) \cup G(\bxo)}$ are linearly independent. When they are linearly dependent, $K$ need not be unique, but the slow solution is. Under~\eqref{eq:SCC}, $G(\bxo) = \emptyset$, so $J^{+}_{*}(\bxo) = J(\bxo)$. For $\tau \in (0,\tau_{\mathrm{feas}}(\bxo))$, the arguments of Steps~2b and~2c of the proof of Theorem~\ref{thm:local_solution} give $\eqset(\by) = J(\bxo)$ and the normal part $\matr{A}_{J(\bxo)}^{\top}\bw(\tau)$, so~\eqref{eq:tau_J_cone} reduces to the sign test in~\eqref{eq:tau_J}. Hence $\tau^{+}_{*}(\bxo) = \tau_{*}(\bxo)$, and Theorem~\ref{thm:scc_local} reduces to Theorem~\ref{thm:local_solution} and Proposition~\ref{prop:continuation}.

\section{Global solution of the slow system}
\label{sec:results_II}
\subsection{Global dynamics of the slow system}
\label{subsec:global_sol}
Section~\ref{sec:results_I} solved the slow system~\eqref{eq:slow_ode} explicitly on an interval $[0,\tau^{+}_{*}(\bxo))$. When $\tau^{+}_{*}(\bxo)$ is finite, continuity extends the solution to $\tau = \tau^{+}_{*}(\bxo)$. In this case, since $\bx(\tau^{+}_{*}(\bxo)) \in C$ and the slow system is autonomous, Theorem~\ref{thm:scc_local} applies again at the point $\bx(\tau^{+}_{*}(\bxo))$ and adds a second piece.

We now repeat these arguments to construct the global solution. Letting $\bx^{(0)} = \bxo \in C$, $\tau^{(0)} = 0$, and $k=0$, we iterate, with $\bp^{(k)}$ chosen so that~\eqref{eq:LI} holds at $\bx^{(k)}$,
\begin{equation}\label{eq:global_iteration}
\begin{alignedat}{1}
\bp^{(k)} &\in \argmin_{\substack{\bp_{\eqset(\bx^{(k)})}\geqslant \boldsymbol{0} \\ \bp_{\eqset(\bx^{(k)})^{\mathsf{C}}} = \boldsymbol{0}}} \frac{1}{2}\normsq{\Atop\bp + \matr{Q}\bx^{(k)} + \bc}, \quad \bd^{(k)} = -\Atop\bp^{(k)} - \matr{Q}\bx^{(k)} - \bc \\
J_{*}^{(k)} &= J^{+}_{*}(\bx^{(k)}), \quad \tau_{*}^{(k)} = \tau_{*}^{+}(\bx^{(k)}) \\
\bx^{(k+1)} &= \bx^{(k)} + \matr{M}_{J_{*}^{(k)}}(\tau_{*}^{(k)})\bd^{(k)}, \quad \tau^{(k+1)} = \tau^{(k)} + \tau_{*}^{(k)}.
\end{alignedat}
\end{equation}
The iteration stops at the first $k$ with $\tau_{*}^{(k)} = +\infty$. The slow solution is then
\begin{equation}\label{eq:global_solution}
\bx(\tau) = \bx^{(k)} + \matr{M}_{J_{*}^{(k)}}(\tau-\tau^{(k)})\bd^{(k)} \ \text{over} \ \tau \in [\tau^{(k)},\tau^{(k+1)}).
\end{equation}
We call the times $\tau^{(k)}$ with $k \geqslant 1$ the \emph{breakpoints} of the slow solution, the interval $[\tau^{(k)},\tau^{(k+1)})$ its $k$-th \emph{piece}, and $J_{*}^{(k)}$ the active set of that piece. Formula~\eqref{eq:global_solution} extends the slow solution to the interval $[0,\tau^{(k)})$ for every $k\in \mathbb{N}$ with $\tau^{(k)} < +\infty$.

\subsection{Finite termination of the slow system}
\label{subsec:finite_general}
What happens to the global dynamics as $k \to +\infty$? We show below that the iteration~\eqref{eq:global_iteration} terminates after finitely many steps, i.e., there exists $K \in \mathbb{N}$ such that $\tau_{*}^{(K)} = +\infty$. Moreover, the final piece either yields an optimal solution of problem~\eqref{eq:intro_main_problem} or certifies that the problem is unbounded.

\begin{thm}
\label{thm:finite_general}
Let $\bxo \in C$ and set $\bx^{(0)} = \bxo$. Then there exists $K \in \N$ such that $\tau_{*}^{(K)} = +\infty$ and
\[
\bd^{s} \coloneqq \lim_{\tau \to +\infty}\bd(\bx(\tau)) = \proj_{\ker\matr{Q}\,\cap\,\ker\matr{A}_{J_{*}^{(K)}}}(\bd^{(K)}).
\]
Moreover:
\begin{itemize}
\item[(i)] If $\bd^{s} = \boldsymbol{0}$, then $\lim_{\tau \to +\infty}\bx(\tau) = \bx^{(K)} + (\matr{P}^{\perp}_{J_{*}^{(K)}}\matr{Q}\matr{P}^{\perp}_{J_{*}^{(K)}})^{\dagger}\bd^{(K)}$ is an optimal solution of problem~\eqref{eq:intro_main_problem}, where the dagger denotes the pseudo-inverse.
\item[(ii)] If $\bd^{s} \neq \boldsymbol{0}$, then $\matr{A}\bd^{s} \leqslant \boldsymbol{0}$, $\matr{Q}\bd^{s} = \boldsymbol{0}$, and $\langle\bc,\bd^{s}\rangle = -\normsq{\bd^{s}} < 0$. In particular, problem~\eqref{eq:intro_main_problem} is unbounded below along $\bd^{s}$.
\end{itemize}
\end{thm}
\begin{proof}
The idea of the proof is that the constraint values $[\matr{A}\bx(\tau)]_{i} - b_{i}$ and the multipliers that decide when a constraint leaves the active set eventually keep a fixed sign. If there were infinitely many breakpoints, the trajectory would approach a motion with constant velocity, and its deviation from this motion would be a sum of fixed vectors with coefficients of decreasing size. The constraint values and the multipliers are linear in this deviation, so the first term that contributes to one of them determines its sign unless it vanishes identically. No constraint would then enter or leave the active set after some breakpoint, leading to a contradiction. The formulas for $\bd^{s}$ and the limit in (i) follow from the final piece~\eqref{eq:global_solution}. The proof is highly technical and is relegated to Appendix~\ref{app:finite_general}.
\end{proof}

\begin{rem}[The symmetric positive definite case]
\label{rem:finite_general_spd}
If $\matr{Q}$ is SPD, then $\matr{Q}\bd^{s} = \boldsymbol{0}$ forces $\bd^{s} = \boldsymbol{0}$, so $\bx(\tau)$ converges to the optimal solution $\bxsol$, unique by strict convexity.
\end{rem}

\subsection{An exact algorithm for linear programs and projections onto polyhedra}
\label{subsec:asymptotic_finite_term}
We consider here the special case where $\matr{Q} = t\matr{I}_{n\times n}$, corresponding to linear programs ($t = 0$) and projections onto a polyhedron $(t>0)$. In this case, the slow solution can be computed exactly and no active set is visited more than once. Each piece of the slow solution is then a ray, as in Section~\ref{subsec:computations_local}, and the following iteration computes these rays. Let $\bx^{(0)} = \bxo \in C$. For $k = 0,1,2,\dots$, choose the solution $\bp^{(k)} = \bp(\bx^{(k)})$ of the NNLS problem~\eqref{eq:nnls_problem} so that~\eqref{eq:LI} holds at $\bx^{(k)}$. Let $\bd^{(k)} = \bd(\bx^{(k)})$ and let $\Delta^{(k)} = \Delta_{\mathrm{feas}}(\bx^{(k)})$ be given by~\eqref{eq:delta_feas}, and set
\begin{equation}\label{eq:tI_iteration}
\bx^{(k+1)} = \bx^{(k)} + \Delta^{(k)}\bd^{(k)} \quad \text{as long as } \Delta^{(k)} < 1/t.
\end{equation}
We write $K$ for the first index with $\Delta^{(K)} \geqslant 1/t$, which we prove exists.

\begin{thm}
\label{thm:finite_tI}
Let $\matr{Q} = t\matr{I}_{n\times n}$ with $t \geqslant 0$, $\tau^{(0)} = 0$, and let $\bx^{(0)},\bx^{(1)},\dots$ be the iterates of~\eqref{eq:tI_iteration}.
\begin{itemize}
\item[(i)] The slow solution is
\[
\bx(\tau) = \bx^{(k)} + \begin{cases}
  (\tau - \tau^{(k)})\bd^{(k)} &\text{if $t=0$}, \\
  \bigl(1-e^{-t(\tau - \tau^{(k)})}\bigr)\bd^{(k)}/t &\text{if $t>0$},
 \end{cases}
\quad \text{for } \tau \in [\tau^{(k)},\tau^{(k+1)}) \text{ and } k = 0,\dots,K.
\]
For $k < K$, the breakpoint is $\tau^{(k+1)} = \tau^{(k)} + \Delta^{(k)}$ if $t = 0$ and $\tau^{(k+1)} = \tau^{(k)} - \log(1 - t\Delta^{(k)})/t$ if $t > 0$. The final piece is $[\tau^{(K)},+\infty)$.
\item[(ii)] The norm $\normtwo{\Atop\bp^{(k)}}$ depends on $J(\bx^{(k)})$ alone and increases strictly with $k$. Hence no active set occurs twice, the iteration~\eqref{eq:tI_iteration} stops, and $K \leqslant 2^{m} - 1$.
\item[(iii)] If $t > 0$, then $\bxsol \coloneqq \bx^{(K)} + \tfrac{1}{t}\bd^{(K)}$ is the unique optimal solution of problem~\eqref{eq:intro_main_problem}. If $t = 0$ and $\bd^{(K)} = \boldsymbol{0}$, then $\bxsol \coloneqq \bx^{(K)}$ is an optimal solution. In both cases, letting $\bp^{s} \coloneqq \bp^{(K)}$, the pair $(\bxsol,\bp^{s})$ satisfies the KKT conditions. If $t = 0$ and $\bd^{(K)} \neq \boldsymbol{0}$, then the ray $\bx^{(K)} + \Delta\bd^{(K)}$, $\Delta \geqslant 0$, is feasible and $V(\bx^{(K)} + \Delta\bd^{(K)}) \to -\infty$ as $\Delta \to +\infty$.
\end{itemize}
\end{thm}
\begin{proof}
Throughout, we write $J^{(k)} \coloneqq J(\bx^{(k)})$ and $\bn^{(k)}$ for the normal component of the antigradient, so that $-\nabla_{\bx}f(\bx^{(k)}) = \bd^{(k)} + \bn^{(k)}$ with $\langle\bd^{(k)},\bn^{(k)}\rangle = 0$ by~\eqref{eq:moreau_decomp} and~\eqref{eq:nnls_oc_2}.

\paragraph{Step~1.} Fix $k \geqslant 0$, suppose $\bx(\tau^{(k)}) = \bx^{(k)}$. For $0 < \Delta < \min(\Delta^{(k)},1/t)$, where $1/t$ reads as $+\infty$ when $t = 0$, and put $\bx_{\Delta} \coloneqq \bx^{(k)} + \Delta\bd^{(k)}$. Every constraint $i$ with $[\bp^{(k)}]_{i} > 0$ lies in $\eqset(\bx^{(k)})$ and satisfies $[\matr{A}\bd^{(k)}]_{i} = 0$ by~\eqref{eq:nnls_oc_2}, so it is active at $\bx_{\Delta}$ for every $\Delta \geqslant 0$ and $\bn^{(k)} \in N_{C}(\bx_{\Delta})$. A constraint $i \in \eqset(\bx^{(k)})$ with $[\matr{A}\bd^{(k)}]_{i} < 0$ is inactive at $\bx_{\Delta}$, and so is every $i \notin \eqset(\bx^{(k)})$ by the definition~\eqref{eq:delta_feas} of $\Delta^{(k)}$. Every constraint active at $\bx_{\Delta}$ therefore satisfies $[\matr{A}\bd^{(k)}]_{i} = 0$, and $(1 - t\Delta)\bd^{(k)} \in T_{C}(\bx_{\Delta})$. Since $\nabla_{\bx}f$ is affine with slope $t\matr{I}_{n\times n}$, we have $-\nabla_{\bx}f(\bx_{\Delta}) = (1 - t\Delta)\bd^{(k)} + \bn^{(k)}$, and the two terms are orthogonal. Moreau's decomposition theorem~\eqref{eq:moreau_decomp} gives $\bd(\bx_{\Delta}) = (1 - t\Delta)\bd^{(k)}$.

The function $\tau \mapsto \Delta(\tau)$ of~\eqref{eq:x_Q_equals_t_id} satisfies $\dot{\Delta}(\tau) = e^{-t\tau} = 1 - t\Delta(\tau)$ and increases strictly and continuously from $0$ to $1/t$, with $1/t = +\infty$ when $t = 0$. The curve $\tilde{\bx}(\tau) \coloneqq \bx^{(k)} + \Delta(\tau - \tau^{(k)})\bd^{(k)}$ therefore solves the slow system~\eqref{eq:slow_ode} as long as $\Delta(\tau - \tau^{(k)}) < \Delta^{(k)}$. Since the slow system is autonomous, the monotonicity argument at the end of the proof of Theorem~\ref{thm:local_solution} shows that $\tilde{\bx}$ is the slow solution there. If $\Delta^{(k)} \geqslant 1/t$, then $k = K$ and $\bx = \tilde{\bx}$ on $[\tau^{(K)},+\infty)$. If $\Delta^{(k)} < 1/t$, then $\Delta(\tau - \tau^{(k)})$ reaches $\Delta^{(k)}$ at $\tau = \tau^{(k+1)}$ by~\eqref{eq:exact_bp_time_qid}, and the continuity of $\bx$ gives $\bx(\tau^{(k+1)}) = \bx^{(k+1)}$. In this case $\Delta^{(k)}$ is finite, so the least quotient in~\eqref{eq:delta_feas} is attained at a constraint $i$ with $[\matr{A}\bd^{(k)}]_{i} > 0$, which is active at $\bx^{(k+1)}$. Induction on $k$ proves~(i).

\paragraph{Step~2.} Let $k$ satisfy $\Delta^{(k)} < 1/t$. We first show that $\normtwo{\bn^{(k+1)}} > \normtwo{\bn^{(k)}}$. Since $\nabla_{\bx}f$ is affine with slope $t\matr{I}_{n\times n}$,
\[
-\nabla_{\bx}f(\bx^{(k+1)}) = -\nabla_{\bx}f(\bx^{(k)}) - t\Delta^{(k)}\bd^{(k)} = (1 - t\Delta^{(k)})\bd^{(k)} + \bn^{(k)}.
\]
Put $\bv \coloneqq (1 - t\Delta^{(k)})\bd^{(k)}$ and $\mathcal{K} \coloneqq \{\bs \in \Rn : [\matr{A}\bs]_{i} \leqslant 0 \text{ for every } i \text{ with } [\bp^{(k)}]_{i} > 0\}$. The vector $\bv$ lies in $\mathcal{K}$ by~\eqref{eq:nnls_oc_2}, the vector $\bn^{(k)}$ is a nonnegative combination of the rows that define $\mathcal{K}$ and so satisfies $\langle\bn^{(k)},\bs\rangle \leqslant 0$ for every $\bs \in \mathcal{K}$, and $\langle\bv,\bn^{(k)}\rangle = 0$. Since $-\nabla_{\bx}f(\bx^{(k+1)}) - \bv = \bn^{(k)}$, these three properties and the characterization~\eqref{eq:char_projections} give $\bv = \proj_{\mathcal{K}}(-\nabla_{\bx}f(\bx^{(k+1)}))$. The constraints with $[\bp^{(k)}]_{i} > 0$ are active at $\bx^{(k+1)}$ by Step~1, so $T_{C}(\bx^{(k+1)}) \subseteq \mathcal{K}$. Step~1 also gives a constraint $i$ active at $\bx^{(k+1)}$ with $[\matr{A}\bd^{(k)}]_{i} > 0$. Since $1 - t\Delta^{(k)} > 0$, this constraint satisfies $[\matr{A}\bv]_{i} > 0$, so $\bv \notin T_{C}(\bx^{(k+1)})$ and $\bv \neq \bd^{(k+1)}$. Since $\bd^{(k+1)} \in T_{C}(\bx^{(k+1)}) \subseteq \mathcal{K}$ and $\bv$ is the unique point of $\mathcal{K}$ nearest to $-\nabla_{\bx}f(\bx^{(k+1)})$, we get $\normtwo{\bn^{(k+1)}} = \normtwo{-\nabla_{\bx}f(\bx^{(k+1)}) - \bd^{(k+1)}} > \normtwo{-\nabla_{\bx}f(\bx^{(k+1)}) - \bv} = \normtwo{\bn^{(k)}}$.

Since $\matr{A}_{J^{(k)}}\bx^{(k)} = \bb_{J^{(k)}}$, formula~\eqref{eq:nnls_explicit} with $\nabla_{\bx}f(\bx^{(k)}) = t\bx^{(k)} + \bc$ gives 
\[
\bp^{(k)}_{J^{(k)}} = -(\matr{A}_{J^{(k)}}\matr{A}_{J^{(k)}}^{\top})^{-1}(t\bb_{J^{(k)}} + \matr{A}_{J^{(k)}}\bc).
\]
Hence $\bn^{(k)} = \matr{A}_{J^{(k)}}^{\top}\bp^{(k)}_{J^{(k)}}$ depends on $J^{(k)}$ alone, with $\bn^{(k)} = \boldsymbol{0}$ when $J^{(k)} = \emptyset$. Since $\normtwo{\bn^{(0)}} < \normtwo{\bn^{(1)}} < \cdots$, the active sets $J^{(0)},J^{(1)},\dots$ are distinct subsets of $[m]$. There are $2^{m}$ such subsets, so the iteration stops and $K + 1 \leqslant 2^{m}$.

\paragraph{Step~3.} First, let $t > 0$. Since $1/t \leqslant \Delta^{(K)}$, the point $\bxsol = \bx^{(K)} + \tfrac{1}{t}\bd^{(K)}$ lies in $C$ by the definition~\eqref{eq:delta_feas} of $\Delta^{(K)}$. Since $\bd^{(K)} = -\Atop\bp^{(K)} - t\bx^{(K)} - \bc$, we have
\[
\nabla_{\bx}f(\bxsol) = t\bxsol + \bc = t\bx^{(K)} + \bd^{(K)} + \bc = -\Atop\bp^{(K)},
\]
where $\bp^{(K)} \geqslant \boldsymbol{0}$ is supported on $J^{(K)}$, whose constraints are active at $\bxsol$ because $[\matr{A}\bx^{(K)}]_{J^{(K)}} = \bb_{J^{(K)}}$ and $[\matr{A}\bd^{(K)}]_{J^{(K)}} = \boldsymbol{0}$ by~\eqref{eq:nnls_oc_2_5}. The pair $(\bxsol,\bp^{(K)})$ therefore satisfies the KKT conditions, so $\bxsol$ is an optimal solution, unique by strict convexity, since $t > 0$. Next, let $t = 0$. If $\bd^{(K)} = \boldsymbol{0}$, then $\nabla_{\bx}f(\bx^{(K)}) = \bc = -\Atop\bp^{(K)}$, where $\bp^{(K)} \geqslant \boldsymbol{0}$ is supported on $J^{(K)} \subseteq \eqset(\bx^{(K)})$, so the pair $(\bx^{(K)},\bp^{(K)})$ satisfies the KKT conditions and $\bx^{(K)}$ is an optimal solution. If $\bd^{(K)} \neq \boldsymbol{0}$, then $\Delta^{(K)} = +\infty$, so $[\matr{A}\bd^{(K)}]_{i} \leqslant 0$ for every $i \notin \eqset(\bx^{(K)})$ by~\eqref{eq:delta_feas}, and~\eqref{eq:nnls_oc_1} gives the same over $\eqset(\bx^{(K)})$. Hence $\matr{A}\bd^{(K)} \leqslant \boldsymbol{0}$, and the ray $\bx^{(K)} + \Delta\bd^{(K)}$, $\Delta \geqslant 0$, is feasible. Formula~\eqref{eq:obj_fun_after_descent} with $\matr{Q} = \matr{0}$ then gives $V(\bx^{(K)} + \Delta\bd^{(K)}) = V(\bx^{(K)}) - \Delta\normsq{\bd^{(K)}} \to -\infty$ as $\Delta \to +\infty$.
\end{proof}

Theorem~\ref{thm:finite_tI} yields Algorithm~\ref{alg:quad_prog}. The algorithm computes the iterates $\bx^{(0)},\dots,\bx^{(K)}$ of~\eqref{eq:tI_iteration} in at most $2^{m}$ iterations. At termination, the algorithm returns an optimal solution to problem~\eqref{eq:intro_main_problem} along with its corresponding dual solution or certifies that problem~\eqref{eq:intro_main_problem} is unbounded.


\section{Minimal selection descent via the explicit Euler method}
\label{sec:computations_global}
For SPSD $\matr{Q}$, each piece of the slow solution requires the matrix exponential $e^{-\matr{P}^{\perp}_{J}\matr{Q}\tau}$, its integral $\matr{M}(\tau)$, and the first zero of the constraint values and multipliers along it (Section~\ref{subsec:computations_local}). These computations are generally expensive unless $\matr{Q} = t\matr{I}_{n\times n}$, in which case Algorithm~\ref{alg:quad_prog} can be used. To overcome this issue, we propose to perform gradient descent on the slow system~\eqref{eq:slow_ode} with a stepsize that keeps the iterates feasible and maximizes the descent. This leads to minimal selection descent (MSD), Algorithm~\ref{alg:quad_prog_2}. Unlike projected gradient descent (PGD), which steps along $-(\matr{Q}\bx^{(k)} + \bc)$ and then projects onto $C$, MSD projects $-(\matr{Q}\bx^{(k)} + \bc)$ onto the tangent cone $T_{C}(\bx^{(k)})$ and then steps along the direction $\bd^{(k)}$. This order is cheaper for a general polyhedron $C$. Indeed, the projection onto $C$ is problem~\eqref{eq:intro_main_problem} with $\matr{Q} = \matr{I}_{n\times n}$, so PGD solves a problem of the same class at every step. In contrast, the projection onto $T_{C}(\bx^{(k)})$ is the NNLS problem~\eqref{eq:nnls_problem} on the constraints active at $\bx^{(k)}$, for which efficient algorithms are available (Remark~\ref{rem:computing_nnls}).

\subsection{Description of the MSD algorithm}
\label{subsec:descent_algorithm}
Let $\matr{Q} \neq \matr{0}$ be SPSD.\footnote{The case $\matr{Q} = \boldsymbol{0}$ reduces Algorithm~\ref{alg:quad_prog_2} to Algorithm~\ref{alg:quad_prog} with $t = 0$.} At every iteration, the MSD algorithm seeks to maintain feasibility while minimizing the convex parabola
\[
V(\bx^{(k)} + \Delta\bd^{(k)}) = V(\bx^{(k)}) - \Delta\normsq{\bd^{(k)}} + \frac{\Delta^2}{2}\langle\bd^{(k)},\matr{Q}\bd^{(k)}\rangle.
\]
The parabola is minimized at $\Delta_{*}^{(k)} \coloneqq \normsq{\bd^{(k)}}/\langle\bd^{(k)},\matr{Q}\bd^{(k)}\rangle$ (equal to $+\infty$ when $\bd^{(k)} \in \ker\matr{Q}$), so we set $\Delta^{(k)} = \min(\Delta_{\mathrm{feas}}(\bx^{(k)}), \Delta_{*}^{(k)})$ to ensure $\matr{A}\bx^{(k)} \leqslant \bb$ for every $k \in \mathbb{N}$ and maximize the descent step. The size of the decrease depends on the largest eigenvalue of $\matr{Q}$.

\begin{lem}\label{lem:sufficient_decrease}
Let $\lambda_{\max}$ denote the largest eigenvalue of $\matr{Q}$. If $\bd^{(k)} \neq \boldsymbol{0}$ and $\Delta^{(k)} < +\infty$, then
\begin{equation}\label{eq:sufficient_decrease}
V(\bx^{(k+1)}) \leqslant V(\bx^{(k)}) -\min\bigl(\Delta_{\rm{feas}}(\bx^{(k)}),\,1/\lambda_{\max}\bigr)\normsq{\bd^{(k)}}/2.
\end{equation}
\end{lem}
\begin{proof}
If $\Delta^{(k)} = \Delta_{*}^{(k)}$, the parabola gives $V(\bx^{(k+1)}) = V(\bx^{(k)}) - \normtwo{\bd^{(k)}}^{4}/(2\langle\bd^{(k)},\matr{Q}\bd^{(k)}\rangle) \leqslant V(\bx^{(k)}) - \normsq{\bd^{(k)}}/(2\lambda_{\max})$, since $\langle\bd^{(k)},\matr{Q}\bd^{(k)}\rangle \leqslant \lambda_{\max}\normsq{\bd^{(k)}}$. If $\Delta^{(k)} = \Delta_{\mathrm{feas}}(\bx^{(k)}) \leqslant \Delta_{*}^{(k)}$, then $\Delta^{(k)}\langle\bd^{(k)},\matr{Q}\bd^{(k)}\rangle \leqslant \normsq{\bd^{(k)}}$, and the parabola gives $V(\bx^{(k+1)}) \leqslant V(\bx^{(k)}) - \Delta_{\mathrm{feas}}(\bx^{(k)})\normsq{\bd^{(k)}}/2$. Inequality~\eqref{eq:sufficient_decrease} follows.
\end{proof}

When $\Delta^{(k)} = +\infty$, the ray $\bx^{(k)} + \Delta\bd^{(k)}$, $\Delta \geqslant 0$, is feasible and $V$ decreases without bound along it, so the algorithm stops and returns $\bd^{(k)}$. The next proposition confirms this and shows that the iteration detects every unbounded problem, either through that exit or through $V(\bx^{(k)}) \to -\infty$. Theorem~\ref{thm:finite_general}(ii) gives the same certificate for the slow solution.

\begin{prop}
\label{prop:descent_dichotomy}
Let $\bx^{(0)} \in C$ and let $\bx^{(k)}$ denote the $k$-th iterate of Algorithm~\ref{alg:quad_prog_2} with $\varepsilon = 0$ and every $\bp^{(k)}$ chosen so that~\eqref{eq:LI} holds at $\bx^{(k)}$. If problem~\eqref{eq:intro_main_problem} is unbounded below, then either $V(\bx^{(k)}) \to -\infty$ or some iterate has $\bd^{(K)} \neq \boldsymbol{0}$, $\Delta^{(K)} = +\infty$, and
\begin{equation}\label{eq:recession_certificate}
\matr{A}\bd^{(K)} \leqslant \boldsymbol{0},
\qquad
\matr{Q}\bd^{(K)} = \boldsymbol{0},
\qquad
\langle\bc,\bd^{(K)}\rangle = -\normsq{\bd^{(K)}} < 0,
\end{equation}
so $V(\bx^{(K)} + \Delta\bd^{(K)}) \to -\infty$ as $\Delta \to +\infty$ along the feasible ray. If problem~\eqref{eq:intro_main_problem} admits an optimal solution, then no iterate with $\bd^{(k)} \neq \boldsymbol{0}$ has $\Delta^{(k)} = +\infty$.
\end{prop}
\begin{proof}
See Appendix~\ref{app:proof_descent_psd}.
\end{proof}

\begin{rem}[Detecting unboundedness in practice]\label{rem:unbounded_detection}
The exit on $\Delta^{(k)} = +\infty$ does not always occur on unbounded problems. Let $\matr{Q} = \operatorname{diag}(1,0)$, $\bc = (0,-1)$, and $C = \{\bx \in \R^{2} : -1 \leqslant x_{1} \leqslant 1\}$, so that the problem is unbounded below along $(0,1)$. From $\bx^{(0)} = (1,0)$, the iterates yield $\bx^{(k)} = ((-1)^{k}, 2k)$ and $\Delta^{(k)} = 2$, so that $V(\bx^{(k)}) \to -\infty$. In practice one must therefore also monitor $V$.
\end{rem}

\subsection{Convergence of the MSD algorithm}
\label{subsec:descent_convergence}
We now assume that problem~\eqref{eq:intro_main_problem} admits an optimal solution. Let $V^{s}$ denote its optimal value. By Lemma~\ref{lem:sufficient_decrease} and Proposition~\ref{prop:descent_dichotomy}, every step of Algorithm~\ref{alg:quad_prog_2} is well-defined and decreases $V$. We show next that the iterates converge to an optimal solution. In particular, when the multipliers at the limit are positive, the iterates identify its active set after finitely many steps and then converge at a linear rate.

\begin{thm}\label{thm:descent_convergence}
Suppose problem~\eqref{eq:intro_main_problem} admits an optimal solution, let $\bx^{(0)} \in C$, and let $\bx^{(k)}$ denote the $k$-th iterate of Algorithm~\ref{alg:quad_prog_2} with $\varepsilon = 0$ and every $\bp^{(k)}$ chosen so that~\eqref{eq:LI} holds at $\bx^{(k)}$.
\begin{itemize}
\item[(i)] Either $\bd^{(K)} = \boldsymbol{0}$ for some $K \in \N$, in which case $\bx^{(K)}$ is an optimal solution and we set $\bxsol \coloneqq \bx^{(K)}$, or the iterates converge to an optimal solution $\bxsol$.

\item[(ii)] Assume $J(\bxsol) = \eqset(\bxsol)$, i.e., the dual solution $\bpsol$ is positive on every constraint active at $\bxsol$. Then there is $K' \in \N$ such that every $k \geqslant K'$ at which the iteration has not stopped satisfies
\begin{equation}\label{eq:descent_identification}
\eqset(\bx^{(k)}) = J(\bx^{(k)}) = \eqset(\bxsol)
\quad \text{and} \quad
\Delta^{(k)} = \Delta_{*}^{(k)}.
\end{equation}
Let $\bar{\matr{S}} \coloneqq \matr{P}^{\perp}_{\eqset(\bxsol)}\matr{Q}\matr{P}^{\perp}_{\eqset(\bxsol)}$ and, when $\bar{\matr{S}} \neq \matr{0}$, let $\bar{\kappa}$ denote the ratio of its largest to its smallest positive eigenvalue. If $\bar{\matr{S}} = \matr{0}$, then $\bd^{(K')} = \boldsymbol{0}$. Otherwise, every such $k$ satisfies
\begin{equation}\label{eq:descent_rate}
V(\bx^{(k+1)}) - V^{s} \;\leqslant\; \left(\frac{\bar{\kappa} - 1}{\bar{\kappa} + 1}\right)^{2}\bigl(V(\bx^{(k)}) - V^{s}\bigr).
\end{equation}
\end{itemize}
\end{thm}
\begin{proof}
The idea is to show that there are no blocked tails, that is, infinitely many consecutive blocked steps of summable length along which $\normtwo{\bd^{(k)}}$ stays bounded away from zero. Without a blocked tail, Lemma~\ref{lem:sufficient_decrease} and a polyhedral error bound give $V(\bx^{(k)}) \downarrow V^{s}$, and a finite-length argument gives convergence of the iterates. For part~(ii), the condition $J(\bxsol) = \eqset(\bxsol)$ keeps the multipliers positive near $\bxsol$, so once an iterate reaches the face of $\bxsol$, every later step is a steepest descent step along that face, and the Kantorovich inequality gives~\eqref{eq:descent_rate}. The proof is technical and is relegated to Appendix~\ref{app:proofs_descent}.
\end{proof}


\section{Numerical experiments and applications}\label{sec:numerics}
In this section, we test Algorithms~\ref{alg:quad_prog} and~\ref{alg:quad_prog_2} on challenging scientific computing problems. Section~\ref{subsec:num_pos_fem2d} applies Algorithm~\ref{alg:quad_prog} to a positivity-preserving limiter for discontinuous Galerkin schemes, and Section~\ref{subsec:num_pde_oc} applies Algorithm~\ref{alg:quad_prog_2} to high-dimensional optimal control problems with bounds on the control, one for the Poisson equation in 3D and one for the 2D, time-dependent heat equation. All experiments were carried out in MATLAB R2025b on an Apple M3 processor with 16\,GB of RAM, with versions of MOSEK~\cite{aps2019mosek}, Gurobi~\cite{gurobi2026}, and HiGHS~\cite{huangfu2018parallelizing} downloaded on 23 May 2026. We measure times with \texttt{tic} and \texttt{toc}. MOSEK runs the primal and dual feasibility with relative gap tolerances $10^{-10}$, Gurobi uses the barrier convergence tolerance $10^{-14}$ with feasibility and optimality tolerances $10^{-9}$, and HiGHS uses optimality and primal and dual feasibility tolerances $10^{-10}$. Gurobi runs with its default method, which is the barrier method for convex quadratic programs, and HiGHS with its interior-point solver. All our code will be available on GitHub.\footnote{\url{https://github.com/Gabriel-P-Langlois/differential_inclusions_algorithms}}

\subsection{Positivity-preserving limiters for finite element methods}\label{subsec:num_pos_fem2d}
We apply Algorithm~\ref{alg:quad_prog} to a positivity-preserving limiter for numerical schemes~\cite{zhang2011maximum,liu2025optimization}. We consider here the limiter of \citet{liu2025optimization}, who enforce positivity in a semi-implicit discontinuous Galerkin (DG) scheme for Fokker--Planck equations, without losing conservation or high-order accuracy, by a two-stage postprocessing applied after each time step. Stage~1 corrects the cell averages of the DG solution by solving a convex quadratic program. Stage~2 then applies the Zhang--Shu scaling limiter~\cite{zhang2010maximum,zhang2010positivity} in each cell, which removes negative point values without changing the cell averages. Stage~1 is the problem we solve with Algorithm~\ref{alg:quad_prog}.

On a uniform mesh of $n$ square cells, let $\bar{\bm{f}} \in \Rn$ hold the cell averages of the DG solution and let $\mathbf{1} \in \Rn$ denote the vector of ones. Stage~1 solves~\cite[Lemma~1]{liu2025optimization}
\begin{equation}\label{eq:limiter_qp}
\min_{\bx \in \Rn} \frac{1}{2}\normsq{\bx - \bar{\bm{f}}} \ \text{subject to} \ \langle\mathbf{1},\bx\rangle = \langle\mathbf{1},\bar{\bm{f}}\rangle \ \text{and} \ \bx \geqslant l\mathbf{1},
\end{equation}
where $l > 0$ is the lower bound. The equality constraint conserves the total mass, and its coefficients are equal because the cells have equal areas. We pose problem~\eqref{eq:limiter_qp} on the hyperplane $\langle\mathbf{1},\bx\rangle = \langle\mathbf{1},\bar{\bm{f}}\rangle$. Problem~\eqref{eq:limiter_qp} is then problem~\eqref{eq:intro_main_problem} with $\matr{Q} = \matr{I}_{n\times n}$ and $\bc = -\bar{\bm{f}}$, the case $t = 1$ of Section~\ref{subsec:asymptotic_finite_term}.

We start Algorithm~\ref{alg:quad_prog} from the feasible point $\bx^{(0)} = l\mathbf{1} + s\bigl(\max(\bar{\bm{f}},l\mathbf{1}) - l\mathbf{1}\bigr)$ with $s = (\langle\mathbf{1},\bar{\bm{f}}\rangle - nl)/(\langle\mathbf{1},\max(\bar{\bm{f}},l\mathbf{1})\rangle - nl)$. For problem~\eqref{eq:limiter_qp}, the NNLS problem~\eqref{eq:nnls_problem} at each iteration reduces to one scalar equation. From $\bx^{(0)}$, the coordinates above the bound reach it in increasing order of their cell averages, so our implementation computes the breakpoints after sorting these coordinates by cell average.

We compare Algorithm~\ref{alg:quad_prog} with DR, the relaxed Douglas--Rachford splitting method of~\cite[Algorithm~DR]{liu2025optimization}, run with the parameters of~\cite[Remark~2]{liu2025optimization}, which come from~\cite{liu2024simple}. DR stops when $h\normtwo{\by^{(k+1)} - \by^{(k)}} < \epsilon_{\mathrm{tol}}$, where $\by^{(k)}$ is its auxiliary variable and $h$ is the side length of a cell~\cite[Algorithm~DR and Remark~3]{liu2025optimization}. We do not test the interior-point solvers of Section~\ref{subsec:num_pde_oc} on problem~\eqref{eq:limiter_qp} because they cannot use the clipped cell averages as a starting point, which Algorithm~\ref{alg:quad_prog} and DR do, and because problem~\eqref{eq:limiter_qp} has up to $4.2\cdot10^{6}$ variables at every time step. Neither our implementation of Algorithm~\ref{alg:quad_prog} nor that of DR was tuned for speed.

\subsubsection{A synthetic benchmark}
\label{subsubsec:num_pos_synthetic}

Liu et al.~\cite[Section~4.6]{liu2025optimization} time the DR algorithm on problem~\eqref{eq:limiter_qp} with the synthetic data of~\cite[Appendix~A]{liu2024optimization}. On this data, $\bar{\bm{f}}$ holds the values of the function $f$ on $[0,1]^{2}$
\begin{equation}\label{eq:synthetic_data}
f(x,y) = \begin{cases}
-0.25 & \text{if} \ |x - 0.25| \leqslant \delta/4 \ \text{or} \ |x - 0.75| \leqslant \delta/4, \\
\cos^{8}(2\pi x) + 10^{-13} & \text{otherwise},
\end{cases}
\end{equation}
at the points of the grid $\{0, h, \dots, 1 - h\}^{2}$, so that $n = h^{-2}$. We take $l = 10^{-13}$, $\epsilon_{\mathrm{tol}} = 10^{-13}$, and $h = 2^{-7}, 2^{-8}, \dots, 2^{-11}$ from~\cite[Section~4 and Table~3]{liu2025optimization}, and we pose problem~\eqref{eq:limiter_qp} on all $n$ points. Neither paper states the grid points or $\delta$. We take the grid above and $\delta = 0.05$, the fraction of $[0,1]^{2}$ that the two bands cover, which makes between $4.69\%$ and $5.47\%$ of the entries of $\bar{\bm{f}}$ negative, close to the $5\%$ of Liu et al. We solve each instance one hundred times for each solver and we report the mean time of one solve. The KKT conditions of problem~\eqref{eq:limiter_qp} give its minimizer as $\bx_{\mathrm{ref}} = \max(\bar{\bm{f}} + \nu^{*}\mathbf{1}, l\mathbf{1})$, where $\nu^{*} \in \R$ solves $\langle\mathbf{1},\max(\bar{\bm{f}} + \nu\mathbf{1}, l\mathbf{1})\rangle = \langle\mathbf{1},\bar{\bm{f}}\rangle$. We compute $\nu^{*}$ by bisection and measure the accuracy of each solver against $\bx_{\mathrm{ref}}$.

\begin{table}[ht]
\centering
\footnotesize
\setlength{\tabcolsep}{3pt}
\caption{Iterations of one solve, mean time of one solve of problem~\eqref{eq:limiter_qp} over one hundred solves, along with the error $\norminf{\bx - \bx_{\mathrm{ref}}}$ and conservation residual $|\langle\mathbf{1},\bx\rangle - \langle\mathbf{1},\bar{\bm{f}}\rangle|$.}
\label{tab:num_synth}
\begin{tabular}{llcccc}
\toprule
$h$ & solver & iterations & time & $\norminf{\bx - \bx_{\mathrm{ref}}}$ & $|\langle\mathbf{1},\bx\rangle - \langle\mathbf{1},\bar{\bm{f}}\rangle|$ \\
\midrule
$2^{-7}$ & Algorithm~\ref{alg:quad_prog} & $13$ & $1.39\cdot10^{-3}$\,s & $1.55\cdot10^{-15}$ & $7.37\cdot10^{-11}$ \\
 & DR & $17$ & $1.95\cdot10^{-3}$\,s & $4.05\cdot10^{-14}$ & $4.22\cdot10^{-10}$ \\
$2^{-8}$ & Algorithm~\ref{alg:quad_prog} & $25$ & $2.08\cdot10^{-3}$\,s & $6.44\cdot10^{-15}$ & $2.28\cdot10^{-9}$ \\
 & DR & $17$ & $2.04\cdot10^{-3}$\,s & $9.54\cdot10^{-14}$ & $4.33\cdot10^{-10}$ \\
$2^{-9}$ & Algorithm~\ref{alg:quad_prog} & $48$ & $5.99\cdot10^{-3}$\,s & $3.55\cdot10^{-15}$ & $2.76\cdot10^{-10}$ \\
 & DR & $17$ & $6.13\cdot10^{-3}$\,s & $1.78\cdot10^{-14}$ & $1.94\cdot10^{-9}$ \\
$2^{-10}$ & Algorithm~\ref{alg:quad_prog} & $95$ & $2.49\cdot10^{-2}$\,s & $6.01\cdot10^{-14}$ & $3.43\cdot10^{-9}$ \\
 & DR & $17$ & $3.02\cdot10^{-2}$\,s & $5.43\cdot10^{-14}$ & $7.28\cdot10^{-9}$ \\
$2^{-11}$ & Algorithm~\ref{alg:quad_prog} & $190$ & $1.08\cdot10^{-1}$\,s & $7.27\cdot10^{-14}$ & $2.56\cdot10^{-9}$ \\
 & DR & $17$ & $1.35\cdot10^{-1}$\,s & $5.68\cdot10^{-14}$ & $3.80\cdot10^{-8}$ \\
\bottomrule
\end{tabular}
\end{table}

Table~\ref{tab:num_synth} reports the results. On every instance, both solvers return points that satisfy the lower bound, lie within $10^{-13}$ of $\bx_{\mathrm{ref}}$ in the maximum norm, and conserve the total mass to within $4\cdot10^{-8}$. Algorithm~\ref{alg:quad_prog} is faster than DR on four of the five grids, by at most a factor of $1.4$, and slower by less than $3\%$ for $h = 2^{-8}$.

\subsubsection{A discontinuous Galerkin scheme for a Fokker--Planck equation}
\label{subsubsec:num_pos_dg}
We now solve problem~\eqref{eq:limiter_qp} at every time step of a DG computation, in which Algorithm~\ref{alg:quad_prog} and DR each postprocess their own DG solution. Stage~1 runs at every time step, up to $30\,400$ steps per computation, so this test measures how the conservation error of each solver builds up over a long computation. We reproduce the accuracy test of~\cite[Section~4.1]{liu2025optimization}, which solves the linear Fokker--Planck equation
\begin{equation}\label{eq:fokker_planck}
\partial_{\tau}f = \Delta f + \nabla\cdot(f\bv) \ \text{in} \ [1,20]\times\Omega, \qquad \Omega = [-10,10]^{2},
\end{equation}
with exact solution
\begin{equation}\label{eq:fokker_planck_exact}
f(\tau,\bv) = \frac{1}{2\pi(1 - e^{-2\tau})}\exp\left(-\frac{\normsq{\bv}}{2(1 - e^{-2\tau})}\right).
\end{equation}

We use the semi-implicit DG method of~\cite[Section~3.1]{liu2025optimization} with polynomials of degree $k = 2$ and $k = 3$ on uniform grids of $64^{2}$, $128^{2}$, and $256^{2}$ square cells. It treats the diffusion term implicitly, with the non-symmetric interior penalty form and penalty $\sigma/h$, $\sigma = 1$, and the convection term explicitly, with a Lax--Friedrichs flux. After each time step, the two-stage postprocessing solves problem~\eqref{eq:limiter_qp} with $l = 10^{-13}$ for the cell averages and then applies the Zhang--Shu scaling limiter. We use the pairs $(h,\Delta\tau)$ of~\cite[Table~1]{liu2025optimization} with the same $h$ in the penalty as in the stopping rule of DR, $0$ for the tolerance in the Zhang--Shu limiter, $\Delta\tau = 19/119$ for $k = 3$ on the grid of $64^{2}$ cells, and problem~\eqref{eq:limiter_qp} posed on all cells. The basis on each cell is the tensor-product Legendre basis, orthonormal on the reference cell $[-\tfrac12,\tfrac12]^{2}$, and every integral uses the tensor-product Gauss rule with $k+1$ points per direction. The initial data is the $L^{2}$ projection of~\eqref{eq:fokker_planck_exact} at $\tau = 1$, followed by the Zhang--Shu limiter. On the boundary faces, the scheme has no interior penalty terms, as in~\cite[Section~3.1]{liu2025optimization}. We set the convective flux to zero, as the conservation of mass in~\cite[Section~3.1]{liu2025optimization} requires. The scheme thus imposes the zero-flux condition $(\nabla f + f\bv)\cdot\boldsymbol{n} = 0$ weakly, where $\boldsymbol{n}$ is the outer unit normal to $\Omega$. Liu et al.~\cite[Section~4.1]{liu2025optimization} instead impose boundary data from the exact solution~\eqref{eq:fokker_planck_exact}. For $\tau \in [1,20]$, this solution is at most $3.1\cdot10^{-23}$ on the boundary of $\Omega$, and its flux $(\nabla f + f\bv)\cdot\boldsymbol{n}$ there is at most $2.3\cdot10^{-24}$ in absolute value.

\begin{table}[ht]
\centering
\small
\setlength{\tabcolsep}{4pt}
\caption{DG computation of Section~\ref{subsubsec:num_pos_dg} with stage~1 solved by Algorithm~\ref{alg:quad_prog} and DR: drift of the total mass and time of stage~1. The time of stage~1 is summed over the time steps, where the time of each step is the mean over five calls.}
\label{tab:num_dg}
\begin{tabular}{rrllcc}
\toprule
$k$ & $20/h$ & $\Delta\tau$ & solver & mass drift & stage~1 \\
\midrule
2 & 64 & $4\cdot10^{-2}$ & Algorithm~\ref{alg:quad_prog} & $1.25\cdot10^{-13}$ & $1.15\cdot10^{-1}$\,s \\
& & & DR & $1.27\cdot10^{-10}$ & $6.27\cdot10^{-2}$\,s \\
2 & 128 & $10^{-2}$ & Algorithm~\ref{alg:quad_prog} & $5.02\cdot10^{-13}$ & $1.34\cdot10^{0}$\,s \\
& & & DR & $2.20\cdot10^{-10}$ & $1.72\cdot10^{-1}$\,s \\
2 & 256 & $2.5\cdot10^{-3}$ & Algorithm~\ref{alg:quad_prog} & $2.13\cdot10^{-12}$ & $2.69\cdot10^{1}$\,s \\
& & & DR & $2.04\cdot10^{-10}$ & $1.49\cdot10^{0}$\,s \\
\midrule
3 & 64 & $19/119$ & Algorithm~\ref{alg:quad_prog} & $1.29\cdot10^{-14}$ & $1.64\cdot10^{-2}$\,s \\
& & & DR & $1.43\cdot10^{-11}$ & $8.76\cdot10^{-3}$\,s \\
3 & 128 & $10^{-2}$ & Algorithm~\ref{alg:quad_prog} & $5.55\cdot10^{-14}$ & $1.23\cdot10^{0}$\,s \\
& & & DR & $2.10\cdot10^{-10}$ & $1.61\cdot10^{-1}$\,s \\
3 & 256 & $6.25\cdot10^{-4}$ & Algorithm~\ref{alg:quad_prog} & $2.25\cdot10^{-13}$ & $1.12\cdot10^{2}$\,s \\
& & & DR & $2.35\cdot10^{-10}$ & $6.90\cdot10^{0}$\,s \\
\bottomrule
\end{tabular}
\end{table}

Both solvers keep every cell average at or above $l$. With Algorithm~\ref{alg:quad_prog}, the errors $L^{2}_{h}$ and $L^{\infty}_{h}$ at $\tau = 20$ agree with~\cite[Table~1]{liu2025optimization} in all four printed digits for both degrees and all three grids, and the convergence rates computed from them agree with those of~\cite[Table~1]{liu2025optimization} to within $10^{-3}$. With DR, the errors also agree in all four digits, except for $k = 3$ on the grid of $256^{2}$ cells, where DR gives $7.556\cdot10^{-8}$ and $8.083\cdot10^{-8}$ against $7.555\cdot10^{-8}$ and $8.079\cdot10^{-8}$.

Algorithm~\ref{alg:quad_prog} keeps the drift $\max_{j}|\langle f_{h}^{j},1\rangle - \langle f_{h}^{0},1\rangle|$ of the total mass over the time steps $j$ below $2.2\cdot10^{-12}$ in every computation, and DR below $2.4\cdot10^{-10}$ (Table~\ref{tab:num_dg}). Across the six computations, the drift of DR is $96$ to $3800$ times that of Algorithm~\ref{alg:quad_prog}. The exact solution~\eqref{eq:fokker_planck_exact} has total mass $1$ on $\R^{2}$, which sets the scale of the drift. This drift comes from stage~1, because DR returns $\max(\by^{(k+1)},l\mathbf{1})$~\cite[Algorithm~DR]{liu2025optimization}, which satisfies the lower bound exactly and the conservation constraint only approximately. Stage~1 is faster with DR than with Algorithm~\ref{alg:quad_prog}, by factors of $1.8$, $7.8$, and $18.1$ for $k = 2$ on the three grids and of $1.9$, $7.7$, and $16.3$ for $k = 3$. In Section~\ref{subsubsec:num_pos_synthetic}, DR takes $17$ iterations per solve. Here, after the first time step, DR takes one iteration per step on the grids of $128^{2}$ and $256^{2}$ cells, and Algorithm~\ref{alg:quad_prog} takes two or three. On the grids of $64^{2}$ cells, DR takes at most $22$ iterations per step and Algorithm~\ref{alg:quad_prog} at most $63$. Our nonoptimized implementation of Algorithm~\ref{alg:quad_prog}, despite being slower than DR, conserves the total mass $96$ to $3800$ times more accurately than DR.

\subsection{PDE-constrained optimization problems}
\label{subsec:num_pde_oc}
Pearson and Gondzio~\cite{pearson2017fast} solve a challenging PDE-constrained optimization problem (PDECO) with bounds on the control. Other PDECOs with different degrees of complexity have also been considered in \cite{benner2016block, BENNER201626, BENNER2020, Antil2023, Antil2025, Antil2025b}. Here, we consider two of the examples from \cite{pearson2017fast}, in which we seek a state $y$ and a control $u$ on a domain $\Omega$ that minimize
\begin{equation}\label{eq:oc_continuous}
\frac{1}{2}\int_{\Omega}(y - y_{d})^{2} + \frac{\beta}{2}\int_{\Omega}u^{2}
\quad \text{subject to} \quad \mathcal{L}y = u \ \text{in} \ \Omega, \quad y = 0 \ \text{on} \ \partial\Omega, \quad u_{a} \leqslant u \leqslant u_{b},
\end{equation}
where $\mathcal{L}$ is a differential operator, $y_{d}$ is the desired state, $\beta > 0$ weights the control, and $u_{a} < u_{b}$ are the bounds. In Section~\ref{subsubsec:num_oc_poisson3d}, $\mathcal{L} = -\Delta$ on a three-dimensional domain. In Section~\ref{subsubsec:num_oc_heat2d}, $\mathcal{L}$ is the heat operator on a two-dimensional domain, and the integrals in~\eqref{eq:oc_continuous} run over space and time. We run each solver five times on each instance, after one discarded call, and report the mean time of one solve. This time covers all the work the solver needs to return the state and the control, and excludes the accuracy measures below.

We discretize problem~\eqref{eq:oc_continuous} before we optimize, as~\cite{pearson2017fast} do, with $Q1$ finite elements for the state, the control, and the adjoint, and consistent mass matrices~\cite[Sections~4.1 and~5]{pearson2017fast}. With $\by$ and $\bu \in \R^{N}$ the vectors of nodal values and $\by_{d}$ those of the desired state, the discrete problem is
\begin{equation}\label{eq:oc_discrete}
\min_{\by,\bu \in \R^{N}} \ \frac{1}{2}\langle\by - \by_{d},\matr{M}_{y}(\by - \by_{d})\rangle + \frac{\beta}{2}\langle\bu,\matr{M}_{u}\bu\rangle
\quad \text{subject to} \quad \matr{K}\by = \matr{B}\bu \ \ \text{and} \ \ u_{a}\mathbf{1} \leqslant \bu \leqslant u_{b}\mathbf{1},
\end{equation}
where $\matr{K}$ discretizes $\mathcal{L}$ together with the conditions on $y$ and is invertible, $\matr{B}\bu$ discretizes the right-hand side $u$ of the state equation, and the mass matrices $\matr{M}_{y}$ and $\matr{M}_{u}$ discretize the two integrals in~\eqref{eq:oc_continuous} and are SPD. Since the control acts on all of $\Omega$, the matrix $\matr{B}$ is built from the mass matrix of $\Omega$. MOSEK, Gurobi, and HiGHS solve problem~\eqref{eq:oc_discrete} in the $2N$ variables $\bx = (\by,\bu)$, with $N$ equality constraints and $2N$ bounds. Before each call, we divide the objective by $\norminf{\matr{M}_{y}\by_{d}}$, which leaves the minimizer unchanged. We score the point each solver returns, whatever status it reports.

Algorithm~\ref{alg:quad_prog_2} solves a smaller, equivalent problem. Since $\matr{K}$ is invertible, the state equation gives $\by = \matr{G}\bu$ with $\matr{G} = \matr{K}^{-1}\matr{B}$, and eliminating $\by$ from~\eqref{eq:oc_discrete} leaves
\begin{equation}\label{eq:oc_reduced}
\min_{\bu \in \R^{N}} \ \frac{1}{2}\langle\bu,\matr{Q}_{\mathrm{red}}\bu\rangle + \langle\bc_{\mathrm{red}},\bu\rangle
\quad \text{subject to} \quad u_{a}\mathbf{1} \leqslant \bu \leqslant u_{b}\mathbf{1},
\end{equation}
where $\matr{Q}_{\mathrm{red}} = \matr{G}^{\top}\matr{M}_{y}\matr{G} + \beta\matr{M}_{u}$ and $\bc_{\mathrm{red}} = -\matr{G}^{\top}\matr{M}_{y}\by_{d}$. Problem~\eqref{eq:oc_reduced} is problem~\eqref{eq:intro_main_problem} with $\matr{Q} = \matr{Q}_{\mathrm{red}}$, $\bc = \bc_{\mathrm{red}}$, $\matr{A} = [\matr{I};-\matr{I}]$, and $\bb = (u_{b}\mathbf{1};-u_{a}\mathbf{1})$. The matrix $\matr{Q}_{\mathrm{red}}$ is SPD because $\matr{M}_{u}$ is and $\beta > 0$. Since the rows of $\matr{A}$ are signed unit vectors and $u_{a} < u_{b}$, at most one bound on each coordinate $u_{i}$ is active at a feasible point. The rows of every active set are therefore linearly independent, and Theorem~\ref{thm:descent_convergence} applies to problem~\eqref{eq:oc_reduced}. We never form the dense matrix $\matr{Q}_{\mathrm{red}}$. Each iteration of Algorithm~\ref{alg:quad_prog_2} needs one product $\matr{Q}_{\mathrm{red}}\bd^{(k)}$, which gives the exact step length $\Delta^{(k)}_{*}$ and, since $\bu^{(k+1)} = \bu^{(k)} + \Delta^{(k)}\bd^{(k)}$, also updates the gradient. We apply $\matr{Q}_{\mathrm{red}}$ as an operator, at the cost given in each subsection below. On a box, the NNLS problem~\eqref{eq:nnls_problem} splits into $N$ scalar problems with closed-form solutions. We start from the midpoint $\bu^{(0)} = \tfrac{1}{2}(u_{a}+u_{b})\mathbf{1}$ and stop when $\normtwo{\bd^{(k)}} \leqslant \varepsilon$ with $\varepsilon = 10^{-8}$. Before solving, we divide $\matr{Q}_{\mathrm{red}}$ and $\bc_{\mathrm{red}}$ by $\norminf{\bc_{\mathrm{red}}}$, which leaves the minimizer unchanged and makes $\varepsilon$ relative to the data.

We measure accuracy in two ways. First, we build a reference solution. We fix the controls in the final active set of Algorithm~\ref{alg:quad_prog_2} at their bounds, solve for the other controls in~\eqref{eq:oc_reduced} by conjugate gradients to relative residual $10^{-14}$, and accept the result only if every free control lies strictly inside its bounds and the bound multipliers it implies are nonnegative. For every solver, we report the gap $V(\bu) - V(\bu_{\mathrm{ref}})$, where $V$ is the objective of~\eqref{eq:oc_reduced}. Second, we compute the optimality residuals following~\cite{qpbenchmark}. Specifically, we write problem~\eqref{eq:oc_discrete} as the minimization of $\tfrac{1}{2}\langle\bx,\matr{Q}\bx\rangle + \langle\bc,\bx\rangle$ subject to $\matr{E}\bx = \boldsymbol{0}$ and $\matr{A}\bx \leqslant \bb$, where $\matr{Q} = \operatorname{diag}(\matr{M}_{y},\beta\matr{M}_{u})$, $\bc = (-\matr{M}_{y}\by_{d};\boldsymbol{0})$, $\matr{E} = [\matr{K},-\matr{B}]$, $\matr{A} = [\matr{0},\matr{I};\matr{0},-\matr{I}]$, and $\bb = (u_{b}\mathbf{1};-u_{a}\mathbf{1})$, and let $\bl$ and $\bp \geqslant \boldsymbol{0}$ be the multipliers of the two constraints. The primal residual, the dual residual, and the duality gap of a point $(\bx,\bl,\bp)$ are
\begin{subequations}\label{eq:oc_residuals}
\begin{align}
r_{p} &= \max\bigl(\norminf{\matr{E}\bx},\ \norminf{\max(0,\matr{A}\bx - \bb)}\bigr), \label{eq:oc_res_p} \\
r_{d} &= \norminf{\matr{Q}\bx + \bc + \matr{E}^{\top}\bl + \matr{A}^{\top}\bp}, \label{eq:oc_res_d} \\
r_{g} &= \bigl|\langle\bx,\matr{Q}\bx\rangle + \langle\bc,\bx\rangle + \langle\bb,\bp\rangle\bigr|. \label{eq:oc_res_g}
\end{align}
\end{subequations}
Since the data of~\eqref{eq:oc_discrete} carry a factor $\beta h^{3}$ in Section~\ref{subsubsec:num_oc_poisson3d} and $\beta\Delta\tau\,h^{2}$ in Section~\ref{subsubsec:num_oc_heat2d}, we divide each residual by the size of the terms that compose it. Finally, each solver is scored from its own multipliers, and Algorithm~\ref{alg:quad_prog_2} also from its adjoint state.

\subsubsection{Control of the Poisson equation in 3D}
\label{subsubsec:num_oc_poisson3d}
The first example, from~\cite[Section~5]{pearson2017fast}, is problem~\eqref{eq:oc_continuous} with $\mathcal{L} = -\Delta$ on $\Omega = [0,1]^{3}$ and desired state $y_{d}(x) = \exp(-64\normsq{x - (1/2,1/2,1/2)})$. We impose $y = 0$ on $\partial\Omega$ and follow the conventions of IFISS~\cite{ers14}. The grid is uniform with spacing $h$ and $1/h - 1$ interior nodes per direction, so the state and the control each have $N = (1/h-1)^{3}$ degrees of freedom, and $\by_{d}$ holds the nodal values of $y_{d}$. The matrix $\matr{K}$ is the $Q1$ stiffness matrix of the negative Laplacian, and $\matr{B} = \matr{M}_{y} = \matr{M}_{u} = \matr{M}$ is the consistent $Q1$ mass matrix. We use the pairs $(\beta,u_{a},u_{b}) = (10^{-4},0,20)$ and $(10^{-2},0,1)$ and the meshes $h = 2^{-2},\dots,2^{-5}$ of~\cite[Table~6]{pearson2017fast}.

We factor $\matr{K}$ once per mesh by a sparse Cholesky factorization under a nested-dissection ordering. One application of $\matr{Q}_{\mathrm{red}} = \matr{M}\matr{K}^{-1}\matr{M}\matr{K}^{-1}\matr{M} + \beta\matr{M}$ then costs four products with $\matr{M}$ and two solves with $\matr{K}$, each a pair of triangular solves.

\begin{table}[!htb]
\centering
\footnotesize
\setlength{\tabcolsep}{4pt}
\begin{minipage}[t]{0.58\linewidth}
\centering
\caption{Poisson control of Section~\ref{subsubsec:num_oc_poisson3d}. Mean time in seconds of one solve of problem~\eqref{eq:oc_discrete}, over five solves, for both pairs $(\beta,u_{a},u_{b})$, and in parentheses the time divided by that of the fastest solver on the same mesh.}
\label{tab:num_oc3d}
\begin{tabular}{rrcccc}
\toprule
$h$ & $N$ & Algorithm~\ref{alg:quad_prog_2} & MOSEK & Gurobi & HiGHS \\
\midrule
\multicolumn{6}{l}{$\beta = 10^{-4}$, $0 \leqslant u \leqslant 20$} \\
$2^{-2}$ & $27$ & 0.002 (1.0) & 0.002 (1.0) & 0.003 (1.3) & 0.005 (2.3) \\
$2^{-3}$ & $343$ & 0.010 (1.0) & 0.077 (7.8) & 0.055 (5.6) & 0.027 (2.7) \\
$2^{-4}$ & $3\,375$ & 0.243 (1.0) & 1.75 (7.2) & 2.27 (9.4) & 0.569 (2.3) \\
$2^{-5}$ & $29\,791$ & 22.4 (1.0) & 119 (5.3) & 261 (11.6) & 23.3 (1.0) \\
\midrule
\multicolumn{6}{l}{$\beta = 10^{-2}$, $0 \leqslant u \leqslant 1$} \\
$2^{-2}$ & $27$ & 0.001 (1.0) & 0.001 (1.6) & 0.002 (2.8) & 0.004 (6.3) \\
$2^{-3}$ & $343$ & 0.008 (1.0) & 0.070 (8.5) & 0.054 (6.6) & 0.034 (4.2) \\
$2^{-4}$ & $3\,375$ & 0.155 (1.0) & 1.77 (11.4) & 3.45 (22.2) & 0.833 (5.4) \\
$2^{-5}$ & $29\,791$ & 3.86 (1.0) & 120 (31.0) & 176 (45.5) & 30.3 (7.9) \\
\bottomrule
\end{tabular}
\end{minipage}
\hfill
\begin{minipage}[t]{0.38\linewidth}
\centering
\caption{Poisson control of Section~\ref{subsubsec:num_oc_poisson3d}. Largest objective gap $V(\bu) - V(\bu_{\mathrm{ref}})$ and relative residual $\max(r_{p},r_{d},r_{g})$ each solver returns.}
\label{tab:num_oc3d_acc}
\begin{tabular}{lcc}
\toprule
solver & gap & residual \\
\midrule
Algorithm~\ref{alg:quad_prog_2} & $5.0\cdot10^{-18}$ & $2.8\cdot10^{-9}$ \\
MOSEK & $6.5\cdot10^{-11}$ & $6.3\cdot10^{-7}$ \\
Gurobi & $6.5\cdot10^{-19}$ & $2.2\cdot10^{-11}$ \\
HiGHS & $1.9\cdot10^{-13}$ & $2.2\cdot10^{-10}$ \\
\bottomrule
\end{tabular}
\end{minipage}
\end{table}

Table~\ref{tab:num_oc3d} reports the times and Table~\ref{tab:num_oc3d_acc} compares the accuracy. Algorithm~\ref{alg:quad_prog_2} is the fastest solver on every mesh of both pairs. Over both pairs and all four meshes, the objective gap of Algorithm~\ref{alg:quad_prog_2} stays at or below $5.0\cdot10^{-18}$, smaller than those of HiGHS and MOSEK. Its largest relative residual, $2.8\cdot10^{-9}$, lies between those of HiGHS and Gurobi and that of MOSEK. Thus Algorithm~\ref{alg:quad_prog_2} performs well and with high accuracy compared to highly optimized quadratic programming solvers.

\subsubsection{Control of the time-dependent heat equation in 2D}
\label{subsubsec:num_oc_heat2d}
The second example is the time-dependent heat control problem of~\cite[Sections~3.3 and~5]{pearson2017fast}, in which $y_{\tau} - \Delta y = u$ holds on $\Omega \times (0,T]$ and the integrals in~\eqref{eq:oc_continuous} run over $\Omega \times (0,T]$. Pearson and Gondzio state neither the domain nor the initial and boundary data for this example. We take $\Omega = [0,1]^{2}$, $T = 1$, $y = 0$ on $\partial\Omega$, and $y = 0$ at $\tau = 0$, as in the heat control examples of~\cite{stoll2010all,pearson2012regularization}, and the desired state $y_{d}(x,\tau) = \tfrac{1}{2}\tau\sin(2\pi x_{1}x_{2})$, the two-dimensional form of the desired state of~\cite[Section~6.2]{stoll2010all}. This desired state changes sign on $\Omega$, so both bounds can be active.

We discretize in space as in Section~\ref{subsubsec:num_oc_poisson3d} and in time by the backward Euler method with step $\Delta\tau = 0.02$, so $T/\Delta\tau = 50$ steps reach $T$, and we apply the trapezoidal rule to the integrals in time~\cite[Section~3.3]{pearson2017fast}. With the time levels stacked, the state and the control each have $N = 50(1/h-1)^{2}$ degrees of freedom. The matrix $\matr{K}$ is block lower bidiagonal, with one backward Euler step on the diagonal and the negated spatial mass matrix below it. The matrix $\matr{B}$ is $\Delta\tau$ times the block diagonal spatial mass matrix, and $\matr{M}_{y} = \matr{M}_{u}$ carries the trapezoidal weights. We use the pairs $(\beta,u_{a},u_{b}) = (10^{-2},0,1)$ and $(10^{-4},0,30)$ of~\cite[Table~7]{pearson2017fast} and the meshes $h = 2^{-2},\dots,2^{-6}$.

On a uniform grid, the one-dimensional $Q1$ stiffness and mass matrices share the eigenvectors of the discrete sine transform, so the two-dimensional sine basis diagonalizes the spatial stiffness, the spatial mass, and one backward Euler step. In that basis, $\matr{K}^{-1}$ and $\matr{K}^{-\top}$ reduce to one scalar recurrence per spatial mode, forward and backward in time, and $\matr{M}_{y}$ and $\matr{M}_{u}$ are diagonal. One product with $\matr{Q}_{\mathrm{red}}$ then costs two sine transforms of every time level, which we compute as dense matrix products.

\begin{table}[!htb]
\centering
\footnotesize
\setlength{\tabcolsep}{4pt}
\begin{minipage}[t]{0.58\linewidth}
\centering
\caption{Heat control of Section~\ref{subsubsec:num_oc_heat2d} with $\Delta\tau = 0.02$, with columns as in Table~\ref{tab:num_oc3d}. The table has no Gurobi entry for $h = 2^{-6}$ because Gurobi did not converge within four hours for that problem.}
\label{tab:num_oc2d}
\begin{tabular}{rrcccc}
\toprule
$h$ & $N$ & Algorithm~\ref{alg:quad_prog_2} & MOSEK & Gurobi & HiGHS \\
\midrule
\multicolumn{6}{l}{$\beta = 10^{-2}$, $0 \leqslant u \leqslant 1$} \\
$2^{-2}$ & $450$ & 0.007 (1.0) & 0.020 (3.0) & 0.010 (1.5) & 0.017 (2.5) \\
$2^{-3}$ & $2\,450$ & 0.058 (1.0) & 0.310 (5.4) & 0.144 (2.5) & 0.173 (3.0) \\
$2^{-4}$ & $11\,250$ & 0.565 (1.0) & 1.96 (3.5) & 1.76 (3.1) & 1.60 (2.8) \\
$2^{-5}$ & $48\,050$ & 6.63 (1.0) & 27.1 (4.1) & 25.2 (3.8) & 17.2 (2.6) \\
$2^{-6}$ & $198\,450$ & 148 (1.0) & 560 (3.8) & --- & 349 (2.4) \\
\midrule
\multicolumn{6}{l}{$\beta = 10^{-4}$, $0 \leqslant u \leqslant 30$} \\
$2^{-2}$ & $450$ & 0.022 (2.0) & 0.023 (2.1) & 0.011 (1.0) & 0.018 (1.7) \\
$2^{-3}$ & $2\,450$ & 0.251 (1.7) & 0.318 (2.2) & 0.147 (1.0) & 0.162 (1.1) \\
$2^{-4}$ & $11\,250$ & 1.55 (1.1) & 3.91 (2.8) & 1.95 (1.4) & 1.39 (1.0) \\
$2^{-5}$ & $48\,050$ & 14.1 (1.0) & 26.0 (1.8) & 28.6 (2.0) & 17.7 (1.3) \\
$2^{-6}$ & $198\,450$ & 300 (1.0) & 653 (2.2) & --- & 334 (1.1) \\
\bottomrule
\end{tabular}
\end{minipage}
\hfill
\begin{minipage}[t]{0.38\linewidth}
\centering
\caption{Heat control of Section~\ref{subsubsec:num_oc_heat2d}. Largest objective gap $V(\bu) - V(\bu_{\mathrm{ref}})$ and relative residual $\max(r_{p},r_{d},r_{g})$ each solver returns.}
\label{tab:num_oc2d_acc}
\begin{tabular}{lcc}
\toprule
solver & gap & residual \\
\midrule
Algorithm~\ref{alg:quad_prog_2} & $9.6\cdot10^{-16}$ & $2.6\cdot10^{-10}$ \\
MOSEK & $8.9\cdot10^{-8}$ & $2.4\cdot10^{-4}$ \\
Gurobi & $3.7\cdot10^{-18}$ & $4.4\cdot10^{-11}$ \\
HiGHS & $2.4\cdot10^{-14}$ & $3.4\cdot10^{-10}$ \\
\bottomrule
\end{tabular}
\end{minipage}
\end{table}

Table~\ref{tab:num_oc2d} reports the times. For $\beta = 10^{-2}$, Algorithm~\ref{alg:quad_prog_2} is the fastest solver on every mesh. For $\beta = 10^{-4}$, it is the fastest on the two finest meshes, and Gurobi or HiGHS is faster on the three coarsest, by a factor of at most $2.0$. From $h = 2^{-5}$ to $h = 2^{-6}$, $N$ grows by a factor of $4.13$ and the time of every solver by a factor of $19$ to $25$. Over both pairs and all five meshes, the objective gap of Algorithm~\ref{alg:quad_prog_2} stays at or below $9.6\cdot10^{-16}$, again smaller than those of HiGHS and MOSEK. Its largest relative residual, $2.6\cdot10^{-10}$, is also smaller than theirs (Table~\ref{tab:num_oc2d_acc}). MOSEK loses accuracy as the mesh is refined. For $\beta = 10^{-4}$ on the finest mesh, its gap is $8.9\cdot10^{-8}$, its largest relative residual $2.4\cdot10^{-4}$, and its control differs from the reference by $1.6\cdot10^{-1}$ in the infinity norm, against a bound $u \leqslant 30$. HiGHS reports an optimal status on every mesh, yet its control for $\beta = 10^{-4}$ at $h = 2^{-4}$ differs from the reference by $7.3\cdot10^{-4}$ in the infinity norm. Again, Algorithm~\ref{alg:quad_prog_2} performs well and with high accuracy compared to highly optimized quadratic programming solvers.

\section{Conclusion and future work}\label{sec:discussion}
In this work, we solved the gradient flow of linear and convex quadratic programs exactly. By the minimal selection principle, the slow solution from every feasible point is piecewise explicit, and one NNLS problem gives the initial direction of each piece. A piece ends when a constraint becomes active or a multiplier changes sign, and the norm of the velocity decreases strictly at every discontinuity. When strict complementarity fails, a recursion of cone projections selects the active set of the next piece. We proved that the slow solution has finitely many pieces for every SPSD matrix $\matr{Q}$ and that its final piece yields an optimal solution or a certificate of unboundedness. For $\matr{Q} = t\matr{I}_{n\times n}$, every piece is a ray with a closed-form breakpoint and no active set occurs twice, which gives the exact Algorithm~\ref{alg:quad_prog}. For general SPSD $\matr{Q}$, we analyzed minimal selection descent (Algorithm~\ref{alg:quad_prog_2}), the explicit Euler discretization of the slow system, whose steps each solve one NNLS problem over the active constraints. Its iterates converge to an optimal solution when one exists and otherwise certify unboundedness or drive the objective to $-\infty$. When the multipliers at the limit are positive, the iterates identify the optimal active set after finitely many steps and then converge at a linear rate. In our experiments, Algorithm~\ref{alg:quad_prog} conserved the total mass in a positivity-preserving limiter more accurately than a state-of-the-art Douglas--Rachford splitting method, and minimal selection descent reached the optimal value of PDE-constrained control problems to machine precision in times competitive with commercial and open-source solvers.

Several directions remain open. First, we expect that our solution method can be adapted to other differential inclusions: second-order gradient inclusions, including the ODE associated with Nesterov's accelerated gradient method~\cite{su2016differential}; inclusions arising from saddle-point problems and from mirror descent; and some gradient inclusions posed in Hilbert spaces. Second, we expect that minimal selection descent can be extended to general convex optimization problems, and our convergence analysis to other discretizations of the slow system. Third, we expect that differential inclusions can yield solution methods for second-order cone programs. Fourth, we expect that implementations of our algorithms on graphics processing units (GPUs) can scale to large linear and quadratic programs and be benchmarked against existing solvers at that scale. Finally, the experiments of Section~\ref{sec:numerics} suggest applications to positivity-preserving variational schemes for PDEs and to PDE-constrained optimization. Some PDE-constrained optimization problems can be posed as second-order cone programs, and their continuous formulations lead to the gradient inclusions in Hilbert spaces mentioned above.


\section*{Conflicts of interest}
The authors declare they have no conflicts of interest.

\section*{Declaration of generative AI and AI-assisted technologies in the writing process}
During the preparation of this work, the authors used Claude Code (Anthropic) to assist with proof development, numerical implementation, and drafting. The authors reviewed and edited the content as needed and take full responsibility for the content of the publication.

\appendix
\section{Mathematical background} \label{app:A}
We list here some definitions and results from convex and functional analysis used in this work. For comprehensive references, see~\cite{rockafellar1997convex,rockafellar2009variational,aubin2012differential,hiriart2013convexI}.

\begin{defn}[Convex sets] \label{def:convex_sets}
    A set $C\subset \Rn$ is convex if for every pair $(\bx_{1},\bx_{2}) \in C \times C$ and every scalar $\lambda \in (0,1)$, the vector $\lambda\bx_{1}+(1-\lambda)\bx_{2}$ is contained in $C$.
\end{defn}
\begin{defn}[Closed convex polyhedra]\label{def:convex_polyhedron}
    A set $C \subset \Rn$ is a closed convex polyhedron if it is the intersection of a finite number of half-spaces:
    \[
    C = \left\{\bx \in \Rn : \langle\bu_{i},\bx\rangle \leqslant r_{i} \ \text{for every $i \in [m]$}  \right\}
    \]
    with $\{\bu_{1},\dots,\bu_{m}\} \subset \Rn$ and $\{r_{1},\dots,r_{m}\}\subset \R$.
\end{defn}
\begin{defn}[Convex cones]\label{def:convex_cones}
    A set $K \subset \Rn$ is a cone if $\boldsymbol{0} \in K$ and $\lambda\bx \in K$ for all $\bx \in K$ and $\lambda > 0$. A cone $K$ is convex if it contains the point $\sum_{i=1}^{m} \eta_{i}\bx_{i}$ whenever $\bx_{i} \in K$ and $\eta_{i} \geqslant 0$ for every $i \in [m]$. 
\end{defn}
\begin{defn}[Polyhedral cones]\label{def:polyhedral_cones}
	The conical hull of some finite set $\{\bx_{1},\dots,\bx_{m}\} \subset \Rn$ is
	\[
	 	\cone\{\bx_{1},\dots,\bx_{m}\} \coloneqq \left\{\sum_{i=1}^{m}\eta_{i}\bx_{i} : \eta_{i} \geqslant 0 \, \text{for every}\ i\in [m]\right\}.
	\]
	A cone is polyhedral if $K = \cone\{\bx_{1},\dots,\bx_{m}\}$ for some finite set $\{\bx_{1},\dots,\bx_{m}\} \subset \Rn$.
\end{defn}
\begin{defn}[Proper functions] \label{def:prop_f}
    A function $f$ defined over $\Rn$ is proper if its domain $\dom f \coloneqq \left\{ \bx \in \Rn : f(\bx) < +\infty\right\}$ is nonempty and $f(\bx) > -\infty$ for every $\bx \in \dom f$.
\end{defn}
\begin{defn}[Lower semicontinuity] \label{def:lsc_f}
    A proper function $f\colon\Rn\to\R\cup\{+\infty\}$ is lower semicontinuous at $\bx \in \dom f$ if for every sequence $\left\{ \bx_k\right\} _{k=1}^{+\infty} \subset \Rn$ converging to $\bx$, $\liminf_{k\to+\infty}f(\bx_k)\geqslant f(\bx)$. We say $f$ is lower semicontinuous if it is lower semicontinuous at every $\bx \in \dom f$.
\end{defn}
\begin{defn}[Convex functions and the space $\gmRn$] \label{def:convex} 
A proper function $f\colon\Rn\to\R\cup\{+\infty\}$ is convex if $\dom f$ is convex and if for every pair $(\bx_{1},\bx_{2}) \in \dom f \times \dom f$ and every $\lambda\in[0,1]$,
\[
f(\lambda\bx_{1}+(1-\lambda)\bx_{2})\leqslant\lambda f(\bx_{1})+(1-\lambda)f(\bx_{2}).
\]
It is \textit{strictly convex} if the inequality above is strict whenever $\bx_{1}\neq\bx_{2}$ and $\lambda \in (0,1)$, and it is \textit{$t$-strongly convex} with $t>0$ if for every pair $(\bx_{1},\bx_{2}) \in \dom f \times \dom f$ and every $\lambda\in[0,1]$, 
\[
f(\lambda\bx_{1}+(1-\lambda)\bx_{2})\leqslant\lambda f(\bx_{1})+(1-\lambda)f(\bx_{2})-\frac{t}{2}\lambda(1-\lambda)\normsq{\bx_{1}-\bx_{2}}.
\]
The space of proper, lower semicontinuous, and convex functions over $\Rn$ is denoted by $\gmRn$.
\end{defn}
\begin{defn}[Subdifferentiability]\label{def:subgrad}
A function $f \in \gmRn$ is subdifferentiable at $\bx \in \Rn$ if there exists $\bp \in \Rn$ such that
\begin{equation}\label{eq:conv_subdiff_char}
f(\by) \geqslant f(\bx) + \left\langle\bp,\by-\bx \right\rangle \ \text{for every $\by \in \dom f$.}
\end{equation}
In this case, $\bp$ is called a subgradient of the function $f$ at $\bx$. The set of subgradients at $\bx \in \Rn$ is called the subdifferential of $f$ at $\bx$ and is denoted by $\partial f(\bx)$. When nonempty, $\partial f(\bx)$ is a closed convex set. We note that $f$ has a unique subgradient at $\bx$ if and only if $f$ is differentiable at $\bx$ and $\partial f(\bx) = \{\nabla f(\bx)\}$~\cite[Theorem~25.1]{rockafellar1997convex}. Finally, subdifferentials satisfy the following calculus rules.
\begin{itemize}
    \item[(i)] Let $f_{1},f_{2} \in \gmRn$ and assume $\ri \, \dom f_{1} \cap \ri \, \dom f_{2} \neq \emptyset$. Then $f_{1} + f_{2} \in \gmRn$ and $\partial(f_{1} + f_{2})(\bx) = \partial f_{1}(\bx) + \partial f_{2}(\bx)$ for every $\bx \in \dom f_{1} \cap \dom f_{2}$.

    \item[(ii)] Let $f \in \gmRn$, $g \in \gmRm$ and assume $\image{\matr{A}} \cap \ri \, \dom g \neq \emptyset$. Then the function $\bx \mapsto g(\matr{A}\bx)$ is in $\gmRn$ and $\partial(g \circ \matr{A})(\bx) = \matr{A}^{\top}\partial g(\matr{A}\bx)$ for every $\bx \in \Rn$ satisfying $\matr{A}\bx \in \dom g$.
    
    \item[(iii)] Let $f \in \gmRn$, $g \in \gmRm$ and assume $\ri \, \dom f \cap \ri \, \dom (g \circ \matr{A}) \neq \emptyset$ and $\image{\matr{A}} \cap \ri \, \dom g \neq \emptyset$. Then $\partial(f + g\circ\matr{A})(\bx) = \partial f(\bx) + \matr{A}^{\top}\partial g(\matr{A}\bx)$ for every $\bx \in \dom f \cap \dom (g\circ\matr{A})$.
\end{itemize}
See~\cite[Theorem~23.8]{rockafellar1997convex} for the proof of (i) and~\cite[Theorem~23.9]{rockafellar1997convex} for the proof of (ii). The proof of (iii) follows by applying (i) with $f_1 = f$ and $f_2 = g \circ \matr{A}$, then (ii).
\end{defn}
\begin{defn}[Monotone and maximal monotone mappings]\label{def:mmm}
Let $F$ denote a set-valued mapping from $\Rn$ to $\Rn$ with \textit{graph} $\{(\bx,\bp) \in \Rn \times \Rn \colon \bp \in F(\bx)\}$. We say $F$ is \textit{monotone} if
\begin{equation}\label{eq:monotone_inequality}
\langle\bp_{1}-\bp_{2},\bx_{1}-\bx_{2}\rangle \geqslant 0
\end{equation}
for every $\bx_{1},\bx_{2} \in \Rn$ and every $\bp_{1} \in F(\bx_{1})$, $\bp_{2} \in F(\bx_{2})$. It is \textit{maximal} if no other set-valued mapping strictly contains the graph of $F$. If $f \in \gmRn$, then its subdifferential is a maximal monotone mapping~\cite[Theorem 12.17]{rockafellar2009variational}. 
\end{defn}
\begin{defn}[Characteristic functions]\label{def:char_fun}
The characteristic function of a subset $C \subset \Rn$ is defined as
\[
\chi_{C}(\bx) \coloneqq \begin{cases} 0, \; &\text{if}\;\bx \in C \\ +\infty \; &\text{otherwise}.\end{cases}
\]
We have $\chi_{C} \in \gmRn$ if and only if $C$ is closed and convex.
\end{defn}
\begin{defn}[Euclidean projection]\label{def:projection}
Let $C \subset \Rn$ denote a closed convex subset and let $\bx \in \Rn$. The projection of $\bx$ onto $C$ is defined as $\proj_{C}(\bx) \coloneqq \argmin_{\by \in C} \normsq{\bx-\by}$. The projection of $\bx$ onto $C$ exists, is unique, and satisfies the characterization~\cite[Section III, 3.1]{hiriart2013convexI}:
\begin{equation}\label{eq:char_projections}
\langle \bx - \proj_{C}(\bx), \by - \proj_{C}(\bx) \rangle \leqslant 0 \; \text{for every } \by \in C.
\end{equation}
\end{defn}
\begin{defn}[Normal and tangent cones]\label{def:normal_cones}
Let $C \subset \Rn$ denote a closed convex set. A vector $\bs \in \Rn$ is said to be normal to $C$ at $\bx \in C$ if
\[
\langle \bs, \by - \bx \rangle \leqslant 0\ \text{for every $\by \in C$.}
\]
This set of directions is called the normal cone of $C$ at $\bx$ and is denoted by $N_{C}(\bx)$. The tangent cone of $C$ at $\bx$ is the polar of the normal cone of $C$ at $\bx$:
\[
T_{C}(\bx) \coloneqq \left\{\bd \in \Rn : \langle\bs,\bd\rangle\leqslant 0 \ \text{for every $\bs \in N_{C}(\bx)$}\right\}.
\]
\end{defn}
\section{Technical proofs}
\label{app:proofs}

\subsection{Proof of Theorem~\ref{thm:scc_local}}
\label{app:thm:scc_local}
We prove part~(i) in Steps~1--3, part~(ii) in Step~4, and part~(iii) in Step~5. Throughout, we write $\bd \coloneqq \bd(\bxo)$, $\bp \coloneqq \bp(\bxo)$, $J \coloneqq J(\bxo)$, and $G \coloneqq G(\bxo)$, where $\bp$ is chosen so that~\eqref{eq:LI} holds. Note that $[\matr{A}\bd]_{J \cup G} = \boldsymbol{0}$, by~\eqref{eq:nnls_oc_2_5} on $J$ and~\eqref{eq:grazing_set} on $G$.

\paragraph{Step~1.} For a set $K$ with $J \subseteq K \subseteq J \cup G$, let $\by_{K}(\tau) \coloneqq \bxo + \matr{M}_{K}(\tau)\bd$ denote the candidate built from $K$, with $\matr{M}_{K}$ as in Section~\ref{subsec:removing_scc_local}. We construct one such set $K$ whose candidate is feasible on a right neighborhood of $0$. Two identities hold. First, the computation behind~\eqref{eq:active_rows_invariant}, with $K$ in place of $J(\bxo)$, gives $\matr{A}_{K}\matr{M}_{K}(\tau)\bd = \boldsymbol{0}$, so the constraints of $K$ remain active along $\by_{K}$. Second, $\dot\by_{K}(\tau) = -\matr{P}_{K}^{\perp}(\matr{Q}\by_{K}(\tau)+\bc)$ for every $\tau \geqslant 0$. Indeed, differentiating~\eqref{eq:int_mexp} gives $\dot\by_{K}(\tau) = e^{-\matr{P}_{K}^{\perp}\matr{Q}\tau}\bd$, while integrating the derivative of $s \mapsto e^{-\matr{P}_{K}^{\perp}\matr{Q}s}$ over $[0,\tau]$ gives $e^{-\matr{P}_{K}^{\perp}\matr{Q}\tau} - \matr{I}_{n\times n} = -\matr{P}_{K}^{\perp}\matr{Q}\matr{M}_{K}(\tau)$, so that $\dot\by_{K}(\tau) = \bd - \matr{P}_{K}^{\perp}\matr{Q}\matr{M}_{K}(\tau)\bd$. The right-hand side $-\matr{P}_{K}^{\perp}(\matr{Q}\by_{K}(\tau)+\bc)$ of the claimed identity expands to $-\matr{P}_{K}^{\perp}(\matr{Q}\bxo+\bc) - \matr{P}_{K}^{\perp}\matr{Q}\matr{M}_{K}(\tau)\bd$, and $-\matr{P}_{K}^{\perp}(\matr{Q}\bxo+\bc) = \bd$ because $\bd \in \ker\matr{A}_{K}$ and $-(\matr{Q}\bxo+\bc) - \bd = \matr{A}_{J}^{\top}\bp_{J} \in \image{\matr{A}_{K}^{\top}} = (\ker\matr{A}_{K})^{\perp}$. We call
\begin{equation}\label{eq:normal_part}
\bn_{K}(\tau) \coloneqq -\matr{Q}\by_{K}(\tau) - \bc - \dot{\by}_{K}(\tau) = -(\matr{I}_{n\times n} - \matr{P}^{\perp}_{K})\bigl(\matr{Q}\by_{K}(\tau) + \bc\bigr)
\end{equation}
the \emph{normal part} along $\by_{K}$, which in particular lies in $\image{\matr{A}_{K}^{\top}}$.

Suppose that $\by_{K}(\tau) \in C$ and $\bn_{K}(\tau) \in N_{C}(\by_{K}(\tau))$ for every $\tau$ in an interval $[0,\tau')$. Then $\dot{\by}_{K}(\tau) = -\nabla_{\bx}f(\by_{K}(\tau)) - \bn_{K}(\tau) \in -\partial_{\bx}V(\by_{K}(\tau))$, so the candidate solves~\eqref{eq:main_grad_incl} on $[0,\tau')$ with $\by_{K}(0) = \bxo$. It therefore coincides there with the slow solution by the monotonicity argument in Step~3 of the proof of Theorem~\ref{thm:local_solution}. Constraints inactive at $\bxo$ and constraints $i \in \eqset(\bxo)$ with $[\matr{A}\bd]_{i} < 0$ stay inactive for all sufficiently small $\tau > 0$ along any candidate. Near $\tau = 0$, the two conditions therefore reduce to the signs of the constraint values $\nu_{i}(\tau) \coloneqq [\matr{A}\by_{K}(\tau)]_{i} - b_{i}$ for $i \in G \setminus K$ and to the membership of $\bn_{K}(\tau)$ in $N_{C}(\by_{K}(\tau))$. This membership holds whenever $\bn_{K}(\tau)$ is a nonnegative combination of the rows of $\matr{A}_{K}$, since the constraints of $K$ are active along $\by_{K}$. Steps~2 and~3 below build $K$ iteratively and determine these signs near $\tau = 0$, respectively.

\paragraph{Step~2.} A recursion selects $K$ and produces vectors $\bd_{r}$, starting from $\bd_{0} \coloneqq \bd$, a running set $L_{1} \coloneqq J$, and an undecided set $U_{1} \coloneqq G$. At stage $r = 1,2,\dots$, let $\bz_{r}$ denote the projection of $-\matr{Q}\bd_{r-1}$ onto the cone $\mathcal{K}_{r} \coloneqq \{\bs \in \ker\matr{A}_{L_{r}} : [\matr{A}\bs]_{U_{r}} \leqslant \boldsymbol{0}\}$. Its polar cone is the normal cone of the polyhedron $\mathcal{K}_{r}$ at $\boldsymbol{0}$, where every constraint of $\mathcal{K}_{r}$ is active, so it is generated by the vectors $\pm\matr{A}_{i}^{\top}$ for $i \in L_{r}$ and $\matr{A}_{i}^{\top}$ for $i \in U_{r}$~\cite[Example~5.2.6b]{hiriart2013convexI}. Moreau's decomposition theorem, applied to $\mathcal{K}_{r}$ and its polar cone as in~\eqref{eq:moreau_decomp}, therefore gives $-\matr{Q}\bd_{r-1} = \bz_{r} + \Atop\bm{\eta}$ for some $\bm{\eta} \in \Rm$ supported on $L_{r} \cup U_{r}$, with $\bm{\eta}_{U_{r}} \geqslant \boldsymbol{0}$ and $\langle\bz_{r},\Atop\bm{\eta}\rangle = 0$. Let $H_{r}$ denote the set of such vectors $\bm{\eta}$, the multipliers of stage $r$. When the rows of $\matr{A}_{L_{r} \cup U_{r}}$ are linearly dependent, $H_{r}$ may contain more than one vector. Every $\bm{\eta} \in H_{r}$ satisfies the complementarity condition
\begin{equation}\label{eq:stage_complementarity}
\eta_{i}[\matr{A}\bz_{r}]_{i} = 0 \qquad \text{for every } i \in U_{r}.
\end{equation}
Indeed, $[\matr{A}\bz_{r}]_{L_{r}} = \boldsymbol{0}$, so $0 = \langle\bz_{r},\Atop\bm{\eta}\rangle = \sum_{i \in U_{r}}\eta_{i}[\matr{A}\bz_{r}]_{i}$, and every term of this sum is nonpositive.

The stage commits $U_{r}^{+} \coloneqq \{i \in U_{r} : \eta_{i} > 0 \text{ for some } \bm{\eta} \in H_{r}\}$, drops $D_{r} \coloneqq \{i \in U_{r} : [\matr{A}\bz_{r}]_{i} < 0\}$, and sets $L_{r+1} \coloneqq L_{r} \cup U_{r}^{+}$, $U_{r+1} \coloneqq U_{r} \setminus (U_{r}^{+} \cup D_{r})$, and $\bd_{r} \coloneqq \bz_{r}$. The sets $U_{r}^{+}$ and $D_{r}$ are disjoint by~\eqref{eq:stage_complementarity}. The set $H_{r}$ is convex, so averaging one vector of $H_{r}$ for each index of $U_{r}^{+}$ gives a vector $\bm{\eta}^{(r)} \in H_{r}$ with $\eta^{(r)}_{i} > 0$ for every $i \in U_{r}^{+}$; we fix such a vector. For every $i \in U_{r+1}$, the definitions of $U_{r}^{+}$ and $D_{r}$ give
\begin{equation}\label{eq:undecided}
\eta_{i} = 0 \ \text{for every } \bm{\eta} \in H_{r}
\qquad \text{and} \qquad
[\matr{A}\bz_{r}]_{i} = 0.
\end{equation}
A stage that commits or drops removes at least one index from the undecided set, so at most $|G|$ stages do. Let $\bar{s}$ denote the last such stage, with $\bar{s} \coloneqq 0$ if there is none. The running set stabilizes at $K \coloneqq L_{\bar{s}+1}$, with $J \subseteq K \subseteq J \cup G$, and the undecided set at $R \coloneqq U_{\bar{s}+1}$. We write $J^{+}_{*}(\bxo)$, $\tau^{+}_{\mathrm{feas}}(\bxo)$, and $\tau^{+}_{*}(\bxo)$ for $K$ and the objects built from it, as in Section~\ref{subsec:removing_scc_local}.

The first stage commits the grazing constraints at which some multiplier of the projection is positive, drops those with $[\matr{A}\bz_{1}]_{i} < 0$, and leaves the others undecided. Step~3 shows that, along the candidate, a constraint committed at stage $r$ carries a weight of order $\tau^{r}$ in the normal part, and that the value of a constraint dropped at stage $r$ is of order $\tau^{r+1}$. For an undecided constraint, both leading terms vanish, and the next stage computes higher-order terms. The next lemma shows that the vectors $\bd_{r}$ depend on the final set $K$ alone.

\begin{lem}[Final-set identity]
\label{lem:final_set}
For every $r \geqslant 1$, every $\bm{\eta} \in H_{r}$ is supported on $L_{r+1} \subseteq K$, and $\bz_{r} \in \ker\matr{A}_{K}$. Hence $-\matr{Q}\bd_{r-1} = \bz_{r} + \matr{A}_{K}^{\top}\bm{\eta}_{K}$ for every $\bm{\eta} \in H_{r}$, and $\bd_{r} = \bz_{r} = -\matr{P}^{\perp}_{K}\matr{Q}\bd_{r-1} = (-\matr{P}^{\perp}_{K}\matr{Q})^{r}\bd$.
\end{lem}
\begin{proof}
Let $\bm{\eta} \in H_{r}$. It is supported on $L_{r} \cup U_{r}$. On $D_{r}$, complementarity~\eqref{eq:stage_complementarity} gives $\eta_{i} = 0$, and on $U_{r+1}$, \eqref{eq:undecided} gives $\eta_{i} = 0$. Hence $\bm{\eta}$ is supported on $L_{r} \cup U_{r}^{+} = L_{r+1} \subseteq K$. Every index of $K$ lies in $L_{r}$, in $U_{r}^{+}$, or in $U_{r+1}$, because an index committed after stage $r$ is undecided at stage $r+1$. The definition of $\mathcal{K}_{r}$ gives $[\matr{A}\bz_{r}]_{L_{r}} = \boldsymbol{0}$. Complementarity~\eqref{eq:stage_complementarity} for $\bm{\eta}^{(r)}$ gives $[\matr{A}\bz_{r}]_{i} = 0$ for $i \in U_{r}^{+}$, and~\eqref{eq:undecided} gives the same for $i \in U_{r+1}$. Hence $\bz_{r} \in \ker\matr{A}_{K}$, while $\Atop\bm{\eta} = \matr{A}_{K}^{\top}\bm{\eta}_{K} \in \image{\matr{A}_{K}^{\top}} = (\ker\matr{A}_{K})^{\perp}$. Uniqueness of the orthogonal decomposition of $-\matr{Q}\bd_{r-1}$ along $\ker\matr{A}_{K} \oplus \image{\matr{A}_{K}^{\top}}$ gives $\bz_{r} = -\matr{P}^{\perp}_{K}\matr{Q}\bd_{r-1}$. Since $\bd_{r} = \bz_{r}$ and $\bd_{0} = \bd$, induction on $r$ gives the last claim.
\end{proof}

\paragraph{Step~3.} Consider the functions of Step~1 along $\by_{K}$. By Lemma~\ref{lem:final_set}, $\dot{\by}_{K}(\tau) = e^{-\matr{P}^{\perp}_{K}\matr{Q}\tau}\bd = \sum_{r \geqslant 0}(\tau^{r}/r!)\bd_{r}$. Every constraint value and the normal part $\bn_{K}$ are therefore entire functions of $\tau$.

\paragraph{Constraint values.} Let $i \in G$ be dropped at stage $r$. Then $[\matr{A}\bd_{0}]_{i} = [\matr{A}\bd]_{i} = 0$, while $[\matr{A}\bd_{j}]_{i} = [\matr{A}\bz_{j}]_{i} = 0$ for $1 \leqslant j \leqslant r-1$ by~\eqref{eq:undecided}, since $i \in U_{j+1}$, and $[\matr{A}\bd_{r}]_{i} = [\matr{A}\bz_{r}]_{i} < 0$. Hence $\nu_{i}(\tau) = (\tau^{r+1}/(r+1)!)[\matr{A}\bz_{r}]_{i} + O(\tau^{r+2})$ is negative for all sufficiently small $\tau > 0$. The constraints of $R$ are never decided and form the \emph{riding set}. Every $i \in R$ has $[\matr{A}\bd_{r}]_{i} = 0$ for every $r \geqslant 0$ by~\eqref{eq:undecided} and $i \in G$. With Step~1, this shows that $R$ is the set of constraints outside $K$ that $\tau^{K}_{\mathrm{feas}}(\bxo)$ ignores. Every constraint outside $K \cup R$ has a value that is negative for all sufficiently small $\tau > 0$ and, by the definition of $\tau^{K}_{\mathrm{feas}}(\bxo)$, has no zero in $(0,\tau^{K}_{\mathrm{feas}}(\bxo))$. Hence $\tau^{K}_{\mathrm{feas}}(\bxo) > 0$ and
\begin{equation}\label{eq:active_piece}
\by_{K}(\tau) \in C
\quad \text{and} \quad
\eqset(\by_{K}(\tau)) = K \cup R
\qquad \text{for every } \tau \in (0,\tau^{K}_{\mathrm{feas}}(\bxo)).
\end{equation}

\paragraph{Normal part.} We have $(\diff^{r}\by_{K}/\diff\tau^{r})(0) = \bd_{r-1}$ for $r\geqslant 1$. Evaluating the normal part~\eqref{eq:normal_part} and its derivatives at $\tau = 0$ gives
\begin{gather*}
\bn_{K}(0) = -(\matr{I}_{n\times n} - \matr{P}^{\perp}_{K})(\matr{Q}\bxo + \bc) = \matr{A}_{J}^{\top}\bp_{J} = \matr{A}_{K}^{\top}\bp_{K},
\\
\frac{\diff^{r}\bn_{K}}{\diff\tau^{r}}(0) = -(\matr{I}_{n\times n} - \matr{P}^{\perp}_{K})\matr{Q}\bd_{r-1} = \matr{A}_{K}^{\top}\bm{\eta}^{(r)}_{K} \quad (r \geqslant 1),
\end{gather*}
where the first line uses $-(\matr{Q}\bxo + \bc) = \bd + \matr{A}_{J}^{\top}\bp_{J}$ with $\bd \in \ker\matr{A}_{K}$, and the second line uses Lemma~\ref{lem:final_set} with $\bz_{r} \in \ker\matr{A}_{K}$. Let $\lambda_{\max}$ denote the largest eigenvalue of $\matr{Q}$. Since 
\[
\normtwo{(\matr{I}_{n\times n} - \matr{P}^{\perp}_{K})\matr{Q}} \leqslant \lambda_{\max} \quad \text{and} \quad \normtwo{\matr{P}^{\perp}_{K}\matr{Q}} \leqslant \lambda_{\max},
\]
Lemma~\ref{lem:final_set} also gives $\normtwo{(\diff^{r}\bn_{K}/\diff\tau^{r})(0)} \leqslant \lambda_{\max}\normtwo{\bd_{r-1}} \leqslant \lambda_{\max}^{r}\normtwo{\bd}$. Let
\begin{equation}\label{eq:w_constructed}
\bw(\tau) \coloneqq \bp_{K} + \sum_{r=1}^{\bar{s}}\frac{\tau^{r}}{r!}\bm{\eta}^{(r)}_{K} + (\matr{A}_{K}^{\top})^{\dagger}\bm{\rho}(\tau),
\quad \text{where} \quad
\bm{\rho}(\tau) \coloneqq \bn_{K}(\tau) - \matr{A}_{K}^{\top}\Bigl(\bp_{K} + \sum_{r=1}^{\bar{s}}\frac{\tau^{r}}{r!}\bm{\eta}^{(r)}_{K}\Bigr),
\end{equation}
and the dagger denotes the pseudo-inverse. The vector $\bm{\rho}(\tau) = \sum_{r > \bar{s}}(\tau^{r}/r!)(\diff^{r}\bn_{K}/\diff\tau^{r})(0)$ lies in $\image{\matr{A}_{K}^{\top}}$ and satisfies $\normtwo{\bm{\rho}(\tau)} \leqslant \normtwo{\bd}\sum_{r > \bar{s}}(\lambda_{\max}\tau)^{r}/r! = O(\tau^{\bar{s}+1})$. Since $\matr{A}_{K}^{\top}(\matr{A}_{K}^{\top})^{\dagger}$ is the orthogonal projection onto $\image{\matr{A}_{K}^{\top}}$, we have $\matr{A}_{K}^{\top}\bw(\tau) = \bn_{K}(\tau)$.

We claim that $\bw(\tau)$ has positive entries for all sufficiently small $\tau > 0$. A constraint $i \in J$ has $w_{i}(0) = p_{i} > 0$. Let $i$ be committed at stage $r \leqslant \bar{s}$. Then $p_{i} = 0$ because $i \notin J$, and $\eta^{(j)}_{i} = 0$ for $1 \leqslant j \leqslant r-1$ by~\eqref{eq:undecided}, because $i \in U_{j+1}$. Since $\eta^{(r)}_{i} > 0$, we obtain $w_{i}(\tau) = (\tau^{r}/r!)\eta^{(r)}_{i} + O(\tau^{r+1})$. Every constraint of $K$ lies in $J$ or is committed at a stage $r \leqslant \bar{s}$, which proves the claim. The constraints of $K$ are active along $\by_{K}$, so $\bn_{K}(\tau) = \matr{A}_{K}^{\top}\bw(\tau)$ lies in $N_{C}(\by_{K}(\tau))$ for all sufficiently small $\tau > 0$. With~\eqref{eq:active_piece}, the definition~\eqref{eq:tau_J_cone} then gives $\tau^{K}_{J}(\bxo) > 0$.

\paragraph{Conclusion.} Both times are positive, so $\tau^{K}_{*}(\bxo) > 0$. Let $\tau \in [0,\tau^{K}_{*}(\bxo))$. By~\eqref{eq:active_piece}, $\by_{K}(\tau) \in C$, and then~\eqref{eq:tau_J_cone} gives $\bn_{K}(\tau) \in N_{C}(\by_{K}(\tau))$. Step~1 shows that $\by_{K}$ is the slow solution on $[0,\tau^{K}_{*}(\bxo))$, which proves~\eqref{eq:grazing_solution}. Since $\bd(\bx(\tau)) = \dot{\bx}(\tau)$ by Theorem~\ref{thm:MSP}(iii), the identity $\bn_{K}(\tau) = \matr{A}_{K}^{\top}\bw(\tau)$ is~\eqref{eq:grazing_multipliers} for all sufficiently small $\tau > 0$. This proves part~(i).

\paragraph{Step~4.} We prove part~(ii). Write $\tau_{*} \coloneqq \tau^{+}_{*}(\bxo) < +\infty$, $\bx_{*} \coloneqq \bx(\tau_{*}) = \by_{K}(\tau_{*})$, and $\bv^{-} \coloneqq e^{-\matr{P}^{\perp}_{K}\matr{Q}\tau_{*}}\bd = \dot{\by}_{K}(\tau_{*})$, the left limit of $\dot{\bx}(\tau)$ as $\tau \to \tau_{*}^{-}$. By the conclusion of Step~3 and~\eqref{eq:active_piece}, $\bn_{K}(\tau)$ lies in the closed cone $\cone\{\matr{A}_{i}^{\top} : i \in K \cup R\}$ for every $\tau \in (0,\tau_{*})$, hence also at $\tau_{*}$. Some $\hat{\bp} \geqslant \boldsymbol{0}$ supported on $K \cup R$ therefore satisfies $\Atop\hat{\bp} = \bn_{K}(\tau_{*})$. The constraints of $K \cup R$ are active at $\bx_{*}$ by continuity, so $\hat{\bp}$ is feasible for the NNLS problem~\eqref{eq:nnls_problem} at $\bx_{*}$, and its residual is $-\Atop\hat{\bp} - \matr{Q}\bx_{*} - \bc = \bv^{-}$. Let $I^{+}$ denote the constraints outside $K \cup R$ active at $\bx_{*}$, so that $\eqset(\bx_{*}) = K \cup R \cup I^{+}$. In the KKT test of Step~2 of the proof of Proposition~\ref{prop:continuation}, the sign condition holds with equality over $K$, because the first identity of Step~1 differentiates to $[\matr{A}\bv^{-}]_{K} = \boldsymbol{0}$, and over $R$, because the values there vanish identically. This equality settles complementarity, since $\hat{\bp}$ is supported on $K \cup R$. The values of the constraints of $I^{+}$ are negative on $(0,\tau_{*})$ and vanish at $\tau_{*}$, so $[\matr{A}\bv^{-}]_{I^{+}} \geqslant \boldsymbol{0}$. Hence $\hat{\bp}$ is an optimal NNLS solution at $\bx_{*}$ if and only if $[\matr{A}\bv^{-}]_{I^{+}} = \boldsymbol{0}$, that is, if and only if the hypothesis of part~(ii) fails. Steps~3a and~3b of that proof then apply, and part~(ii) follows.

\paragraph{Step~5.} We prove part~(iii). Suppose that the rows of $\matr{A}_{J \cup G}$ are linearly independent, and let $K'$ be a set with $J \subseteq K' \subseteq J \cup G$ for which part~(i) holds. By part~(i), $\bx(\tau) = \by_{K'}(\tau)$ for $\tau \in [0,\tau^{K'}_{*}(\bxo))$, so the constraints of $K'$ are active at $\bx(\tau)$. By Step~1, the constraints outside $J \cup G$ are inactive at $\bx(\tau)$ for all sufficiently small $\tau > 0$. For such $\tau$, the rows of $\matr{A}_{\eqset(\bx(\tau))}$ are linearly independent, and part~(i) gives $\bw$ with positive entries such that $-\matr{Q}\bx(\tau) - \bc - \bd(\bx(\tau)) = \matr{A}_{K'}^{\top}\bw$. The vector $\bp'$ equal to $\bw$ on $K'$ and to zero elsewhere is therefore feasible for the NNLS problem~\eqref{eq:nnls_problem} at $\bx(\tau)$, and its residual $\bd(\bx(\tau))$ attains the optimal value, so $\bp'$ is an optimal NNLS solution. Every optimal NNLS solution $\bp''$ satisfies $\Atop\bp'' = \bn(\bx(\tau)) = \Atop\bp'$ and is supported on $\eqset(\bx(\tau))$, whose rows are linearly independent, so $\bp'' = \bp'$. Hence $J(\bx(\tau)) = K'$ for all sufficiently small $\tau > 0$. The same holds for $K' = K$, so $K' = K$.
\hfill$\square$

\subsection{Proof of Theorem~\ref{thm:finite_general}}
\label{app:finite_general}
Throughout, let $\tau \mapsto \bx(\tau)$ denote the slow solution with $\bx(0) = \bx^{(0)}$, given piecewise by~\eqref{eq:global_solution}, and let $\lambda_{\max}$ denote the largest eigenvalue of $\matr{Q}$. Theorem~\ref{thm:finite_general} claims the breakpoints are finitely many, the direction $\bd(\bx(\tau))$ converges to a limit $\bd^{s}$, and $\bx(\tau)$ converges to an optimal solution if $\bd^{s} = \boldsymbol{0}$, while the problem is unbounded below otherwise. The main difficulty is the first claim. Theorem~\ref{thm:scc_local} gives every piece a positive length but no lower bound on it, so infinitely many breakpoints could accumulate at a finite time or tend to infinity. We exclude both cases by contradiction. We first collect tools in Appendix~\ref{app:shared_objects}, among them three lemmas on trajectories whose pieces accumulate at a finite time. Appendix~\ref{app:finite_general_setting} shows that every breakpoint is a zero of an affine \emph{event value}, and that no breakpoint occurs once every event value either vanishes identically or has no zero. Appendix~\ref{app:limit_direction} identifies the limit $\bd^{s}$ and proves the last claim of the theorem. Appendix~\ref{app:no_finite_accum} excludes accumulation at a finite time by expanding the deviation of the trajectory from an affine path in powers of the remaining time. Appendix~\ref{app:no_infinite_accum} excludes accumulation at infinity by expanding the deviation of the shifted trajectory $\bx(\tau) - \tau\bd^{s}$ from its limit along eigenvectors. Appendix~\ref{app:finite_general_proof} assembles the proof and derives the two formulas from the final piece.

\subsubsection{Tools}
\label{app:shared_objects}
Given $J \subseteq [m]$, let $\matr{P}^{\perp}_{J}$ denote the orthogonal projection onto $\ker\matr{A}_{J}$ and define
\begin{equation}\label{eq:hessian_compressed}
\matr{S}_{J} \coloneqq \matr{P}^{\perp}_{J}\matr{Q}\matr{P}^{\perp}_{J}.
\end{equation}
The matrix $\matr{S}_{J}$ is SPSD, with $\normtwo{\matr{S}_{J}} \leqslant \lambda_{\max}$ and $\image{\matr{S}_{J}} \subseteq \ker\matr{A}_{J}$, and $\matr{S}_{J}\bz = \matr{P}^{\perp}_{J}\matr{Q}\bz$ for every $\bz \in \ker\matr{A}_{J}$. When the rows of $\matr{A}_{J}$ are linearly independent, we write
\begin{equation}
\label{eq:pJ_flow}
\bp_{J}(\bx) \coloneqq -\left(\matr{A}_{J}\matr{A}_{J}^{\top}\right)^{-1}\matr{A}_{J}\nabla_{\bx}f(\bx)
\end{equation}
for the least-squares multiplier of~\eqref{eq:nnls_explicit}, so that $-\nabla_{\bx}f(\bx) = -\matr{P}^{\perp}_{J}\nabla_{\bx}f(\bx) + \matr{A}_{J}^{\top}\bp_{J}(\bx)$. The next lemma extends formula~\eqref{eq:steepest_dd} to cones larger than the tangent cone.

\begin{lem}
\label{lem:velocity_identity}
Let $J$ be an index set with linearly independent rows, let $\mathcal{K} \coloneqq \{\bs \in \Rn : \matr{A}_{J}\bs \leqslant \boldsymbol{0}\}$, and let $\bx \in \Rn$ satisfy $\bp_{J}(\bx) \geqslant \boldsymbol{0}$. Then $-\matr{P}^{\perp}_{J}\nabla_{\bx}f(\bx) = \proj_{\mathcal{K}}(-\nabla_{\bx}f(\bx))$. In particular, $\bd(\bx) = \proj_{\mathcal{K}}(-\nabla_{\bx}f(\bx))$ for every $\bx \in C$ and $J = J(\bx)$, with $\bp(\bx)$ chosen so that~\eqref{eq:LI} holds.
\end{lem}
\begin{proof}
The vector $-\matr{P}^{\perp}_{J}\nabla_{\bx}f(\bx)$ lies in $\ker\matr{A}_{J} \subseteq \mathcal{K}$, the residual $\matr{A}_{J}^{\top}\bp_{J}(\bx)$ is a nonnegative combination of the rows of $J$ and so lies in the polar cone of $\mathcal{K}$, and the two are orthogonal. These three properties characterize the projection onto $\mathcal{K}$. For $J = J(\bx)$, formula~\eqref{eq:nnls_explicit} shows that $\bp_{J}(\bx)$ is the restriction of $\bp(\bx)$ to $J$, which is positive, and formula~\eqref{eq:steepest_dd} gives $\bd(\bx) = -\matr{P}^{\perp}_{J}\nabla_{\bx}f(\bx)$.
\end{proof}

The next two lemmas reduce membership in a cone generated by possibly dependent rows to sign conditions on the maps~\eqref{eq:pJ_flow} and recall a basic inequality about closed convex cones.
\begin{lem}
\label{lem:row_bases}
Let $E \subseteq [m]$ and $\bv \in \image{\matr{A}_{E}^{\top}}$. Then $\bv \in \cone\{\matr{A}_{i}^{\top} : i \in E\}$ if and only if some $B \subseteq E$ such that the rows of $\matr{A}_{B}$ form a basis of the row space of $\matr{A}_{E}$ satisfies $(\matr{A}_{B}\matr{A}_{B}^{\top})^{-1}\matr{A}_{B}\bv \geqslant \boldsymbol{0}$.
\end{lem}
\begin{proof}
Suppose that $\bv$ lies in the cone. By Carath\'eodory's theorem~\cite[Corollary~17.1.2]{rockafellar1997convex}, $\bv = \matr{A}_{L}^{\top}\bm{u}$ for some $L \subseteq E$ with linearly independent rows and some $\bm{u} \geqslant \boldsymbol{0}$. Extend $L$ to a set $B \subseteq E$ whose rows form a basis of the row space of $\matr{A}_{E}$, and extend $\bm{u}$ by zero to $B$. Then $\bv = \matr{A}_{B}^{\top}\bm{u}$, and multiplying by $(\matr{A}_{B}\matr{A}_{B}^{\top})^{-1}\matr{A}_{B}$ gives $(\matr{A}_{B}\matr{A}_{B}^{\top})^{-1}\matr{A}_{B}\bv = \bm{u} \geqslant \boldsymbol{0}$. Conversely, $\image{\matr{A}_{B}^{\top}} = \image{\matr{A}_{E}^{\top}}$ contains $\bv$, and $\matr{A}_{B}^{\top}(\matr{A}_{B}\matr{A}_{B}^{\top})^{-1}\matr{A}_{B}$ is the orthogonal projection onto it. Hence $\bv = \matr{A}_{B}^{\top}(\matr{A}_{B}\matr{A}_{B}^{\top})^{-1}\matr{A}_{B}\bv$, a nonnegative combination of rows of $\matr{A}_{E}$.
\end{proof}

\begin{lem}
\label{lem:nested_cones}
Let $\mathcal{K}_{2} \subseteq \mathcal{K}_{1}$ be closed convex cones, let $\bz \in \Rn$, and let $\bz_{r} \coloneqq \proj_{\mathcal{K}_{r}}(\bz)$ for $r = 1,2$. Then $\normsq{\bz_{1} - \bz_{2}} \leqslant \normsq{\bz_{1}} - \normsq{\bz_{2}}$.
\end{lem}
\begin{proof}
Expanding the squares shows that the claim is equivalent to $\langle\bz_{1}-\bz_{2},\bz_{2}\rangle \geqslant 0$. The projection onto $\mathcal{K}_{2}$ satisfies $\langle\bz-\bz_{2},\bz_{2}\rangle = 0$, so $\langle\bz_{1}-\bz_{2},\bz_{2}\rangle = \langle\bz_{1}-\bz,\bz_{2}\rangle$, and the projection onto $\mathcal{K}_{1}$ satisfies $\langle\bz-\bz_{1},\bs\rangle \leqslant 0$ for every $\bs \in \mathcal{K}_{1}$, in particular for $\bs = \bz_{2}$.
\end{proof}

The next three lemmas prove control results about a trajectory whose pieces accumulate at a finite time, which will be applied in Appendix~\ref{app:no_finite_accum} to the slow solution. Let $\bar{\tau} < +\infty$, let $\tau^{(0)} < \tau^{(1)} < \cdots$ increase to $\bar{\tau}$, and call the intervals $[\tau^{(k)},\tau^{(k+1)})$ \emph{pieces}. Let $\mathcal{J}$ be a finite family of index sets, and let $J(\cdot)$ be constant on each piece with values in $\mathcal{J}$. A piece on which $J(\cdot) = J$ is a \emph{$J$-piece}. Every $J \in \mathcal{J}$ is assumed to occur on infinitely many pieces. We write $\sigma(\tau) \coloneqq \bar{\tau} - \tau$ for the remaining time.

\begin{lem}
\label{lem:driven_shared}
In this setting, let $s$ be positive and nonincreasing on $[\tau^{(0)},\bar{\tau})$, and put $\Lambda(\tau) \coloneqq \int_{\tau}^{\bar{\tau}}s(\rho)\diff\rho$, so that $\Lambda \leqslant \sigma s$. Let $\bh \in \Rn$, and for $J \in \mathcal{J}$ let $\bu_{J} \coloneqq \matr{P}^{\perp}_{J}\bh$ be the \emph{drift} of $J$ and $\bq_{J} \coloneqq \bh - \bu_{J} \in \image{\matr{A}_{J}^{\top}}$; let $u_{\max} \coloneqq \max_{J \in \mathcal{J}}\normtwo{\bu_{J}}$. Let $\bm{\omega}$ be bounded and measurable on $[\tau^{(0)},\bar{\tau})$ and satisfy, for every $\tau$ there,
\begin{equation}
\label{eq:driven_shared}
\bm{\omega}(\tau) = -\int_{\tau}^{\bar{\tau}}\bigl(s(\rho)\bu_{J(\rho)} - \matr{S}_{J(\rho)}\bm{\omega}(\rho)\bigr)\diff\rho
\qquad \text{and} \qquad
\matr{A}_{J(\tau)}\bm{\omega}(\tau) = \boldsymbol{0}.
\end{equation}
Then:
\begin{itemize}
\item[(i)] $\normtwo{\bm{\omega}(\tau)} \leqslant 2u_{\max}\,\Lambda(\tau)$ whenever $\lambda_{\max}\sigma(\tau) \leqslant \tfrac12$;
\item[(ii)] the drifts $\bu_{J}$, $J \in \mathcal{J}$, all equal one vector, which lies in $\ker\matr{A}_{J}$ for every $J \in \mathcal{J}$.
\end{itemize}
\end{lem}
\begin{proof}
Let $\Omega(\tau) \coloneqq \sup_{\rho \in [\tau,\bar{\tau})}\normtwo{\bm{\omega}(\rho)}$, finite because $\bm{\omega}$ is bounded. For $\rho \geqslant \tau$, the first identity of~\eqref{eq:driven_shared} gives
\[
\normtwo{\bm{\omega}(\rho)} \leqslant u_{\max}\,\Lambda(\rho) + \lambda_{\max}\sigma(\rho)\Omega(\rho) \leqslant u_{\max}\,\Lambda(\tau) + \tfrac12\Omega(\tau),
\]
because $\Lambda$ and $\sigma$ are nonincreasing and $\lambda_{\max}\sigma(\tau) \leqslant \tfrac12$. Taking the supremum over $\rho$ gives $\Omega(\tau) \leqslant u_{\max}\,\Lambda(\tau) + \tfrac12\Omega(\tau)$. This proves (i).

Now, choose $\tilde{J} \in \mathcal{J}$ minimizing $\normtwo{\bu_{J}}$ and write $\bu \coloneqq \bu_{\tilde{J}}$, $\bq \coloneqq \bq_{\tilde{J}}$. For $J' \in \mathcal{J}$, since $\matr{P}^{\perp}_{J'}$ is symmetric and idempotent,
\begin{equation}
\label{eq:qtest_shared}
\bq^{\top}\bu_{J'} = \langle\bu_{J'},\bu_{J'} - \bu\rangle = \tfrac12\normsq{\bu_{J'} - \bu} + \tfrac12\bigl(\normsq{\bu_{J'}} - \normsq{\bu}\bigr) \geqslant \tfrac12\normsq{\bu_{J'} - \bu},
\end{equation}
and if $\bu_{J'} = \bu$, then $\bu \in \image{\matr{P}^{\perp}_{J'}}$ and $\matr{P}^{\perp}_{J'}\bq = \bu_{J'} - \matr{P}^{\perp}_{J'}\bu = \boldsymbol{0}$, hence
\begin{equation}
\label{eq:fullann_shared}
\bq^{\top}\matr{S}_{J'} = \bigl(\matr{P}^{\perp}_{J'}\bq\bigr)^{\top}\matr{Q}\matr{P}^{\perp}_{J'} = \boldsymbol{0}^{\top}.
\end{equation}
Suppose that some $J' \in \mathcal{J}$ has $\bu_{J'} \neq \bu$, so that $\beta = \tfrac12\min\{\normsq{\bu_{J'} - \bu} : \bu_{J'} \neq \bu\} > 0$, the family $\mathcal{J}$ being finite. Let $\tau^{(k)}$ be the left endpoint of a $\tilde{J}$-piece with $\lambda_{\max}\sigma(\tau^{(k)}) \leqslant \tfrac12$, and let $\Theta$ be the set of times $\rho \geqslant \tau^{(k)}$ with $\bu_{J(\rho)} \neq \bu$. Then $\bq^{\top}\bm{\omega}(\tau^{(k)}) = 0$, by the second identity of~\eqref{eq:driven_shared} and $\bq \in \image{\matr{A}_{\tilde{J}}^{\top}}$. We test the first identity at $\tau = \tau^{(k)}$ with $\bq$. A time $\rho \notin \Theta$ contributes nothing, by~\eqref{eq:qtest_shared} and~\eqref{eq:fullann_shared}. At a time $\rho \in \Theta$, the first term of the integrand contributes at least $\beta s(\rho)$, by~\eqref{eq:qtest_shared}, and the second at most $2\lambda_{\max}u_{\max}\normtwo{\bq}\sigma(\tau^{(k)})s(\rho)$ in absolute value, by (i) and $\Lambda(\rho) \leqslant \sigma(\rho)s(\rho) \leqslant \sigma(\tau^{(k)})s(\rho)$. Hence $\bigl(\beta - 2\lambda_{\max}u_{\max}\normtwo{\bq}\sigma(\tau^{(k)})\bigr)\int_{\Theta}s(\rho)\diff\rho \leqslant 0$. Since $\tilde{J}$ occurs on infinitely many pieces and $\sigma(\tau^{(k)}) \to 0$, some such $\tau^{(k)}$ makes the factor in parentheses positive, and then the integral vanishes. The set $\Theta$ is a union of pieces, each of positive length, and $s > 0$, so $\Theta$ is empty. But every $J' \in \mathcal{J}$ has pieces after $\tau^{(k)}$, a contradiction. Hence $\bu_{J'} = \bu$ for every $J' \in \mathcal{J}$, and $\bu = \matr{P}^{\perp}_{J'}\bh \in \ker\matr{A}_{J'}$.
\end{proof}

\begin{lem}
\label{lem:alignment_shared}
Let $\bar{\bv} \in \ker\matr{A}_{J}$ for every $J \in \mathcal{J}$, and let $\bm{\zeta}$ be bounded and measurable and satisfy~\eqref{eq:driven_shared} with $s = \sigma$ and $\bh = \matr{Q}\bar{\bv}$. Then there exist vectors $\bu^{[0]} \coloneqq \bar{\bv}, \bu^{[1]}, \bu^{[2]}, \dots$ such that $\matr{S}_{J}\bu^{[j-1]} = \bu^{[j]} \in \ker\matr{A}_{J}$ for every $J \in \mathcal{J}$ and every $j \geqslant 1$, and such that, for every $r \geqslant 0$, the function $\bm{\omega}_{r} \coloneqq \bm{\zeta} + \sum_{j=1}^{r}\frac{\sigma^{j+1}}{(j+1)!}\bu^{[j]}$ satisfies
\begin{equation}
\label{eq:levelrig_shared}
\normtwo{\bm{\omega}_{r}(\tau)} \leqslant 2\normtwo{\bu^{[r+1]}}\frac{\sigma(\tau)^{r+2}}{(r+2)!}
\qquad \text{whenever $\lambda_{\max}\sigma(\tau) \leqslant \tfrac12$.}
\end{equation}
\end{lem}
\begin{proof}
We show by induction on $r \geqslant 0$ that $\bu^{[1]},\dots,\bu^{[r]}$ exist with the stated relations and that $\bm{\omega}_{r}$ satisfies the hypotheses of Lemma~\ref{lem:driven_shared} with $s = \sigma^{r+1}/(r+1)!$ and $\bh = \matr{Q}\bu^{[r]}$, whose drifts are $\matr{P}^{\perp}_{J}\matr{Q}\bu^{[r]} = \matr{S}_{J}\bu^{[r]}$ because $\bu^{[r]} \in \ker\matr{A}_{J}$. Part~(ii) of that lemma then supplies the common vector $\bu^{[r+1]} \coloneqq \matr{S}_{J}\bu^{[r]} \in \ker\matr{A}_{J}$, and part~(i), with $\Lambda = \sigma^{r+2}/(r+2)!$ and $u_{\max} = \normtwo{\bu^{[r+1]}}$, gives~\eqref{eq:levelrig_shared}.

For $r = 0$, the function $\bm{\omega}_{0} = \bm{\zeta}$ satisfies the hypotheses by assumption. Let $r \geqslant 1$. The function $\bm{\omega}_{r} = \bm{\omega}_{r-1} + \frac{\sigma^{r+1}}{(r+1)!}\bu^{[r]}$ is bounded, and $\matr{A}_{J(\tau)}\bm{\omega}_{r}(\tau) = \boldsymbol{0}$ because $\bu^{[r]} \in \ker\matr{A}_{J(\tau)}$. Since $\frac{\sigma(\tau)^{r+1}}{(r+1)!} = \int_{\tau}^{\bar{\tau}}\frac{\sigma(\rho)^{r}}{r!}\diff\rho$, the identity~\eqref{eq:driven_shared} for $\bm{\omega}_{r-1}$, whose drifts are $\bu^{[r]}$, gives
\[
\bm{\omega}_{r}(\tau) = \int_{\tau}^{\bar{\tau}}\matr{S}_{J(\rho)}\bm{\omega}_{r-1}(\rho)\diff\rho = -\int_{\tau}^{\bar{\tau}}\Bigl(\frac{\sigma(\rho)^{r+1}}{(r+1)!}\matr{S}_{J(\rho)}\bu^{[r]} - \matr{S}_{J(\rho)}\bm{\omega}_{r}(\rho)\Bigr)\diff\rho,
\]
which is~\eqref{eq:driven_shared} for $\bm{\omega}_{r}$ with $s = \sigma^{r+1}/(r+1)!$ and the drifts $\matr{S}_{J}\bu^{[r]}$.
\end{proof}

\begin{lem}
\label{lem:sign_shared}
In the setting of Lemma~\ref{lem:alignment_shared}, every $\ba \in \Rn$ admits a time $\tau_{\ba} < \bar{\tau}$ such that $\tau \mapsto \langle\ba,\bm{\zeta}(\tau)\rangle$ either vanishes at every $\tau \in [\tau_{\ba},\bar{\tau})$ or at none.
\end{lem}
\begin{proof}
Let $\bu^{[j]}$ and $\bm{\omega}_{r}$ be as in Lemma~\ref{lem:alignment_shared}, put $\alpha_{j} \coloneqq \langle\ba,\bu^{[j]}\rangle$ for $j \geqslant 1$, and let $\tau' \geqslant \tau^{(0)}$ satisfy $\lambda_{\max}\sigma(\tau') \leqslant \tfrac12$, so that~\eqref{eq:levelrig_shared} holds on $[\tau',\bar{\tau})$. Note that $\normtwo{\bu^{[r+1]}} \leqslant \lambda_{\max}^{r+1}\normtwo{\bar{\bv}}$.

If every $\alpha_{j}$ vanishes, then $\langle\ba,\bm{\zeta}\rangle = \langle\ba,\bm{\omega}_{r}\rangle$ for every $r$, and~\eqref{eq:levelrig_shared} gives
\[
\bigl|\langle\ba,\bm{\zeta}(\tau)\rangle\bigr| \leqslant 2\normtwo{\ba}\normtwo{\bar{\bv}}\sigma(\tau)\bigl(\lambda_{\max}\sigma(\tau)\bigr)^{r+1} \leqslant 2^{-r}\normtwo{\ba}\normtwo{\bar{\bv}}\sigma(\tau)
\]
for every $\tau \in [\tau',\bar{\tau})$. Letting $r \to \infty$ makes $\langle\ba,\bm{\zeta}\rangle$ vanish there, and we take $\tau_{\ba} = \tau'$. Otherwise, let $r$ be minimal with $\alpha_{r} \neq 0$. Then $\langle\ba,\bm{\zeta}\rangle = \langle\ba,\bm{\omega}_{r}\rangle - \alpha_{r}\sigma^{r+1}/(r+1)!$, and~\eqref{eq:levelrig_shared} gives
\[
\bigl|\langle\ba,\bm{\zeta}(\tau)\rangle\bigr| \geqslant \bigl(|\alpha_{r}| - 2\normtwo{\ba}\normtwo{\bu^{[r+1]}}\sigma(\tau)\bigr)\frac{\sigma(\tau)^{r+1}}{(r+1)!},
\]
which is positive once $2\normtwo{\ba}\normtwo{\bu^{[r+1]}}\sigma(\tau) < |\alpha_{r}|$, that is, on some interval $[\tau_{\ba},\bar{\tau})$ with $\tau_{\ba} \geqslant \tau'$.
\end{proof}

\subsubsection{Pieces and events}
\label{app:finite_general_setting}
We now turn to the breakpoints of the slow solution and first describe how its pieces end. As before, a piece of the slow solution with active set $J$ is a \emph{$J$-piece}. For the $k$-th piece, let $R^{(k)}$ denote the riding set of Appendix~\ref{app:thm:scc_local} at $\bx^{(k)}$, and let $N^{(k)} \coloneqq \cone\{\matr{A}_{i}^{\top} : i \in J_{*}^{(k)} \cup R^{(k)}\}$, which contains $-\nabla_{\bx}f(\bx(\tau)) - \bd(\bx(\tau))$ on the open piece by Lemma~\ref{lem:pieces_flow}(i).

\begin{lem}[Dynamics on a piece]
\label{lem:pieces_flow}
Consider the $k$-th piece and let $J = J_{*}^{(k)}$.
\begin{itemize}
\item[(i)] For every $\tau \in [\tau^{(k)},\tau^{(k+1)})$, the constraints of $J$ are active at $\bx(\tau)$, and
\begin{equation}
\label{eq:piece_flow}
\bd(\bx(\tau)) = -\matr{P}^{\perp}_{J}\nabla_{\bx}f(\bx(\tau)) \in \ker\matr{A}_{J}.
\end{equation}
For every $\tau \in (\tau^{(k)},\tau^{(k+1)})$, we have $\eqset(\bx(\tau)) = J \cup R^{(k)}$, $\matr{A}_{J \cup R^{(k)}}\bd(\bx(\tau)) = \boldsymbol{0}$, and $-\nabla_{\bx}f(\bx(\tau)) - \bd(\bx(\tau)) \in N^{(k)}$. If $\tau^{(k+1)} < +\infty$, then the constraints of $J \cup R^{(k)}$ are active at $\bx^{(k+1)}$, and $\bd(\bx(\tau)) \to -\matr{P}^{\perp}_{J}\nabla_{\bx}f(\bx^{(k+1)})$ as $\tau \uparrow \tau^{(k+1)}$. This limit lies in $\ker\matr{A}_{J \cup R^{(k)}}$, and $-(\matr{I}_{n\times n} - \matr{P}^{\perp}_{J})\nabla_{\bx}f(\bx^{(k+1)}) \in N^{(k)}$.

\item[(ii)] If $\tau^{(k+1)} < +\infty$, then at least one of two events occurs at $\tau^{(k+1)}$. In an \emph{enter event}, some $i \notin J$ satisfies $[\matr{A}\bx(\tau)]_{i} < b_{i}$ for every $\tau \in (\tau^{(k)},\tau^{(k+1)})$ and $[\matr{A}\bx^{(k+1)}]_{i} = b_{i}$. In a \emph{leave event}, some $B \subseteq J \cup R^{(k)}$ such that the rows of $\matr{A}_{B}$ form a basis of the row space of $\matr{A}_{J \cup R^{(k)}}$, and some $i \in B$, are such that $\tau \mapsto [\bp_{B}(\bx(\tau))]_{i}$ does not vanish identically on the piece and $[\bp_{B}(\bx^{(k+1)})]_{i} = 0$.
\end{itemize}
\end{lem}
\begin{proof}

Measure time from $\tau^{(k)}$, write $\bd \coloneqq \bd^{(k)}$, $R \coloneqq R^{(k)}$, and $\by \coloneqq \by_{J}$ for the candidate of Appendix~\ref{app:thm:scc_local} at $\bx^{(k)}$, and let $\bn_{J}$ be its normal part~\eqref{eq:normal_part}.

The slow system is autonomous, so Theorem~\ref{thm:scc_local}(i) applies at $\bx^{(k)}$ with $K = J$, and $\bx(\tau) = \by(\tau)$ on the piece. Step~1 of Appendix~\ref{app:thm:scc_local} shows that the constraints of $J$ stay active and that $\dot{\by}(\tau) = -\matr{P}^{\perp}_{J}\nabla_{\bx}f(\by(\tau))$, which lies in $\ker\matr{A}_{J}$. Since $\bd(\bx(\tau)) = \dot{\bx}(\tau)$ by Theorem~\ref{thm:MSP}(iii), this gives~\eqref{eq:piece_flow}. By~\eqref{eq:active_piece}, $\eqset(\bx(\tau)) = J \cup R$ on the open piece. The values of the constraints of $R$ vanish identically along $\by$, so $\matr{A}_{J \cup R}\dot{\by}(\tau) = \boldsymbol{0}$ for every $\tau \geqslant 0$. By Moreau's decomposition~\eqref{eq:moreau_decomp} and~\eqref{eq:subdiff_V}, $-\nabla_{\bx}f(\bx(\tau)) - \bd(\bx(\tau))$ lies in $N_{C}(\bx(\tau))$, which equals $N^{(k)}$ on the open piece. The claims at $\tau^{(k+1)}$ follow by continuity, because $\bx \mapsto \matr{P}^{\perp}_{J}\nabla_{\bx}f(\bx)$ is affine and $N^{(k)}$ is closed. This proves (i).

Suppose that $\tau^{(k+1)} < +\infty$, and abbreviate $\tau_{\mathrm{feas}} \coloneqq \tau^{+}_{\mathrm{feas}}(\bx^{(k)})$ and $\tau_{J} \coloneqq \tau^{+}_{J}(\bx^{(k)})$, so that $\tau_{*}^{(k)} = \min(\tau_{\mathrm{feas}},\tau_{J})$ is finite, and positive by Theorem~\ref{thm:scc_local}(i). Suppose first that $\tau_{*}^{(k)} = \tau_{\mathrm{feas}}$. The infimum~\eqref{eq:tau_feas} runs over the constraints $i \notin J$ whose values do not vanish identically on the piece (Section~\ref{subsec:removing_scc_local}). The positive zeros of each such value form a set that is closed in $(0,+\infty)$, with infimum at least $\tau_{\mathrm{feas}} > 0$, so some $i \notin J$ has a zero at $\tau_{\mathrm{feas}}$ and none in $(0,\tau_{\mathrm{feas}})$. Its value is nonpositive because $\bx(\tau) \in C$, hence negative on $(0,\tau_{\mathrm{feas}})$: an enter event occurs.

Suppose now that $\tau_{J} < \tau_{\mathrm{feas}}$. For every $\tau \in (0,\tau_{\mathrm{feas}})$, \eqref{eq:active_piece} gives $\by(\tau) \in C$ and $N_{C}(\by(\tau)) = N^{(k)}$, and the definition~\eqref{eq:tau_J_cone} gives $\bn_{J}(\tau) \in N^{(k)}$ for $\tau \in (0,\tau_{J})$. The cone $N^{(k)}$ is closed, so $\bn_{J}(\tau_{J}) \in N^{(k)}$. Hence $\tau_{J}$ does not belong to the set in~\eqref{eq:tau_J_cone}, and there are times in $(\tau_{J},\tau_{\mathrm{feas}})$ arbitrarily close to $\tau_{J}$ at which $\bn_{J} \notin N^{(k)}$. Let $B \subseteq J \cup R$ be such that the rows of $\matr{A}_{B}$ form a basis of the row space of $\matr{A}_{J \cup R}$. Since $\matr{A}_{B}\dot{\by}(\tau) = \boldsymbol{0}$ and $\bn_{J}(\tau) \in \image{\matr{A}_{J}^{\top}} \subseteq \image{\matr{A}_{B}^{\top}}$ for every $\tau \geqslant 0$, the definitions~\eqref{eq:normal_part} and~\eqref{eq:pJ_flow} give
\begin{equation}\label{eq:normal_basis}
\bn_{J}(\tau) = \matr{A}_{B}^{\top}(\matr{A}_{B}\matr{A}_{B}^{\top})^{-1}\matr{A}_{B}\bn_{J}(\tau) = \matr{A}_{B}^{\top}\bp_{B}(\by(\tau)) \qquad \text{for every } \tau \geqslant 0.
\end{equation}
By Lemma~\ref{lem:row_bases} and $\bn_{J}(\tau_{J}) \in N^{(k)}$, some such $B$ satisfies $\bp_{B}(\by(\tau_{J})) \geqslant \boldsymbol{0}$. The function $\tau \mapsto \bp_{B}(\by(\tau))$ is entire, because $\bp_{B}$ is affine and $\by$ is entire. Suppose that every $i \in B$ with $[\bp_{B}(\by(\tau_{J}))]_{i} = 0$ had $[\bp_{B}(\by(\cdot))]_{i} \equiv 0$. The positive entries stay positive near $\tau_{J}$ by continuity, so $\bp_{B}(\by(\tau)) \geqslant \boldsymbol{0}$ for every $\tau$ near $\tau_{J}$, and~\eqref{eq:normal_basis} would give $\bn_{J}(\tau) \in N^{(k)}$ for every $\tau > \tau_{J}$ near $\tau_{J}$, a contradiction. Hence some $i \in B$ has $[\bp_{B}(\by(\tau_{J}))]_{i} = 0$ and $[\bp_{B}(\by(\cdot))]_{i} \not\equiv 0$. An entire function that does not vanish identically does not vanish identically on any interval. Since $\bx = \by$ on the piece and $\bx^{(k+1)} = \by(\tau_{J})$, a leave event of $(B,i)$ occurs. This proves (ii).
\end{proof}

Suppose there are infinitely many breakpoints. At each breakpoint $\tau^{(k+1)}$, we fix one event of Lemma~\ref{lem:pieces_flow}(ii), either an enter event of a constraint $i \notin J_{*}^{(k)}$ or a leave event, which we label by the pair $(B,i)$ of its set $B$ and index $i$ in Lemma~\ref{lem:pieces_flow}(ii). The \emph{event value} of an enter event of $i$ is $[\matr{A}\bx]_{i} - b_{i}$, and the event value of a leave event of $(B,i)$ is $[\bp_{B}(\bx)]_{i}$. Both are affine in $\bx$, the second by~\eqref{eq:pJ_flow}, and vanish at $\bx^{(k+1)}$. An active set, a constraint of an enter event, or a pair of a leave event is \emph{recurring} if it occurs on infinitely many pieces. To study late breakpoints, we discard the finitely many pieces with a nonrecurring object and call the remaining pieces the \emph{tail}. Every object on the tail is recurring, even after we discard finitely many further pieces. Let $\Phi$ be the finite set of gradients of recurring event values. Every breakpoint of the tail is therefore a zero of a recurring event value. The next lemma excludes late breakpoints when each recurring event value eventually vanishes identically or has no zero.

\begin{lem}
\label{lem:events_flow}
Suppose there are infinitely many breakpoints. Let $\tau \mapsto \bm{D}(\tau)$ satisfy $\phi(\bx(\tau)) = \langle\ba,\bm{D}(\tau)\rangle$ on the tail at every $\tau$, for every recurring event value $\phi$ and its gradient $\ba \in \Phi$. Suppose that every $\ba \in \Phi$ admits a time $\tau_{\ba}$ on the tail after which $\tau \mapsto \langle\ba,\bm{D}(\tau)\rangle$ either vanishes identically or has no zero. Then no piece of the tail starts after the largest of the times $\tau_{\ba}$.
\end{lem}
\begin{proof}
Let $\tau'$ be the largest of the times $\tau_{\ba}$. Consider a piece $[\tau^{(k)},\tau^{(k+1)})$ of the tail with $\tau^{(k)} > \tau'$, with the event fixed at $\tau^{(k+1)}$. This event is recurring, so its gradient $\ba$ lies in $\Phi$. At an enter event, Lemma~\ref{lem:pieces_flow}(ii) makes $\langle\ba,\bm{D}\rangle$ negative on $(\tau^{(k)},\tau^{(k+1)})$ and zero at $\tau^{(k+1)}$. At a leave event, it makes $\langle\ba,\bm{D}\rangle$ zero at $\tau^{(k+1)}$ without vanishing identically on the piece. Both cases contradict the dichotomy, because the closed piece lies after $\tau_{\ba}$.
\end{proof}

\subsubsection{The limit direction}
\label{app:limit_direction}

We now show that the limit of $\bd(\bx(\tau))$ in Theorem~\ref{thm:finite_general} is the projection $\bd^{s}$ defined by
\begin{equation}
\label{eq:ds}
\mathcal{R} \coloneqq \left\{\bv \in \Rn : \matr{A}\bv \leqslant \boldsymbol{0}, \ \matr{Q}\bv = \boldsymbol{0}\right\},
\qquad
\bd^{s} \coloneqq \proj_{\mathcal{R}}(-\bc),
\qquad
I_{0} \coloneqq \left\{i \in [m] : [\matr{A}\bd^{s}]_{i} = 0\right\}.
\end{equation}
We also prove the alternative between optimality and unboundedness in that theorem. Problem~\eqref{eq:intro_main_problem} is bounded below if and only if no direction $\bd$ satisfies $\matr{A}\bd \leqslant \boldsymbol{0}$, $\matr{Q}\bd = \boldsymbol{0}$, and $\langle\bc,\bd\rangle < 0$, and it then attains its infimum~\cite{frank1956}. Since $\mathcal{R}$ is a closed convex cone, Moreau's decomposition theorem gives
\begin{equation}
\label{eq:ds_moreau}
\langle\bc + \bd^{s},\bd^{s}\rangle = 0
\qquad \text{and} \qquad
\langle\bc + \bd^{s},\br\rangle \geqslant 0 \ \text{for every $\br \in \mathcal{R}$.}
\end{equation}
We also let $C_{0} \coloneqq \{\by \in \Rn : [\matr{A}\by]_{I_{0}} \leqslant \bb_{I_{0}}\}$, $f_{0}(\by) \coloneqq f(\by) + \langle\bd^{s},\by\rangle$, and $V_{0} \coloneqq f_{0} + \chi_{C_{0}}$, so that $V_{0}$ is the objective function of problem~\eqref{eq:intro_main_problem} with the data $\matr{A}_{I_{0}}$, $\bb_{I_{0}}$, and $\bc + \bd^{s}$.

\begin{lem}\label{lem:ds}
The following hold:
\begin{itemize}
\item[(i)] The direction $\bd^{s}$ satisfies $\matr{A}\bd^{s} \leqslant \boldsymbol{0}$, $\matr{Q}\bd^{s} = \boldsymbol{0}$, and $\langle\bc,\bd^{s}\rangle = -\normsq{\bd^{s}}$. If $\bd^{s} \neq \boldsymbol{0}$, then for every $\bx \in C$ the ray $\bx + \Delta\bd^{s}$, $\Delta \geqslant 0$, is feasible and $V(\bx + \Delta\bd^{s}) = V(\bx) - \Delta\normsq{\bd^{s}} \to -\infty$ as $\Delta \to +\infty$.
\item[(ii)] Every $\bx \in C$ satisfies $\normsq{\bd(\bx) - \bd^{s}} \leqslant \normsq{\bd(\bx)} - \normsq{\bd^{s}}$.
\item[(iii)] The function $f_{0}$ attains its minimum over $C_{0}$.
\end{itemize}
\end{lem}
\begin{proof}
The first two relations hold because $\bd^{s} \in \mathcal{R}$, and the third is the first relation in~\eqref{eq:ds_moreau}. The ray is feasible because $\matr{A}\bd^{s} \leqslant \boldsymbol{0}$, and the formula for $V(\bx + \Delta\bd^{s})$ follows from $\matr{Q}\bd^{s} = \boldsymbol{0}$ and the third relation. This proves (i).

Let $\bx \in C$. Every $\br \in \mathcal{R}$ satisfies $\langle\matr{Q}\bx,\br\rangle = \langle\bx,\matr{Q}\br\rangle = 0$, so $\normsq{-\nabla_{\bx}f(\bx) - \br}$ and $\normsq{-\bc - \br}$ differ by a constant independent of $\br$, and therefore $\proj_{\mathcal{R}}(-\nabla_{\bx}f(\bx)) = \bd^{s}$. Moreover $\mathcal{R} \subseteq T_{C}(\bx) = \{\bs \in \Rn : [\matr{A}\bs]_{\eqset(\bx)} \leqslant \boldsymbol{0}\}$. Lemma~\ref{lem:nested_cones}, applied to $\mathcal{R} \subseteq T_{C}(\bx)$ and the vector $-\nabla_{\bx}f(\bx)$, gives (ii).

The function $f_{0}$ is a convex quadratic with linear term $\bc + \bd^{s}$, and $C \subset C_{0}$ is nonempty, so the criterion above, applied to problem~\eqref{eq:intro_main_problem} with the data $\matr{A}_{I_{0}}$, $\bb_{I_{0}}$, and $\bc + \bd^{s}$, reduces the claim to showing that $\langle\bc + \bd^{s},\br\rangle \geqslant 0$ for every $\br$ with $[\matr{A}\br]_{I_{0}} \leqslant \boldsymbol{0}$ and $\matr{Q}\br = \boldsymbol{0}$. For such $\br$ and for $\Delta \geqslant 0$ large enough, $\br + \Delta\bd^{s} \in \mathcal{R}$, because $[\matr{A}\bd^{s}]_{i} < 0$ for every $i \notin I_{0}$. Both relations in~\eqref{eq:ds_moreau} then give $\langle\bc + \bd^{s},\br\rangle = \langle\bc + \bd^{s},\br + \Delta\bd^{s}\rangle \geqslant 0$. This proves (iii).
\end{proof}

\begin{prop}\label{prop:limit_direction}
The following hold:
\begin{itemize}
\item[(i)] The direction satisfies $\bd(\bx(\tau)) \to \bd^{s}$ as $\tau \to +\infty$.
\item[(ii)] There exists $\tau_{0} \geqslant 0$ such that $\eqset(\bx(\tau)) \subseteq I_{0}$ for every $\tau \geqslant \tau_{0}$.
\item[(iii)] The function $\tau \mapsto \by(\tau_{0} + \tau)$, where $\by(\tau) \coloneqq \bx(\tau) - \tau\bd^{s}$, is the solution of the differential inclusion~\eqref{eq:main_grad_incl} for $V_{0}$ with initial value $\by(\tau_{0})$. In particular, $\tau \mapsto \normtwo{\bd(\bx(\tau)) - \bd^{s}}$ is nonincreasing on $[\tau_{0},+\infty)$. The function $\by$ converges as $\tau \to +\infty$ to a point $\bar{\by}$ that minimizes $f_{0}$ over $C_{0}$. If $\bd^{s} = \boldsymbol{0}$, then $\bar{\by} = \lim_{\tau \to +\infty}\bx(\tau)$ is an optimal solution of problem~\eqref{eq:intro_main_problem}.
\end{itemize}
\end{prop}

\begin{proof}
The function $\tau \mapsto f_{0}(\bx(\tau))$ is Lipschitz continuous on bounded intervals, and at almost every $\tau$ its derivative is $\langle\nabla_{\bx}f(\bx(\tau)) + \bd^{s},\bd(\bx(\tau))\rangle = -\normsq{\bd(\bx(\tau))} + \langle\bd^{s},\bd(\bx(\tau))\rangle$, by~\eqref{eq:nnls_oc_3}. Expanding the square in Lemma~\ref{lem:ds}(ii) gives $\langle\bd(\bx),\bd^{s}\rangle \geqslant \normsq{\bd^{s}}$ for every $\bx \in C$, hence
\[
\normsq{\bd(\bx) - \bd^{s}} = \normsq{\bd(\bx)} - 2\langle\bd(\bx),\bd^{s}\rangle + \normsq{\bd^{s}} \leqslant \normsq{\bd(\bx)} - \langle\bd(\bx),\bd^{s}\rangle.
\]
Integrating, and using $\bx(\tau) \in C \subset C_{0}$ with Lemma~\ref{lem:ds}(iii), we obtain for every $\tau \geqslant 0$
\[
\int_{0}^{\tau}\normsq{\bd(\bx(\rho)) - \bd^{s}}\diff\rho \leqslant f_{0}(\bx^{(0)}) - f_{0}(\bx(\tau)) \leqslant f_{0}(\bx^{(0)}) - \min_{C_{0}}f_{0},
\]
so $\int_{0}^{+\infty}\normsq{\bd(\bx(\rho)) - \bd^{s}}\diff\rho$ is finite. By Theorem~\ref{thm:MSP}(i) and~(iii), $\normtwo{\bd(\bx(\tau))}$ decreases to a limit $\nu$, and $\nu \geqslant \normtwo{\bd^{s}}$ by Lemma~\ref{lem:ds}(ii). Since $\normtwo{\bd(\bx(\tau)) - \bd^{s}} \geqslant \nu - \normtwo{\bd^{s}}$ for every $\tau$ and the integral is finite, $\nu = \normtwo{\bd^{s}}$. Lemma~\ref{lem:ds}(ii) at $\bx = \bx(\tau)$ then gives $\normsq{\bd(\bx(\tau)) - \bd^{s}} \leqslant \normsq{\bd(\bx(\tau))} - \normsq{\bd^{s}} \to 0$. This proves (i).

Let $i \notin I_{0}$. By (i), $[\matr{A}\bd(\bx(\tau))]_{i} \to [\matr{A}\bd^{s}]_{i} < 0$, so $[\matr{A}\bx(\tau)]_{i} - b_{i} = [\matr{A}\bx^{(0)}]_{i} - b_{i} + \int_{0}^{\tau}[\matr{A}\bd(\bx(\rho))]_{i}\diff\rho \to -\infty$. The constraints outside $I_{0}$ are finitely many, so all of them are inactive at $\bx(\tau)$ from some time $\tau_{0}$ on. This proves (ii).

Let $\tau \geqslant \tau_{0}$. Since $[\matr{A}\by(\tau)]_{I_{0}} = [\matr{A}\bx(\tau)]_{I_{0}}$, we have $\by(\tau) \in C_{0}$, and by (ii) the constraints of $C_{0}$ active at $\by(\tau)$ are those of $\eqset(\bx(\tau))$, so $N_{C_{0}}(\by(\tau)) = N_{C}(\bx(\tau))$, by~\cite[Example 5.2.6b]{hiriart2013convexI}. Since $\matr{Q}\bd^{s} = \boldsymbol{0}$, we have $\nabla f_{0}(\by(\tau)) = \nabla_{\bx}f(\bx(\tau)) + \bd^{s}$. With $\bn(\bx) = -\nabla_{\bx}f(\bx) - \bd(\bx) \in N_{C}(\bx)$ as in Section~\ref{subsec:computing_nnls}, the right derivative of $\by$ is therefore
\[
\dot{\by}(\tau) = \bd(\bx(\tau)) - \bd^{s} = -\nabla f_{0}(\by(\tau)) - \bn(\bx(\tau)) \in -\partial V_{0}(\by(\tau)).
\]
The function $\tau \mapsto \by(\tau_{0} + \tau)$ is Lipschitz continuous, so it solves~\eqref{eq:main_grad_incl} for $V_{0}$ with initial value $\by(\tau_{0}) \in C_{0}$. It is the slow solution of that problem. Indeed, two Lipschitz continuous solutions with the same initial value coincide because the monotonicity of $\partial V_{0}$ makes the squared distance between them nonincreasing. Theorem~\ref{thm:MSP}(i) then makes the norm of its right derivative $\bd(\bx(\tau)) - \bd^{s}$ nonincreasing. Since $V_{0}$ attains its minimum by Lemma~\ref{lem:ds}(iii), Theorem~\ref{thm:MSP}(ii) shows that $\by(\tau)$ converges to a minimizer $\bar{\by}$ of $V_{0}$. If $\bd^{s} = \boldsymbol{0}$, then $I_{0} = [m]$, $C_{0} = C$, $V_{0} = V$, and $\by = \bx$, which gives the last claim.
\end{proof}

\subsubsection{No accumulation of breakpoints at a finite time}
\label{app:no_finite_accum}
\begin{prop}
\label{prop:no_finite_accum}
The breakpoints of the slow solution do not accumulate at a finite time, in the sense that if there are infinitely many, then $\tau^{(k)} \to +\infty$.
\end{prop}
Suppose, for contradiction, that there are infinitely many breakpoints and that $\lim_{k \to +\infty} \tau^{(k)} = \bar{\tau} < +\infty$. We fix the events and the tail as in Appendix~\ref{app:finite_general_setting} and write $\sigma(\tau) \coloneqq \bar{\tau} - \tau$.

\begin{lem}
\label{lem:limit_velocity_flow}
In this setting, the following holds.
\begin{itemize}
\item[(i)] The limit $\bar{\bx} \coloneqq \lim_{\tau \uparrow \bar{\tau}}\bx(\tau)$ exists. Every recurring active set $J$ satisfies $\matr{A}_{J}\bar{\bx} = \bb_{J}$ and, with $\bar{\bd}_{J} \coloneqq -\matr{P}^{\perp}_{J}\nabla_{\bx}f(\bar{\bx})$,
\begin{equation}
\label{eq:pin_flow}
\bd(\bx(\tau)) = \bar{\bd}_{J} - \matr{S}_{J}(\bx(\tau) - \bar{\bx})
\quad \text{on every $J$-piece of the tail.}
\end{equation}
\item[(ii)] The vectors $\bar{\bd}_{J}$ of the recurring active sets coincide. Their common value $\bar{\bv}$ satisfies $\bd(\bx(\tau)) \to \bar{\bv}$ as $\tau \uparrow \bar{\tau}$ and $\bar{\bv} \in \ker\matr{A}_{J}$ for every recurring $J$.
\end{itemize}
\end{lem}
\begin{proof}
The slow solution is Lipschitz continuous by Theorem~\ref{thm:MSP}, so $\bx(\tau)$ has a limit $\bar{\bx}$ as $\tau \uparrow \bar{\tau}$, which lies in the closed set $C$. Let $J$ be recurring. Its constraints are active on the $J$-pieces by Lemma~\ref{lem:pieces_flow}(i), and these pieces contain times arbitrarily close to $\bar{\tau}$, so $\matr{A}_{J}\bar{\bx} = \bb_{J}$ and $\bx(\tau) - \bar{\bx} \in \ker\matr{A}_{J}$ on every $J$-piece. Since $\nabla_{\bx}f$ is affine with slope $\matr{Q}$, formula~\eqref{eq:piece_flow} gives $\bd(\bx(\tau)) = \bar{\bd}_{J} - \matr{P}^{\perp}_{J}\matr{Q}(\bx(\tau) - \bar{\bx}) = \bar{\bd}_{J} - \matr{S}_{J}(\bx(\tau) - \bar{\bx})$ there. This proves (i).

Let $\tau^{(k+1)}$ be a breakpoint of the tail, let $J = J_{*}^{(k)}$, and let $\bd^{-} \coloneqq -\matr{P}^{\perp}_{J}\nabla_{\bx}f(\bx^{(k+1)})$ be the limit of $\bd(\bx(\tau))$ as $\tau \uparrow \tau^{(k+1)}$. By Lemma~\ref{lem:pieces_flow}(i), $\bn^{-} \coloneqq -\nabla_{\bx}f(\bx^{(k+1)}) - \bd^{-}$ lies in $N^{(k)}$, and $\matr{A}_{J \cup R^{(k)}}\bd^{-} = \boldsymbol{0}$. By Lemma~\ref{lem:row_bases}, some $B \subseteq J \cup R^{(k)}$ whose rows form a basis of the row space of $\matr{A}_{J \cup R^{(k)}}$ satisfies $(\matr{A}_{B}\matr{A}_{B}^{\top})^{-1}\matr{A}_{B}\bn^{-} \geqslant \boldsymbol{0}$. Since $\matr{A}_{B}\bd^{-} = \boldsymbol{0}$, this vector is $\bp_{B}(\bx^{(k+1)})$, and $\bd^{-} = -\matr{P}^{\perp}_{B}\nabla_{\bx}f(\bx^{(k+1)})$, because $\bd^{-} \in \ker\matr{A}_{B}$ and $\bn^{-} \in \image{\matr{A}_{B}^{\top}}$. Lemma~\ref{lem:velocity_identity} with $B$ for $J$ then gives $\bd^{-} = \proj_{\mathcal{K}}(-\nabla_{\bx}f(\bx^{(k+1)}))$ for $\mathcal{K} \coloneqq \{\bs \in \Rn : \matr{A}_{B}\bs \leqslant \boldsymbol{0}\}$. The constraints of $B$ are active at $\bx^{(k+1)}$, so $T_{C}(\bx^{(k+1)}) \subseteq \mathcal{K}$, and Lemma~\ref{lem:nested_cones} gives
\[
\normsq{\bd^{(k+1)} - \bd^{-}} \leqslant \normsq{\bd^{-}} - \normsq{\bd^{(k+1)}} \leqslant \normsq{\bd^{(k)}} - \normsq{\bd^{(k+1)}},
\]
where the second inequality holds because $\normtwo{\bd^{-}} = \lim_{\tau \uparrow \tau^{(k+1)}}\normtwo{\bd(\bx(\tau))} \leqslant \normtwo{\bd^{(k)}}$, the function $\tau \mapsto \normtwo{\dot{\bx}(\tau)} = \normtwo{\bd(\bx(\tau))}$ being nonincreasing by Theorem~\ref{thm:MSP}(i) and~(iii). Similarly, the sequence $\normtwo{\bd^{(k)}}$ converges, so the right-hand side tends to $0$ as $k \to \infty$. By~\eqref{eq:pin_flow} and continuity, $\bd^{-} = \bar{\bd}_{J_{*}^{(k)}} - \matr{S}_{J_{*}^{(k)}}(\bx^{(k+1)} - \bar{\bx})$ and $\bd^{(k+1)} = \bar{\bd}_{J_{*}^{(k+1)}} - \matr{S}_{J_{*}^{(k+1)}}(\bx^{(k+1)} - \bar{\bx})$, so $\normtwo{\bar{\bd}_{J_{*}^{(k+1)}} - \bar{\bd}_{J_{*}^{(k)}}} \leqslant \normtwo{\bd^{(k+1)} - \bd^{-}} + 2\lambda_{\max}\normtwo{\bx^{(k+1)} - \bar{\bx}} \to 0$. The vectors $\bar{\bd}_{J}$ of the recurring active sets are finitely many, so $\bar{\bd}_{J_{*}^{(k+1)}} = \bar{\bd}_{J_{*}^{(k)}}$ for every large $k$: either they all coincide, or the bound eventually falls below their least pairwise distance. Hence $\bar{\bd}_{J_{*}^{(k)}}$ converges to some constant $\bar{\bv}$. Every recurring active set occurs at arbitrarily large $k$, so its vector $\bar{\bd}_{J}$ equals $\bar{\bv}$. Then~\eqref{eq:pin_flow} gives $\bd(\bx(\tau)) \to \bar{\bv}$ as $\tau \uparrow \bar{\tau}$, and the directions on the $J$-pieces lie in the closed subspace $\ker\matr{A}_{J}$, so $\bar{\bv} \in \ker\matr{A}_{J}$.
\end{proof}

We measure the deviation of the trajectory from the affine path $\tau \mapsto \bar{\bx} - \sigma(\tau)\bar{\bv}$ via
\begin{equation}
\label{eq:zeta_flow}
\bm{\zeta}(\tau) \coloneqq \bx(\tau) - \bar{\bx} + \sigma(\tau)\bar{\bv}.
\end{equation}
The slow solution is Lipschitz continuous with $\dot{\bx} = \bd(\bx)$, so $\bx(\tau) - \bar{\bx} = -\int_{\tau}^{\bar{\tau}}\bd(\bx(\rho))\diff\rho$ and
\begin{equation}
\label{eq:zeta_small_flow}
\bm{\zeta}(\tau) = -\int_{\tau}^{\bar{\tau}}\bigl(\bd(\bx(\rho)) - \bar{\bv}\bigr)\diff\rho,
\qquad
\normtwo{\bm{\zeta}(\tau)} \leqslant \sigma(\tau)\sup_{\rho \in [\tau,\bar{\tau})}\normtwo{\bd(\bx(\rho)) - \bar{\bv}}.
\end{equation}
By Lemma~\ref{lem:limit_velocity_flow}(ii), the supremum tends to $0$ as $\tau \uparrow \bar{\tau}$.

\begin{lem}
\label{lem:deviation_flow}
In this setting, the following holds.
\begin{itemize}
\item[(i)] On every $J$-piece of the tail, $\dot{\bm{\zeta}} = \sigma\matr{S}_{J}\bar{\bv} - \matr{S}_{J}\bm{\zeta}$, and $\matr{A}_{J}\bm{\zeta} = \boldsymbol{0}$ on the piece and at both of its endpoints.
\item[(ii)] Let $\phi(\bx) = \phi(\bar{\bx}) + \langle\ba,\bx - \bar{\bx}\rangle$ be an affine function that vanishes at $\bx^{(k)}$ for infinitely many $k$. Then $\phi(\bar{\bx}) = 0$, $\langle\ba,\bar{\bv}\rangle = 0$, and $\phi(\bx(\tau)) = \langle\ba,\bm{\zeta}(\tau)\rangle$ for every $\tau$ on the tail.
\end{itemize}
\end{lem}
\begin{proof}
By Lemma~\ref{lem:limit_velocity_flow}, $\matr{A}_{J}(\bx(\tau) - \bar{\bx}) = \boldsymbol{0}$ on the piece and at its endpoints, and $\matr{A}_{J}\bar{\bv} = \boldsymbol{0}$, which gives the second claim. For the first, \eqref{eq:pin_flow} with $\bar{\bd}_{J} = \bar{\bv}$ and~\eqref{eq:zeta_flow} give $\dot{\bx} = \bar{\bv} - \matr{S}_{J}(\bm{\zeta} - \sigma\bar{\bv}) = \bar{\bv} + \sigma\matr{S}_{J}\bar{\bv} - \matr{S}_{J}\bm{\zeta}$, and $\dot{\bm{\zeta}} = \dot{\bx} - \bar{\bv}$ because $\dot{\sigma} = -1$. This proves (i).

The iterates $\bx^{(k)}$ converge to $\bar{\bx}$, so $\phi(\bar{\bx}) = 0$ by continuity. At the iterates where $\phi$ vanishes, \eqref{eq:zeta_flow} gives $0 = \langle\ba,\bx^{(k)} - \bar{\bx}\rangle = -\sigma(\tau^{(k)})\langle\ba,\bar{\bv}\rangle + \langle\ba,\bm{\zeta}(\tau^{(k)})\rangle$. Dividing by $\sigma(\tau^{(k)})$ and letting $k \to \infty$ gives $\langle\ba,\bar{\bv}\rangle = 0$, by~\eqref{eq:zeta_small_flow}. Hence $\phi(\bx(\tau)) = \langle\ba,\bm{\zeta}(\tau) - \sigma(\tau)\bar{\bv}\rangle = \langle\ba,\bm{\zeta}(\tau)\rangle$. This proves (ii).
\end{proof}

\begin{proof}[Proof of Proposition~\ref{prop:no_finite_accum}]
We apply the lemmas of Appendix~\ref{app:shared_objects} to the pieces of the tail, relabeled from $0$, with the recurring active sets as the family $\mathcal{J}$. The hypotheses of Lemma~\ref{lem:alignment_shared} hold for $\bm{\zeta}$ and $\bar{\bv}$. Indeed, the limit velocity lies in $\ker\matr{A}_{J}$ for every recurring $J$ by Lemma~\ref{lem:limit_velocity_flow}(ii), the deviation is continuous and bounded by~\eqref{eq:zeta_small_flow}, and Lemma~\ref{lem:deviation_flow}(i) gives $\matr{A}_{J(\tau)}\bm{\zeta}(\tau) = \boldsymbol{0}$ together with $\dot{\bm{\zeta}} = \sigma\matr{S}_{J}\bar{\bv} - \matr{S}_{J}\bm{\zeta}$ on every $J$-piece. The integrand of that equation is bounded and $[\tau,\tau']$ meets finitely many pieces for $\tau < \tau' < \bar{\tau}$, so $\bm{\zeta}(\tau) - \bm{\zeta}(\tau') = -\int_{\tau}^{\tau'}\dot{\bm{\zeta}}(\rho)\diff\rho$; letting $\tau' \uparrow \bar{\tau}$, with $\bm{\zeta}(\tau') \to \boldsymbol{0}$ by~\eqref{eq:zeta_small_flow}, gives~\eqref{eq:driven_shared} with $s = \sigma$ and $\bh = \matr{Q}\bar{\bv}$, whose drifts are $\matr{P}^{\perp}_{J}\matr{Q}\bar{\bv} = \matr{S}_{J}\bar{\bv}$. Lemmas~\ref{lem:alignment_shared} and~\ref{lem:sign_shared} therefore apply to $\bm{\zeta}$. Every recurring event value vanishes at the iterates of its own events, so Lemma~\ref{lem:deviation_flow}(ii) gives $\phi(\bx(\tau)) = \langle\ba,\bm{\zeta}(\tau)\rangle$ at every $\tau$ on the tail, for every recurring event value $\phi$ and its gradient $\ba \in \Phi$. Finally, by Lemma~\ref{lem:sign_shared}, every $\ba \in \Phi$ admits a time $\tau_{\ba} < \bar{\tau}$ such that $\tau \mapsto \langle\ba,\bm{\zeta}(\tau)\rangle$ either vanishes on all of $[\tau_{\ba},\bar{\tau})$ or has no zero there. By Lemma~\ref{lem:events_flow} with $\bm{D} = \bm{\zeta}$, no piece of the tail starts after the largest of these times, which contradicts $\tau^{(k)} \uparrow \bar{\tau}$.
\end{proof}

\subsubsection{No accumulation of breakpoints at infinity}
\label{app:no_infinite_accum}
\begin{prop}
\label{prop:no_infinite_accum}
The slow solution has finitely many breakpoints.
\end{prop}
We prove this below. Write $\by(\tau) = \bx(\tau) - \tau\bd^{s}$, with $\tau_{0}$ and $\bar{\by}$ as in Proposition~\ref{prop:limit_direction}, and let
\[
\be(\tau) \coloneqq \by(\tau) - \bar{\by},
\]
so that $\be(\tau) \to \boldsymbol{0}$ as $\tau \to +\infty$, and the right derivative of $\be$ is $\dot{\be}(\tau) = \bd(\bx(\tau)) - \bd^{s}$.

\begin{lem}
\label{lem:late_flow}
Let $J$ be an index set whose $J$-pieces contain arbitrarily large times. Then, at every time $\tau$ of every $J$-piece,
\[
\be(\tau) \in \ker\matr{A}_{J}
\qquad \text{and} \qquad
\dot{\be}(\tau) = -\matr{P}^{\perp}_{J}\matr{Q}\be(\tau) = -\matr{S}_{J}\be(\tau).
\]
\end{lem}
\begin{proof}
By Lemma~\ref{lem:pieces_flow}(i) and Proposition~\ref{prop:limit_direction}(ii), at the times $\tau \geqslant \tau_{0}$ of the $J$-pieces we have $J \subseteq \eqset(\bx(\tau)) \subseteq I_{0}$ and $[\matr{A}\bx(\tau)]_{J} = \bb_{J}$. Since $J \subseteq I_{0}$, we have $[\matr{A}\bd^{s}]_{J} = \boldsymbol{0}$, so $[\matr{A}\by(\tau)]_{J} = [\matr{A}\bx(\tau)]_{J} = \bb_{J}$ at every time of every $J$-piece. Letting $\tau \to +\infty$ gives $[\matr{A}\bar{\by}]_{J} = \bb_{J}$, hence $\be(\tau) \in \ker\matr{A}_{J}$ there.

Since $\matr{Q}\bd^{s} = \boldsymbol{0}$, we have $\nabla_{\bx}f(\bx(\tau)) = \nabla_{\bx}f(\by(\tau)) \to \nabla_{\bx}f(\bar{\by})$ as $\tau \to +\infty$. Along the $J$-pieces, $\bd(\bx(\tau)) = -\matr{P}^{\perp}_{J}\nabla_{\bx}f(\by(\tau))$, by~\eqref{eq:piece_flow}, converges to $-\matr{P}^{\perp}_{J}\nabla_{\bx}f(\bar{\by})$, and to $\bd^{s}$ by Proposition~\ref{prop:limit_direction}(i), so $\bd^{s} = -\matr{P}^{\perp}_{J}\nabla_{\bx}f(\bar{\by})$. Hence, at every time $\tau$ of every $J$-piece, $\dot{\be}(\tau) = -\matr{P}^{\perp}_{J}\bigl(\nabla_{\bx}f(\by(\tau)) - \nabla_{\bx}f(\bar{\by})\bigr) = -\matr{P}^{\perp}_{J}\matr{Q}\be(\tau)$, and $\matr{P}^{\perp}_{J}\matr{Q}\be(\tau) = \matr{S}_{J}\be(\tau)$ because $\be(\tau) \in \ker\matr{A}_{J}$.
\end{proof}

Suppose there are infinitely many breakpoints, so that $\tau^{(k)} \to +\infty$ by Proposition~\ref{prop:no_finite_accum}. We fix the events and the tail as in Appendix~\ref{app:finite_general_setting}, discarding further pieces so that every piece of the tail starts after $\tau_{0}$. Every recurring active set then satisfies the hypothesis of Lemma~\ref{lem:late_flow}, and $\tau \mapsto \normtwo{\dot{\be}(\tau)}$ is nonincreasing on the tail by Proposition~\ref{prop:limit_direction}(iii).

\begin{lem}
\label{lem:direction_flow}
Let $W \subseteq \Rn$ be a subspace, let $\matr{H} \in \R^{n \times n}$ be symmetric, and for recurring $J$ let $\matr{P}^{W}_{J}$ denote the orthogonal projection onto $W \cap \ker\matr{A}_{J}$. Let $\tau \mapsto \bxi(\tau)$ be defined and continuous on the tail, continuously differentiable on the interior of each piece and from the right at its left endpoint, and nonzero at every time of the tail. Suppose that on every $J$-piece of the tail
\begin{equation}
\label{eq:level_flow}
\bxi(\tau) \in W \cap \ker\matr{A}_{J}
\qquad \text{and} \qquad
\dot{\bxi} = -\matr{P}^{W}_{J}\matr{H}\bxi,
\end{equation}
and that at every breakpoint $\tau^{(k)}$ of the tail
\begin{equation}
\label{eq:jump_flow}
\normtwo{\dot{\bxi}(\tau^{(k)})} \leqslant \lim_{\tau \uparrow \tau^{(k)}}\normtwo{\dot{\bxi}(\tau)}.
\end{equation}
Then:
\begin{itemize}
\item[(i)] The unit vector $\btheta(\tau) \coloneqq \bxi(\tau)/\normtwo{\bxi(\tau)}$ converges as $\tau \to +\infty$ to a unit vector $\bar{\btheta}$, and there exists $\mu \in \R$ such that $\bar{\btheta} \in W \cap \ker\matr{A}_{J}$ and $\matr{P}^{W}_{J}\matr{H}\bar{\btheta} = \mu\bar{\btheta}$ for every recurring $J$.
\item[(ii)] From some time on, $\gamma(\tau) \coloneqq \langle\bxi(\tau),\bar{\btheta}\rangle$ is positive, and from that time the function $\bxi^{+} \coloneqq \bxi/\gamma - \bar{\btheta}$ is continuous, continuously differentiable on each piece, takes its values in $W^{+} \coloneqq \{\bs \in W : \langle\bs,\bar{\btheta}\rangle = 0\}$, tends to $\boldsymbol{0}$, and satisfies~\eqref{eq:level_flow} and~\eqref{eq:jump_flow} with $W^{+}$ and $\matr{H} - \mu\matr{I}_{n\times n}$ for $W$ and $\matr{H}$.
\end{itemize}
\end{lem}
\begin{proof}
For recurring $J$, let $\matr{H}_{J}$ denote the restriction of $\matr{P}^{W}_{J}\matr{H}$ to $W \cap \ker\matr{A}_{J}$. It is a symmetric linear map from that subspace into itself, with $\normtwo{\matr{H}_{J}} \leqslant \normtwo{\matr{H}}$ and $\langle\bs,\matr{H}_{J}\bs\rangle = \langle\bs,\matr{H}\bs\rangle$ there. Let $\varrho \coloneqq \langle\btheta,\matr{H}\btheta\rangle$ and $\psi \coloneqq \normsq{\dot{\btheta}}$, the dot denoting a right derivative at the left endpoint of a piece. On each $J$-piece, \eqref{eq:level_flow} is a linear differential equation with constant coefficients, so $\bxi$ is real analytic there, and so are $\btheta$, $\varrho$, and $\psi$, because $\bxi$ does not vanish.

\paragraph{Step~1.} We show that $\varrho$ converges to some $\mu \in \R$ and that $\psi(\tau) \to 0$ as $\tau \to +\infty$. On a $J$-piece, \eqref{eq:level_flow} gives $\langle\bxi,\dot{\bxi}\rangle = -\langle\bxi,\matr{H}\bxi\rangle$, so $\tfrac{\diff}{\diff\tau}\normtwo{\bxi} = -\varrho\normtwo{\bxi}$ and the vector of $W \cap \ker\matr{A}_{J}$
\begin{equation}
\label{eq:udot_flow}
\dot{\btheta} = -\matr{H}_{J}\btheta + \varrho\btheta
\end{equation}
is orthogonal to $\btheta$. Hence $\dot{\varrho} = 2\langle\matr{H}\btheta,\dot{\btheta}\rangle = 2\langle\matr{H}_{J}\btheta,\dot{\btheta}\rangle = 2\langle\varrho\btheta - \dot{\btheta},\dot{\btheta}\rangle = -2\psi$, and
\begin{equation}
\label{eq:psi_flow}
\psi = \normsq{\matr{H}_{J}\btheta} - \varrho^{2} = \normsq{\dot{\bxi}}/\normsq{\bxi} - \varrho^{2}.
\end{equation}
Differentiating~\eqref{eq:udot_flow} gives $\ddot{\btheta} = -\matr{H}_{J}\dot{\btheta} + \dot{\varrho}\btheta + \varrho\dot{\btheta}$, so $\dot{\psi} = 2\langle\dot{\btheta},\ddot{\btheta}\rangle = -2\langle\dot{\btheta},\matr{H}_{J}\dot{\btheta}\rangle + 2\varrho\psi \leqslant 4\normtwo{\matr{H}}\psi$ on the piece. At a breakpoint, $\bxi$ and $\varrho$ are continuous and $\normtwo{\dot{\bxi}}$ does not increase, by~\eqref{eq:jump_flow}, so $\psi$ does not increase, by~\eqref{eq:psi_flow}. Hence $\tau \mapsto e^{-4\normtwo{\matr{H}}\tau}\psi(\tau)$ is nonincreasing on the tail. The function $\varrho$ is continuous, nonincreasing, and bounded below by $-\normtwo{\matr{H}}$, so it converges to some $\mu \in \R$, and $\int_{\tau_{1}}^{+\infty}\psi(\rho)\diff\rho = \tfrac12(\varrho(\tau_{1}) - \mu) < +\infty$ for every $\tau_{1}$ on the tail. For $\rho \in [\tau - 1,\tau]$ we have $\psi(\rho) \geqslant e^{-4\normtwo{\matr{H}}}\psi(\tau)$, so $\psi(\tau) \leqslant e^{4\normtwo{\matr{H}}}\int_{\tau-1}^{\tau}\psi(\rho)\diff\rho \to 0$.

\paragraph{Step~2.} We show that $0 \leqslant \varrho - \mu \leqslant \kappa^{2}\psi$ on $[\tau_{1},+\infty)$, for some $\kappa > 0$ and some time $\tau_{1}$ on the tail. Let $J$ be recurring. By~\eqref{eq:udot_flow}, $\normtwo{\dot{\btheta}} = \normtwo{(\matr{H}_{J} - \varrho)\btheta}$ is at least the distance from $\varrho$ to the spectrum of $\matr{H}_{J}$. Since $\psi \to 0$, $\varrho \to \mu$, and the $J$-pieces contain arbitrarily large times, $\mu$ is an eigenvalue of $\matr{H}_{J}$. There are finitely many recurring active sets, so we can fix $\delta > 0$ smaller than the distance from $\mu$ to the other eigenvalues of every $\matr{H}_{J}$ with $J$ recurring. If $\mu$ is the only eigenvalue of every such $\matr{H}_{J}$, we set $\delta \coloneqq 1$. Let $\tau_{1}$ be a time on the tail after which $\varrho - \mu \leqslant \delta/2$, consider a time $\tau \geqslant \tau_{1}$ of a $J$-piece, and let $\btheta_{\mu}$ be the orthogonal projection of $\btheta(\tau)$ onto the $\mu$-eigenspace of $\matr{H}_{J}$. The other eigenvalues of $\matr{H}_{J}$ lie at distance at least $\delta/2$ from $\varrho$, so expanding $\btheta$ in an orthonormal eigenbasis of $\matr{H}_{J}$ gives
\[
\psi = \normsq{(\matr{H}_{J} - \varrho)\btheta} \geqslant \tfrac{\delta^{2}}{4}\normsq{\btheta - \btheta_{\mu}},
\qquad
\varrho - \mu = \langle\btheta - \btheta_{\mu},(\matr{H}_{J} - \mu)(\btheta - \btheta_{\mu})\rangle \leqslant 2\normtwo{\matr{H}}\normsq{\btheta - \btheta_{\mu}}.
\]
The claim holds with $\kappa^{2} \coloneqq 8(1 + \normtwo{\matr{H}})/\delta^{2}$, and $\varrho \geqslant \mu$ because $\varrho$ is nonincreasing with limit $\mu$.

\paragraph{Step~3.} We prove (i). If $\varrho(\tau_{2}) = \mu$ at some time $\tau_{2} \geqslant \tau_{1}$, then $\varrho = \mu$ on $[\tau_{2},+\infty)$, so $\psi = -\dot{\varrho}/2$ vanishes there and $\btheta$ is constant. Otherwise $\varrho > \mu$ on the tail, and on each piece after $\tau_{1}$, Step~2 gives $\tfrac{\diff}{\diff\tau}\sqrt{\varrho - \mu} = -\psi/\sqrt{\varrho - \mu} \leqslant -\sqrt{\psi}/\kappa = -\normtwo{\dot{\btheta}}/\kappa$. The function $\sqrt{\varrho - \mu}$ is continuous, so integrating gives $\int_{\tau_{1}}^{+\infty}\normtwo{\dot{\btheta}(\rho)}\diff\rho \leqslant \kappa\sqrt{\varrho(\tau_{1}) - \mu}$. In both cases, $\btheta(\tau)$ converges to a unit vector $\bar{\btheta}$. For recurring $J$, the directions on the $J$-pieces lie in the closed subspace $W \cap \ker\matr{A}_{J}$, so $\bar{\btheta}$ does, and along the $J$-pieces $(\matr{H}_{J} - \varrho)\btheta = -\dot{\btheta} \to \boldsymbol{0}$ by Step~1. Hence $\matr{H}_{J}\bar{\btheta} = \mu\bar{\btheta}$, that is, $\matr{P}^{W}_{J}\matr{H}\bar{\btheta} = \mu\bar{\btheta}$.

\paragraph{Step~4.} We prove (ii). Since $\langle\btheta,\bar{\btheta}\rangle \to 1$, the function $\gamma = \normtwo{\bxi}\langle\btheta,\bar{\btheta}\rangle$ is positive from some time on. From that time on, $\bxi^{+} = \btheta/\langle\btheta,\bar{\btheta}\rangle - \bar{\btheta}$ is continuous, continuously differentiable on each piece, orthogonal to $\bar{\btheta}$, and tends to $\boldsymbol{0}$, and on every $J$-piece it lies in $W^{+} \cap \ker\matr{A}_{J}$, since $\bxi$ and $\bar{\btheta}$ lie in $W \cap \ker\matr{A}_{J}$. Consider a $J$-piece and write $\matr{P} \coloneqq \matr{P}^{W}_{J}$. Both $\matr{P}$ and $\matr{H}$ are symmetric and $\matr{P}\bs = \bs$ for $\bs \in \{\bxi,\bxi^{+},\bar{\btheta}\}$, so (i) gives
\begin{equation}
\label{eq:sym_flow}
\langle\matr{P}\matr{H}\bs,\bar{\btheta}\rangle = \langle\bs,\matr{P}\matr{H}\bar{\btheta}\rangle = \mu\langle\bs,\bar{\btheta}\rangle
\qquad \text{for $\bs \in \{\bxi,\bxi^{+}\}$.}
\end{equation}
With $\bs = \bxi$, this identity and~\eqref{eq:level_flow} give $\dot{\gamma} = -\mu\gamma$, and then
\[
\dot{\bxi}^{+} = (\dot{\bxi} - \dot{\gamma}\,\bxi)/\gamma = -\matr{P}\matr{H}(\bar{\btheta} + \bxi^{+}) + \mu(\bar{\btheta} + \bxi^{+}) = -(\matr{P}\matr{H} - \mu)\bxi^{+}.
\]
With $\bs = \bxi^{+}$, identity~\eqref{eq:sym_flow} shows that $\matr{P}\matr{H}\bxi^{+}$ is orthogonal to $\bar{\btheta}$, so it lies in $W^{+} \cap \ker\matr{A}_{J}$. Since this subspace lies in $W \cap \ker\matr{A}_{J}$, we have $\matr{P}^{W^{+}}_{J} = \matr{P}^{W^{+}}_{J}\matr{P}$, and therefore $\matr{P}^{W^{+}}_{J}\matr{H}\bxi^{+} = \matr{P}^{W^{+}}_{J}\matr{P}\matr{H}\bxi^{+} = \matr{P}\matr{H}\bxi^{+}$. Hence $\dot{\bxi}^{+} = -\matr{P}^{W^{+}}_{J}(\matr{H} - \mu\matr{I}_{n\times n})\bxi^{+}$, which is~\eqref{eq:level_flow} for $\bxi^{+}$.

It remains to prove~\eqref{eq:jump_flow} for $\bxi^{+}$. On a piece, $\dot{\bxi} = \dot{\gamma}(\bar{\btheta} + \bxi^{+}) + \gamma\dot{\bxi}^{+} = \gamma\bigl(\dot{\bxi}^{+} - \mu(\bar{\btheta} + \bxi^{+})\bigr)$, where $\dot{\bxi}^{+}$ is orthogonal to $\bar{\btheta}$ and $\langle\dot{\bxi}^{+},\bxi^{+}\rangle = -\langle\bxi^{+},\matr{H}\bxi^{+}\rangle + \mu\normsq{\bxi^{+}}$. Therefore
\[
\normsq{\dot{\bxi}} = \gamma^{2}\Bigl(\normsq{\dot{\bxi}^{+}} + 2\mu\langle\bxi^{+},\matr{H}\bxi^{+}\rangle - 2\mu^{2}\normsq{\bxi^{+}} + \mu^{2}\normsq{\bar{\btheta} + \bxi^{+}}\Bigr).
\]
At a breakpoint, $\gamma$ and $\bxi^{+}$ are continuous, so every term on the right-hand side except $\normsq{\dot{\bxi}^{+}}$ is continuous, and $\normtwo{\dot{\bxi}^{+}}$ does not increase because $\normtwo{\dot{\bxi}}$ does not.
\end{proof}

\begin{proof}[Proof of Proposition~\ref{prop:no_infinite_accum}]
We now derive a contradiction in three steps.

\paragraph{Step~1.} We show that there exist an integer $r$ with $0 \leqslant r \leqslant n$, unit vectors $\bar{\btheta}_{1},\dots,\bar{\btheta}_{r}$, a time $\tau'$ on the tail, and functions $\gamma_{0},\dots,\gamma_{r-1}$, positive on $[\tau',+\infty)$, such that
\begin{equation}
\label{eq:expansion_flow}
\be(\tau) = \sum_{j=1}^{r}\gamma_{0}(\tau)\cdots\gamma_{j-1}(\tau)\,\bar{\btheta}_{j} \quad \text{for every $\tau \geqslant \tau'$,}
\qquad
\gamma_{j}(\tau) \longrightarrow 0 \quad \text{as $\tau \to +\infty$, for $1 \leqslant j < r$.}
\end{equation}
For $r = 0$ the sum is empty and $\be$ vanishes on $[\tau',+\infty)$. Let $\bxi$ satisfy~\eqref{eq:level_flow}. If $\bxi$ vanishes at a time $\tau_{1}$, then it vanishes on $[\tau_{1},+\infty)$. Indeed, equation~\eqref{eq:level_flow} is linear on each piece, so $\bxi$ vanishes on the rest of the piece containing $\tau_{1}$ and, by continuity, at the start of the next piece. Hence $\bxi$ either vanishes from some time on or vanishes at no time.

Set $\bxi_{0} \coloneqq \be$ and $W_{0} \coloneqq \Rn$. Then $\matr{P}^{W_{0}}_{J} = \matr{P}^{\perp}_{J}$, so $\bxi_{0}$ satisfies~\eqref{eq:level_flow} with $\matr{H} = \matr{Q}$ by Lemma~\ref{lem:late_flow}, and it satisfies~\eqref{eq:jump_flow} because $\normtwo{\dot{\be}}$ is nonincreasing on the tail. If $\bxi_{0}$ vanishes from some time on, the claim holds with $r = 0$. Otherwise Lemma~\ref{lem:direction_flow} gives a unit vector $\bar{\btheta}_{1}$, a number $\mu_{1}$, a function $\gamma_{0}$, positive from some time on, and the function $\bxi_{1} \coloneqq \bxi_{0}^{+}$, with values in the $(n-1)$-dimensional subspace $W_{1}$ orthogonal to $\bar{\btheta}_{1}$, satisfying~\eqref{eq:level_flow} and~\eqref{eq:jump_flow} with $W_{1}$ and $\matr{Q} - \mu_{1}\matr{I}_{n\times n}$. Discarding the pieces that start before $\gamma_{0}$ becomes positive, we repeat the argument with $\bxi_{1}$. After $j$ repetitions,
\[
\be = \gamma_{0}\Bigl(\bar{\btheta}_{1} + \gamma_{1}\bigl(\bar{\btheta}_{2} + \dots + \gamma_{j-1}(\bar{\btheta}_{j} + \bxi_{j})\bigr)\Bigr),
\]
with $\bxi_{j}$ taking its values in a subspace $W_{j}$ of dimension $n - j$. Since $W_{n} = \{\boldsymbol{0}\}$, some $\bxi_{r}$ with $r \leqslant n$ vanishes from a time $\tau'$ on, and~\eqref{eq:expansion_flow} holds, because $\gamma_{j} = \langle\bxi_{j},\bar{\btheta}_{j+1}\rangle \to 0$, as $\bxi_{j} \to \boldsymbol{0}$ for $j \geqslant 1$ by Lemma~\ref{lem:direction_flow}(ii).

\paragraph{Step~2.} We show that $\phi(\bx(\tau)) = \langle\ba,\be(\tau)\rangle$ at every $\tau$ on the tail, for every recurring event value $\phi$ and its gradient $\ba \in \Phi$. Let $i$ be a constraint that recurs in enter events. At the iterates of these events, $i \in \eqset(\bx^{(k+1)}) \subseteq I_{0}$ by Proposition~\ref{prop:limit_direction}(ii), so $[\matr{A}\bd^{s}]_{i} = 0$ and $[\matr{A}\bx(\tau)]_{i} - b_{i} = [\matr{A}\by(\tau)]_{i} - b_{i}$ for every $\tau$. This affine function of $\by(\tau)$ vanishes at the iterates of the enter events of $i$, whose times tend to infinity, so it vanishes at $\bar{\by}$ and equals $\langle\ba,\be(\tau)\rangle$. Now let $(B,i)$ be a recurring pair of a leave event. Since $\matr{Q}\bd^{s} = \boldsymbol{0}$, formula~\eqref{eq:pJ_flow} gives $[\bp_{B}(\bx(\tau))]_{i} = [\bp_{B}(\by(\tau))]_{i}$, an affine function of $\by(\tau)$ that vanishes at the iterates of the leave events of $(B,i)$, hence at $\bar{\by}$, and equals $\langle\ba,\be(\tau)\rangle$.

\paragraph{Step~3.} Fix $\ba \in \Phi$ and put $\alpha_{j} \coloneqq \langle\ba,\bar{\btheta}_{j}\rangle$. By~\eqref{eq:expansion_flow}, $\langle\ba,\be(\tau)\rangle = \sum_{j=1}^{r}\gamma_{0}(\tau)\cdots\gamma_{j-1}(\tau)\,\alpha_{j}$ for $\tau \geqslant \tau'$. If every $\alpha_{j}$ vanishes, then $\langle\ba,\be\rangle$ vanishes identically on $[\tau',+\infty)$, and we put $\tau_{\ba} \coloneqq \tau'$. Otherwise let $j$ be minimal with $\alpha_{j} \neq 0$. Then
\[
\langle\ba,\be(\tau)\rangle = \gamma_{0}(\tau)\cdots\gamma_{j-1}(\tau)\Bigl(\alpha_{j} + \sum_{l=j+1}^{r}\gamma_{j}(\tau)\cdots\gamma_{l-1}(\tau)\,\alpha_{l}\Bigr),
\]
and $\gamma_{j}\cdots\gamma_{l-1} \to 0$ for $l > j$, so $\langle\ba,\be\rangle$ has no zero on some interval $[\tau_{\ba},+\infty)$. By Lemma~\ref{lem:events_flow} with $\bm{D} = \be$, no piece of the tail starts after the largest of the times $\tau_{\ba}$, $\ba \in \Phi$, which contradicts $\tau^{(k)} \to +\infty$.
\end{proof}

\subsubsection{Conclusion of the proof}
\label{app:finite_general_proof}
\begin{proof}[Proof of Theorem~\ref{thm:finite_general}]
By Proposition~\ref{prop:no_infinite_accum}, the breakpoints are finitely many, so the final piece is $[\tau^{(K)},+\infty)$ and $\tau^{(K+1)} = \tau^{(K)} + \tau_{*}^{(K)} = +\infty$. By Proposition~\ref{prop:limit_direction}(i), $\bd(\bx(\tau))$ converges to the vector $\bd^{s}$ of~\eqref{eq:ds}. 
By~\eqref{eq:piece_flow}, $\bd^{(K)} \in \ker\matr{A}_{J_{*}^{(K)}}$, and by~\eqref{eq:global_solution}, $\bx(\tau) = \bx^{(K)} + \matr{M}_{J_{*}^{(K)}}(\tau - \tau^{(K)})\bd^{(K)}$ on the final piece, where $\dot{\matr{M}}_{J_{*}^{(K)}}(\tau) = e^{-\matr{P}^{\perp}_{J_{*}^{(K)}}\matr{Q}\tau}$ by~\eqref{eq:int_mexp}. The matrices $\matr{P}^{\perp}_{J_{*}^{(K)}}\matr{Q}$ and $\matr{S}_{J_{*}^{(K)}}$ agree on $\ker\matr{A}_{J_{*}^{(K)}}$ and map it into itself. A vector $\bv \in \ker\matr{A}_{J_{*}^{(K)}}$ satisfies $\matr{S}_{J_{*}^{(K)}}\bv = \boldsymbol{0}$ if and only if $\langle\bv,\matr{Q}\bv\rangle = 0$, that is, $\matr{Q}\bv = \boldsymbol{0}$, so $\ker\matr{A}_{J_{*}^{(K)}}$ is the orthogonal sum of $\ker\matr{Q} \cap \ker\matr{A}_{J_{*}^{(K)}}$ and $\image{\matr{S}_{J_{*}^{(K)}}}$, on which $\matr{S}_{J_{*}^{(K)}}$ is positive definite. Then
\[
\bd(\bx(\tau)) = e^{-\matr{S}_{J_{*}^{(K)}}(\tau - \tau^{(K)})}\bd^{(K)} \to \proj_{\ker\matr{Q} \cap \ker\matr{A}_{J_{*}^{(K)}}}(\bd^{(K)}),
\]
so $\bd^{s}$ is this projection. Finally, if $\bd^{s} = \boldsymbol{0}$, then $\bd^{(K)} \in \image{\matr{S}_{J_{*}^{(K)}}}$ and so
\[
\matr{M}_{J_{*}^{(K)}}(\tau)\bd^{(K)} = (\matr{S}_{J_{*}^{(K)}})^{\dagger}\bigl(\matr{I}_{n\times n} - e^{-\matr{S}_{J_{*}^{(K)}}\tau}\bigr)\bd^{(K)}.
\]
Hence $\bx(\tau)$ converges to $\bx^{(K)} + (\matr{S}_{J_{*}^{(K)}})^{\dagger}\bd^{(K)}$, an optimal solution by Proposition~\ref{prop:limit_direction}(iii). If $\bd^{s} \neq \boldsymbol{0}$, then Lemma~\ref{lem:ds}(i) gives (ii).
\end{proof}

\subsection{Proofs for Section~\ref{sec:computations_global}}
\label{app:proofs_descent}
This appendix proves Proposition~\ref{prop:descent_dichotomy} and Theorem~\ref{thm:descent_convergence}. Appendix~\ref{app:descent_setup} collects the objects the proofs need beyond Section~\ref{sec:computations_global}, Appendix~\ref{app:blocked_tail} excludes the one obstruction to convergence, Appendix~\ref{app:proof_descent_psd} proves the proposition, and Appendix~\ref{app:proof_descent_convergence} proves the theorem.

Throughout, $\matr{Q}$ is SPSD with largest eigenvalue $\lambda_{\max}$. We write $\bx^{(k)}$, $\bd^{(k)}$, $\Delta^{(k)}$ for the iterates, directions, and steps of Algorithm~\ref{alg:quad_prog_2}, run with tolerance $\varepsilon = 0$, and we abbreviate $\eqset^{(k)} \coloneqq \eqset(\bx^{(k)})$, $J^{(k)} \coloneqq J(\bx^{(k)})$. We choose every NNLS solution $\bp(\bx)$ below so that~\eqref{eq:LI} holds.
\subsubsection{Setup}
\label{app:descent_setup}
Recall the compressed Hessian $\matr{S}_{J}$ of~\eqref{eq:hessian_compressed} and its properties, collected in Appendix~\ref{app:shared_objects}. It is positive semidefinite on $\ker\matr{A}_{J}$. The direction $\bd(\bx)$ vanishes exactly at the optimal solutions, since $\bd(\bx) = \boldsymbol{0}$ if and only if $\boldsymbol{0} \in \partial_{\bx}V(\bx)$ (Section~\ref{subsec:char_di}). We may therefore assume $\bd^{(k)} \neq \boldsymbol{0}$, since Algorithm~\ref{alg:quad_prog_2} stops at $\bd^{(k)} = \boldsymbol{0}$, its test being $\normtwo{\bd^{(k)}} \leqslant \varepsilon = 0$. Likewise $\Delta^{(k)} < +\infty$ at every step of an infinite iteration, since $\Delta^{(k)} = +\infty$ stops it.

We call step $k$ \emph{blocked} if $\Delta^{(k)} = \Delta_{\mathrm{feas}}(\bx^{(k)}) < \Delta_{*}^{(k)}$, and \emph{full} otherwise. A blocked step ends on a constraint: some $i_{k} \notin \eqset^{(k)}$ attains the minimum in~\eqref{eq:delta_feas}, with
\begin{equation}\label{eq:blocking}
[\matr{A}\bd^{(k)}]_{i_{k}} > 0,
\qquad
[\matr{A}\bx^{(k)}]_{i_{k}} < b_{i_{k}},
\qquad
[\matr{A}\bx^{(k+1)}]_{i_{k}} = b_{i_{k}},
\end{equation}
and we call $i_{k}$ the \emph{blocking wall} of step $k$. The first relation holds because $\Delta_{\mathrm{feas}}(\bx^{(k)}) < \Delta_{*}^{(k)} \leqslant +\infty$, so the minimizing quotient in~\eqref{eq:delta_feas} is finite and positive. Every step has $\Delta^{(k)} > 0$. Indeed, on a blocked step, the first two relations of~\eqref{eq:blocking} give $\Delta_{\mathrm{feas}}(\bx^{(k)}) > 0$, while on a full step $\Delta^{(k)} = \Delta_{*}^{(k)} = \normsq{\bd^{(k)}}/\langle\bd^{(k)},\matr{Q}\bd^{(k)}\rangle > 0$ because $\bd^{(k)} \neq \boldsymbol{0}$. The constraints of $J^{(k)}$ are active at $\bx^{(k)}$ because $J^{(k)} \subseteq \eqset^{(k)}$, and stay active at $\bx^{(k+1)}$ because $[\matr{A}\bd^{(k)}]_{J^{(k)}} = \boldsymbol{0}$ by~\eqref{eq:nnls_oc_2_5}. Hence $J^{(k)} \subseteq \eqset^{(k+1)}$.

\begin{defn}[Blocked tail]
\label{def:blocked_tail}
A \emph{blocked tail} is an infinite sequence of consecutive iterates of Algorithm~\ref{alg:quad_prog_2}, indexed by $k \geqslant 0$ after relabeling, such that every step is blocked, $\inf_{k}\normtwo{\bd^{(k)}} > 0$, and $\sum_{k}\Delta^{(k)} < +\infty$.
\end{defn}
\begin{lem}
\label{lem:tail_extraction}
Suppose the iteration is infinite, that $\normtwo{\bd^{(k)}} \geqslant \theta > 0$ for every $k$, and that $\inf_{k}V(\bx^{(k)}) > -\infty$. Then only finitely many steps have $\Delta^{(k)} \geqslant 1/\lambda_{\max}$, and the steps after the last of them, or all steps if there is none, form a blocked tail.
\end{lem}
\begin{proof}
Since $\Delta^{(k)} \leqslant \Delta_{\mathrm{feas}}(\bx^{(k)})$, the bound~\eqref{eq:sufficient_decrease} gives
\[
V(\bx^{(k)}) - V(\bx^{(k+1)}) \geqslant \tfrac12\min(\Delta^{(k)},1/\lambda_{\max})\theta^{2} \geqslant 0.
\]
Summing over $k$ bounds $\tfrac12\theta^{2}\sum_{k}\min(\Delta^{(k)},1/\lambda_{\max})$ by $V(\bx^{(0)}) - \inf_{k}V(\bx^{(k)})$. Hence finitely many steps have $\Delta^{(k)} \geqslant 1/\lambda_{\max}$, and the others have summable length. A full step has $\Delta^{(k)} = \Delta_{*}^{(k)} \geqslant 1/\lambda_{\max}$, the quotient being finite because $\Delta^{(k)} < +\infty$, so the steps after the last one with $\Delta^{(k)} \geqslant 1/\lambda_{\max}$ are blocked, summable, and have $\normtwo{\bd^{(k)}} \geqslant \theta$.
\end{proof}

\subsubsection{The blocked tail}
\label{app:blocked_tail}
Theorem~\ref{thm:descent_convergence} rests on excluding a single pattern.
\begin{lem}
\label{lem:residual_growth}
Every step of Algorithm~\ref{alg:quad_prog_2} satisfies $\normtwo{\bd^{(k+1)}} \leqslant \bigl(1 + \lambda_{\max}\Delta^{(k)}\bigr)\normtwo{\bd^{(k)}}$. More precisely, with $\tilde{\bd} \coloneqq \proj_{T_{C}(\bx^{(k+1)})}(-\nabla_{\bx}f(\bx^{(k)}))$, we have that
\[
\normsq{\tilde{\bd} - \bd^{(k)}} \leqslant \normsq{\bd^{(k)}} - \normsq{\tilde{\bd}} \quad \text{and} \quad \normtwo{\bd^{(k+1)} - \tilde{\bd}} \leqslant \lambda_{\max}\Delta^{(k)}\normtwo{\bd^{(k)}}.
\]
\end{lem}
\begin{proof}
On every step, full or blocked, $J^{(k)} \subseteq \eqset^{(k+1)}$ as shown in Appendix~\ref{app:descent_setup}, so $T \coloneqq T_{C}(\bx^{(k+1)}) \subseteq \mathcal{K} \coloneqq \{\bs : \matr{A}_{J^{(k)}}\bs \leqslant \boldsymbol{0}\}$, and Lemma~\ref{lem:velocity_identity} gives $\bd^{(k)} = \proj_{\mathcal{K}}(-\nabla_{\bx}f(\bx^{(k)}))$. Set $\tilde{\bd} \coloneqq \proj_{T}(-\nabla_{\bx}f(\bx^{(k)}))$. Lemma~\ref{lem:nested_cones} applied to $T \subseteq \mathcal{K}$ gives the first estimate, hence $\normtwo{\tilde{\bd}} \leqslant \normtwo{\bd^{(k)}}$, and the $1$-Lipschitz property of $\proj_{T}$ gives $\normtwo{\bd^{(k+1)} - \tilde{\bd}} \leqslant \normtwo{\matr{Q}(\bx^{(k+1)} - \bx^{(k)})} \leqslant \lambda_{\max}\Delta^{(k)}\normtwo{\bd^{(k)}}$. Adding the two bounds proves the claim.
\end{proof}

Along a blocked tail, iterating Lemma~\ref{lem:residual_growth} with $1 + t \leqslant e^{t}$ gives
\begin{equation}
\label{eq:tail_dir_bound}
\sup_{l \geqslant 0}\normtwo{\bd^{(l)}} \;\leqslant\; \normtwo{\bd^{(0)}}\exp\Bigl(\lambda_{\max}\sum_{l \geqslant 0}\Delta^{(l)}\Bigr) \;=:\; B_{0} \;<\; +\infty,
\end{equation}
the sum being finite by Definition~\ref{def:blocked_tail}, and $\bd^{(0)}$ being the first direction of the tail. Hence $\normtwo{\bx^{(k+1)} - \bx^{(k)}} \leqslant B_{0}\Delta^{(k)}$, so the iterates are Cauchy and converge to some $\bar{\bx} \in C$.

\begin{thm}
\label{thm:no_blocked_tail}
Algorithm~\ref{alg:quad_prog_2}, run with tolerance $\varepsilon = 0$, has no blocked tails.
\end{thm}
We prove this below. Fix a blocked tail with limit $\bar{\bx}$, step lengths $\Delta^{(k)}$, blocking walls $i_{k}$, and active sets $J^{(k)}$. Both index collections are finite, so we may discard finitely many further steps until every occurring active set and every blocking wall repeats infinitely often. We call these \emph{recurring}.

\paragraph{The limit velocity.} For recurring $J$, the map $\bx \mapsto -\matr{P}^{\perp}_{J}\nabla_{\bx}f(\bx)$ is affine, so the directions along the steps with $J^{(k)} = J$ converge:
\begin{equation}\label{eq:dbar}
\bd^{(k)} = -\matr{P}^{\perp}_{J}\nabla_{\bx}f(\bx^{(k)}) \longrightarrow \bar{\bd}_{J} \coloneqq -\matr{P}^{\perp}_{J}\nabla_{\bx}f(\bar{\bx}),
\qquad
\normtwo{\bd^{(k)} - \bar{\bd}_{J}} \leqslant \lambda_{\max}\normtwo{\bx^{(k)} - \bar{\bx}}
\end{equation}
for every $k$ with $J^{(k)} = J$. The next two results show that these limits coincide.

\begin{prop}
\label{prop:limit_velocity}
Let a blocked tail be given as above. Its directions converge to a single limit $\bar{\bv}$ with $\normtwo{\bar{\bv}} > 0$. Moreover $\bar{\bd}_{J} = \bar{\bv}$ and $\bar{\bv} \in \ker\matr{A}_{J}$ for every recurring active set $J$.
\end{prop}
\begin{proof}
We proceed in four steps.

\paragraph{Step~1.} Set $\tilde{\bd} \coloneqq \proj_{T_{C}(\bx^{(k+1)})}(-\nabla_{\bx}f(\bx^{(k)}))$. Lemma~\ref{lem:residual_growth} gives
\begin{equation}\label{eq:jumpest}
\normsq{\tilde{\bd} - \bd^{(k)}} \;\leqslant\; \normsq{\bd^{(k)}} - \normsq{\tilde{\bd}}
\quad\text{and}\quad
\normtwo{\bd^{(k+1)} - \tilde{\bd}} \leqslant \lambda_{\max}\Delta^{(k)}\normtwo{\bd^{(k)}} \leqslant \lambda_{\max}B_{0}\Delta^{(k)},
\end{equation}
the last inequality by~\eqref{eq:tail_dir_bound}. Hence $\normtwo{\bd^{(k+1)}} \leqslant \normtwo{\bd^{(k)}} + \lambda_{\max}B_{0}\Delta^{(k)}$. Since $\sum_{k}\Delta^{(k)} < +\infty$, the sequence $\normtwo{\bd^{(k)}} + \lambda_{\max}B_{0}\sum_{l \geqslant k}\Delta^{(l)}$ is nonincreasing and bounded below, so $\lim_{k}\normtwo{\bd^{(k)}}$ exists and is at least $\inf_{k}\normtwo{\bd^{(k)}} > 0$.

\paragraph{Step~2.} From $\normtwo{\tilde{\bd}} \geqslant \normtwo{\bd^{(k+1)}} - \lambda_{\max}B_{0}\Delta^{(k)}$ and~\eqref{eq:jumpest}, the case $\normtwo{\bd^{(k+1)}} < \lambda_{\max}B_{0}\Delta^{(k)}$ being trivial,
\[
\normsq{\tilde{\bd} - \bd^{(k)}} \;\leqslant\; \normsq{\bd^{(k)}} - \normsq{\bd^{(k+1)}} + 2\lambda_{\max}B_{0}^{2}\Delta^{(k)} + \bigl(\lambda_{\max}B_{0}\Delta^{(k)}\bigr)^{2},
\]
and the right-hand side tends to $0$ because $\normtwo{\bd^{(k)}}$ converges and $\Delta^{(k)} \to 0$. The triangle inequality now gives $\normtwo{\bd^{(k+1)} - \bd^{(k)}} \to 0$.

\paragraph{Step~3.} By~\eqref{eq:dbar}, $\normtwo{\bd^{(k)} - \bar{\bd}_{J^{(k)}}} \leqslant \lambda_{\max}\normtwo{\bx^{(k)} - \bar{\bx}} \to 0$ uniformly in the active set. The finitely many values $\{\bar{\bd}_{J} : J \ \text{recurring}\}$ lie at least $2\rho > 0$ apart, where $\rho = 1$ if only one value occurs. For $k$ large,
\[
\normtwo{\bar{\bd}_{J^{(k+1)}} - \bar{\bd}_{J^{(k)}}} \leqslant \normtwo{\bd^{(k+1)} - \bd^{(k)}} + \lambda_{\max}\normtwo{\bx^{(k)} - \bar{\bx}} + \lambda_{\max}\normtwo{\bx^{(k+1)} - \bar{\bx}} < 2\rho,
\]
which forces $\bar{\bd}_{J^{(k+1)}} = \bar{\bd}_{J^{(k)}}$. Hence $\bar{\bd}_{J^{(k)}}$ is eventually constant, equal to some $\bar{\bv}$. Every recurring active set attains it, so $\bar{\bd}_{J} = \bar{\bv}$ for all recurring $J$, and $\bd^{(k)} \to \bar{\bv}$ with $\normtwo{\bar{\bv}} > 0$.

\paragraph{Step~4.} For recurring $J$, the directions with $J^{(k)} = J$ lie in the closed subspace $\ker\matr{A}_{J}$ and converge to $\bar{\bv}$, so $\bar{\bv} \in \ker\matr{A}_{J}$.
\end{proof}

\paragraph{Exact reduction.} We place the blocked tail in coordinates adapted to its limit, namely
\begin{equation}\label{eq:tail_coords}
\sigma_{k} \coloneqq \sum_{l \geqslant k}\Delta^{(l)},
\qquad
\bm{\zeta}^{(k)} \coloneqq \bx^{(k)} - \bar{\bx} + \sigma_{k}\bar{\bv},
\end{equation}
so $\sigma_{k} \downarrow 0$, $\bx^{(k)} - \bar{\bx} = -\sum_{l \geqslant k}\Delta^{(l)}\bd^{(l)}$ by telescoping, and $\bm{\zeta}^{(k)} = -\sum_{l \geqslant k}\Delta^{(l)}(\bd^{(l)} - \bar{\bv}) \to \boldsymbol{0}$. Define
\begin{equation}\label{eq:tail_drifts}
\bg \coloneqq \matr{Q}\bar{\bv},
\qquad
\bu_{J} \coloneqq \matr{P}^{\perp}_{J}\bg = \matr{S}_{J}\bar{\bv}
\end{equation}
for recurring $J$, where $\matr{S}_{J}\bar{\bv} = \matr{P}^{\perp}_{J}\matr{Q}\bar{\bv}$ because $\bar{\bv} \in \ker\matr{A}_{J}$ by Proposition~\ref{prop:limit_velocity}.

\begin{lem}[Recurring walls]
\label{lem:recurring_walls}
Every recurring blocking wall $i$ satisfies:
\begin{itemize}
\item[(i)] $\matr{A}_{i}\bar{\bx} = b_{i}$ and $\matr{A}_{i}\bar{\bv} = 0$;

\item[(ii)] $\matr{A}_{i}\bm{\zeta}^{(k)} = \matr{A}_{i}(\bx^{(k)}-\bar{\bx}) \leqslant 0$ for every $k$ on the tail, with strict inequality whenever $i$ is the blocking wall of step $k$, and $\matr{A}_{i}\bm{\zeta}^{(k+1)} = 0$ for every such $k$.
\end{itemize}
\end{lem}
\begin{proof}
By~\eqref{eq:tail_coords}, $\bm{\zeta}^{(k)} = -\sum_{l \geqslant k}\Delta^{(l)}(\bd^{(l)} - \bar{\bv})$, so $\normtwo{\bm{\zeta}^{(k)}} \leqslant \sigma_{k}\sup_{l \geqslant k}\normtwo{\bd^{(l)} - \bar{\bv}} = o(\sigma_{k})$ by Proposition~\ref{prop:limit_velocity}. At the steps $k$ blocked by $i$, $\matr{A}_{i}\bx^{(k+1)} = b_{i}$ by~\eqref{eq:blocking}; these iterates converge to $\bar{\bx}$, so $\matr{A}_{i}\bar{\bx} = b_{i}$, and then $0 = \matr{A}_{i}(\bx^{(k+1)} - \bar{\bx}) = -\sigma_{k+1}\matr{A}_{i}\bar{\bv} + \matr{A}_{i}\bm{\zeta}^{(k+1)}$. Dividing by $\sigma_{k+1}$ and letting $k \to \infty$ along the hits gives $\matr{A}_{i}\bar{\bv} = 0$. This proves (i).

Feasibility gives $\matr{A}_{i}(\bx^{(k)}-\bar{\bx}) = \matr{A}_{i}\bx^{(k)} - b_{i} \leqslant 0$, and (i) gives $\matr{A}_{i}\bm{\zeta}^{(k)} = \matr{A}_{i}(\bx^{(k)}-\bar{\bx})$. The inequality is strict at the blocking steps of $i$ by~\eqref{eq:blocking}, where $\matr{A}_{i}\bm{\zeta}^{(k+1)} = \matr{A}_{i}(\bx^{(k+1)}-\bar{\bx}) = 0$.
\end{proof}

\begin{lem}
\label{lem:deviation}
On the blocked tail, relabeled as above so that every $J^{(k)}$ is recurring, for every $k$:
\begin{itemize}
\item[(i)] $\matr{A}_{J^{(k)}}\bm{\zeta}^{(k)} = \matr{A}_{J^{(k)}}\bm{\zeta}^{(k+1)} = \boldsymbol{0}$;
\item[(ii)] $\bm{\zeta}^{(k+1)} = \bm{\zeta}^{(k)} + \Delta^{(k)}\sigma_{k}\bu_{J^{(k)}} - \Delta^{(k)}\matr{S}_{J^{(k)}}\bm{\zeta}^{(k)}$.
\end{itemize}
\end{lem}
\begin{proof}
The constraints of $J^{(k)}$ are active at $\bx^{(k)}$ and $\bx^{(k+1)}$, as shown after~\eqref{eq:blocking}. Since $J^{(k)}$ recurs, they are active at infinitely many iterates converging to $\bar{\bx}$, so $\matr{A}_{J^{(k)}}\bar{\bx} = \bb_{J^{(k)}}$ and $\matr{A}_{J^{(k)}}(\bx^{(k)}-\bar{\bx}) = \matr{A}_{J^{(k)}}(\bx^{(k+1)}-\bar{\bx}) = \boldsymbol{0}$. Proposition~\ref{prop:limit_velocity} gives $\matr{A}_{J^{(k)}}\bar{\bv} = \boldsymbol{0}$. Adding $\sigma_{k}$ and $\sigma_{k+1}$ gives (i).

By Proposition~\ref{prop:limit_velocity}, $-\matr{P}^{\perp}_{J^{(k)}}\nabla_{\bx}f(\bar{\bx}) = \bar{\bd}_{J^{(k)}} = \bar{\bv}$. Hence,
\[
\bd^{(k)} = -\matr{P}^{\perp}_{J^{(k)}}\nabla_{\bx}f(\bx^{(k)}) = \bar{\bv} - \matr{P}^{\perp}_{J^{(k)}}\matr{Q}(\bx^{(k)}-\bar{\bx}) = \bar{\bv} + \sigma_{k}\bu_{J^{(k)}} - \matr{S}_{J^{(k)}}\bm{\zeta}^{(k)},
\]
where the last step uses $\matr{P}^{\perp}_{J^{(k)}}\matr{Q}\bm{\zeta}^{(k)} = \matr{S}_{J^{(k)}}\bm{\zeta}^{(k)}$, valid because $\bm{\zeta}^{(k)} \in \ker\matr{A}_{J^{(k)}}$ by (i). Then
\[
\bm{\zeta}^{(k+1)} = \bx^{(k)} - \bar{\bx} + \Delta^{(k)}\bd^{(k)} + \sigma_{k+1}\bar{\bv} = \bm{\zeta}^{(k)} + (\sigma_{k+1} - \sigma_{k} + \Delta^{(k)})\bar{\bv} + \Delta^{(k)}\sigma_{k}\bu_{J^{(k)}} - \Delta^{(k)}\matr{S}_{J^{(k)}}\bm{\zeta}^{(k)},
\]
and $\sigma_{k+1} - \sigma_{k} + \Delta^{(k)} = 0$.
\end{proof}

\paragraph{Exclusion.} Summing the recursion of Lemma~\ref{lem:deviation}(ii) from $k$ to $k'-1$ and letting $k' \to \infty$ gives
\begin{equation}
\label{eq:driven_discrete}
\bm{\zeta}^{(k)} = -\sum_{l \geqslant k}\Delta^{(l)}\bigl(\sigma_{l}\bu_{J^{(l)}} - \matr{S}_{J^{(l)}}\bm{\zeta}^{(l)}\bigr),
\end{equation}
the series converging absolutely because $\normtwo{\bm{\zeta}^{(l)}}$ is bounded, $\normtwo{\matr{S}_{J}} \leqslant \lambda_{\max}$, and $\sigma_{k} < +\infty$. This identity is~\eqref{eq:driven_shared} with sums over steps in place of integrals, and its weights are the \emph{iterated tail sums} $\sigma^{[0]}_{k} \coloneqq 1$ and $\sigma^{[j+1]}_{k} \coloneqq \sum_{l \geqslant k}\Delta^{(l)}\sigma^{[j]}_{l}$, so that $\sigma^{[1]}_{k} = \sigma_{k}$. They are positive and nonincreasing in $k$, and they satisfy $\sigma^{[j+1]}_{k} \leqslant \sigma_{k}\sigma^{[j]}_{k}$, hence $\sigma^{[j]}_{k} \leqslant \sigma_{k}^{j}$.
\begin{lem}
\label{lem:discrete_tail}
On the blocked tail, relabeled so that every $J^{(k)}$ is recurring:
\begin{itemize}
\item[(i)] There exist vectors $\bu^{[0]} \coloneqq \bar{\bv}, \bu^{[1]}, \bu^{[2]}, \dots$ such that $\matr{S}_{J}\bu^{[j-1]} = \bu^{[j]} \in \ker\matr{A}_{J}$ for every recurring $J$ and every $j \geqslant 1$, and such that, for every $r \geqslant 0$, the sequence $\bm{\omega}_{r}^{(k)} \coloneqq \bm{\zeta}^{(k)} + \sum_{j=1}^{r}\sigma^{[j+1]}_{k}\bu^{[j]}$ satisfies $\normtwo{\bm{\omega}_{r}^{(k)}} \leqslant 2\normtwo{\bu^{[r+1]}}\sigma^{[r+2]}_{k}$ whenever $\lambda_{\max}\sigma_{k} \leqslant \tfrac12$;
\item[(ii)] Every $\ba \in \Rn$ has an index $k_{\ba}$ such that $\langle\ba,\bm{\zeta}^{(k)}\rangle$ is either all zero or all nonzero for $k \geqslant k_{\ba}$.
\end{itemize}
\end{lem}
\begin{proof}
The proofs of Lemmas~\ref{lem:driven_shared}, \ref{lem:alignment_shared}, and~\ref{lem:sign_shared} carry over to~\eqref{eq:driven_discrete}, with sums over $l \geqslant k$ in place of integrals over $[\tau,\bar{\tau})$, $\Delta^{(l)}$ in place of $\diff\rho$, $\sigma_{k}$ in place of $\sigma(\tau)$, and $\sigma^{[j]}_{k}$ in place of $\sigma^{j}/j!$. They use the weights only for the properties above and the recursion that defines them, and of the steps only $\Delta^{(l)} > 0$, in place of the positive length of the pieces. The hypotheses hold by Proposition~\ref{prop:limit_velocity}, \eqref{eq:tail_coords}, and Lemma~\ref{lem:deviation}(i).
\end{proof}

\begin{proof}[Proof of Theorem~\ref{thm:no_blocked_tail}]
Suppose a blocked tail exists and adopt the relabeling above. Fix a recurring blocking wall $i$ and put $\varphi(k) \coloneqq \matr{A}_{i}\bm{\zeta}^{(k)} = \matr{A}_{i}\bx^{(k)} - b_{i}$, by Lemma~\ref{lem:recurring_walls}(ii), so that $\varphi(k+1) = 0$ whenever $i$ blocks step $k$. By Lemma~\ref{lem:discrete_tail}(ii) with $\ba = \matr{A}_{i}^{\top}$, either $\varphi$ vanishes at every iterate beyond some index, or it vanishes at none of them. In the first case $i \in \eqset^{(k)}$ at every such iterate, whereas~\eqref{eq:blocking} requires $i \notin \eqset^{(k)}$ at every step that $i$ blocks. In the second case $\varphi(k+1) \neq 0$ at every such $k$, so no such step is blocked by $i$. Either way $i$ blocks finitely many steps. The recurring walls are finitely many, so from some step on no step is blocked by a recurring wall, although every late step is blocked by one. No blocked tail exists.
\end{proof}

\subsubsection{Proof of Proposition~\ref{prop:descent_dichotomy}}
\label{app:proof_descent_psd}
When problem~\eqref{eq:intro_main_problem} admits an optimal solution, we write $X^{s}$ for its optimal solution set and $V^{s}$ for its optimal value. The proofs of this appendix rest on one observation. At a feasible $\bx$ with active set $J = J(\bx)$ and multipliers $\bp_{J} = \bp_{J}(\bx) \geqslant \boldsymbol{0}$, the relations below hold, the last one by Moreau's decomposition~\eqref{eq:moreau_decomp}:
\begin{equation}\label{eq:perturbed_kkt}
\matr{A}\bx \leqslant \bb,
\qquad
\matr{A}_{J}\bx = \bb_{J},
\qquad
\bp_{J} \geqslant \boldsymbol{0},
\qquad
\matr{Q}\bx + \matr{A}_{J}^{\top}\bp_{J} = -\bc - \bd(\bx).
\end{equation}
These are the KKT conditions of problem~\eqref{eq:intro_main_problem} with one inexact equation, whose residual is exactly $-\bd(\bx)$. Accordingly, for an index set $J \subseteq [m]$, let $\Upsilon_{J}$ denote the solution set, possibly empty, of the exact system, in the variables $(\by,\bs) \in \Rn \times \R^{|J|}$,
\[
\Upsilon_{J} \coloneqq \bigl\{(\by,\bs) : \matr{A}\by \leqslant \bb,\ \matr{A}_{J}\by = \bb_{J},\ \bs \geqslant \boldsymbol{0},\ \matr{Q}\by + \matr{A}_{J}^{\top}\bs = -\bc\bigr\},
\]
and let $\mathcal{D}_{J}$ be the set of stationarity residuals compatible with the other three relations of $\Upsilon_{J}$,
\[
\mathcal{D}_{J} \coloneqq \bigl\{\matr{Q}\by + \matr{A}_{J}^{\top}\bs + \bc : \matr{A}\by \leqslant \bb,\ \matr{A}_{J}\by = \bb_{J},\ \bs \geqslant \boldsymbol{0}\bigr\} \subseteq \Rn.
\]

We use four facts. First, the set $\mathcal{D}_{J}$ is the image of a polyhedron under an affine map, hence a polyhedron, hence closed: a polyhedron is finitely generated and an affine image of a finitely generated set is finitely generated, so~\cite[Theorem~19.1]{rockafellar1997convex} applies twice. Second, the set $\Upsilon_{J}$ is nonempty if and only if $\boldsymbol{0} \in \mathcal{D}_{J}$. Third, every feasible $\bx$ satisfies $-\bd(\bx) \in \mathcal{D}_{J(\bx)}$, by~\eqref{eq:perturbed_kkt}. Finally, every $(\hat{\by},\hat{\bs}) \in \Upsilon_{J}$ makes $\hat{\by}$ an optimal solution of problem~\eqref{eq:intro_main_problem}, since $\hat{\by}$ is feasible and convexity of $f$ gives $f(\by) - f(\hat{\by}) \geqslant \langle\nabla_{\bx}f(\hat{\by}),\by - \hat{\by}\rangle = -\langle\hat{\bs},\matr{A}_{J}\by - \bb_{J}\rangle \geqslant 0$ at every feasible $\by$.

\begin{lem}
\label{lem:hoffman_error_bound}
Suppose problem~\eqref{eq:intro_main_problem} admits an optimal solution. There exist $\varepsilon_{0} > 0$ and $\kappa_{0} > 0$, depending only on $\matr{A}$, $\bb$, $\matr{Q}$, and $\bc$, such that every $\bx \in C$ with $\normtwo{\bd(\bx)} < \varepsilon_{0}$ satisfies
\begin{equation}\label{eq:hoffman_error_bound}
V(\bx) - V^{s} \;\leqslant\; \kappa_{0}\normsq{\bd(\bx)}.
\end{equation}
\end{lem}
\begin{proof}
Define $\varepsilon_{0} \coloneqq \min\bigl(\{1\} \cup \{\mathrm{dist}(\boldsymbol{0},\mathcal{D}_{J}) : J\ \text{with}\ \boldsymbol{0} \notin \mathcal{D}_{J}\}\bigr)$ with $\mathrm{dist}(\boldsymbol{0},\emptyset) \coloneqq +\infty$. The minimum runs over $1$ and finitely many distances to closed sets that avoid the origin, so $\varepsilon_{0} \in (0,1]$. For each $J$ with $\Upsilon_{J} \neq \emptyset$, let $h_{J}$ denote the constant of the error bound~\cite{hoffman1952} for the linear system defining $\Upsilon_{J}$, which depends only on the coefficient matrices, not on the right-hand sides, and set $\kappa_{0} \coloneqq \tfrac{\lambda_{\max}}{2}\max\{h_{J}^{2} : \Upsilon_{J} \neq \emptyset\}$, a maximum over a nonempty finite set since every optimal $\bx$ satisfies $(\bx,\bp_{J(\bx)}(\bx)) \in \Upsilon_{J(\bx)}$, by~\eqref{eq:perturbed_kkt} with $\bd(\bx) = \boldsymbol{0}$. If it vanishes, replace $\kappa_{0}$ by any positive number.

Now let $\bx$ be feasible with $\normtwo{\bd(\bx)} < \varepsilon_{0}$ and put $J \coloneqq J(\bx)$. Since $-\bd(\bx) \in \mathcal{D}_{J}$, the case $\boldsymbol{0} \notin \mathcal{D}_{J}$ would give $\normtwo{\bd(\bx)} \geqslant \mathrm{dist}(\boldsymbol{0},\mathcal{D}_{J}) \geqslant \varepsilon_{0}$, a contradiction. Hence $\Upsilon_{J} \neq \emptyset$. The point $(\bx,\bp_{J}(\bx))$ satisfies every relation defining $\Upsilon_{J}$ exactly except the stationarity equation, whose violation is $-\bd(\bx)$ by~\eqref{eq:perturbed_kkt}, so the error bound yields $(\hat{\by},\hat{\bs}) \in \Upsilon_{J}$ with $\normtwo{\bx - \hat{\by}} \leqslant h_{J}\normtwo{\bd(\bx)}$. Finally, since $V(\hat{\by}) = V^{s}$ by the fourth fact above, expanding $V$ at $\hat{\by}$ and using $\matr{Q}\hat{\by} + \bc = -\matr{A}_{J}^{\top}\hat{\bs}$ and $\matr{A}_{J}\bx = \matr{A}_{J}\hat{\by} = \bb_{J}$,
\[
V(\bx) - V^{s} = -\langle\hat{\bs},\matr{A}_{J}(\bx - \hat{\by})\rangle + \tfrac12\langle\bx - \hat{\by},\matr{Q}(\bx - \hat{\by})\rangle
\leqslant \tfrac{\lambda_{\max}}{2}h_{J}^{2}\normsq{\bd(\bx)} \leqslant \kappa_{0}\normsq{\bd(\bx)},
\]
where the first term vanishes.
\end{proof}

\begin{proof}[Proof of Proposition~\ref{prop:descent_dichotomy}]
Suppose first that some iterate satisfies $\bd^{(K)} \neq \boldsymbol{0}$ and $\Delta^{(K)} = +\infty$, and write $\bx = \bx^{(K)}$, $\bd = \bd^{(K)}$, $J = J(\bx)$. Then $\Delta_{*}^{(K)} = +\infty$ gives $\langle\bd,\matr{Q}\bd\rangle = 0$, hence $\bd \in \ker\matr{Q}$, since $\langle\bd,\matr{Q}\bd\rangle = \normsq{\matr{Q}^{1/2}\bd}$ for SPSD $\matr{Q}$. Moreover, $\matr{A}_{J}\bd = \boldsymbol{0}$ by the KKT condition~\eqref{eq:nnls_oc_2_5}, $[\matr{A}\bd]_{i} \leqslant 0$ for $i \in \eqset(\bx) \setminus J$ because $\bd \in T_{C}(\bx)$, and $[\matr{A}\bd]_{i} \leqslant 0$ for $i \notin \eqset(\bx)$ because $\Delta_{\mathrm{feas}}(\bx) = +\infty$ in~\eqref{eq:delta_feas}. Hence $\matr{A}\bd \leqslant \boldsymbol{0}$. Finally, the optimality condition~\eqref{eq:nnls_oc_3} gives $\langle\bc,\bd\rangle = \langle\matr{Q}\bx + \bc,\bd\rangle - \langle\bx,\matr{Q}\bd\rangle = -\normsq{\bd}$. The ray $\bx + \Delta\bd$ is therefore feasible, and $V(\bx + \Delta\bd) = V(\bx) - \Delta\normsq{\bd} \to -\infty$, which proves~\eqref{eq:recession_certificate} and unboundedness along the ray. In particular, no such iterate exists when problem~\eqref{eq:intro_main_problem} is bounded below, and $V(\bx^{(k)}) \geqslant V^{s}$ holds there because the iterates are feasible, which proves the final claim.

Now suppose the problem is unbounded below and no iterate has $\Delta^{(k)} = +\infty$. No optimal solution exists, so $\Upsilon_{J} = \emptyset$ for every $J$, every closed set $\mathcal{D}_{J}$ avoids the origin, and $\theta_{0} \coloneqq \min_{J}\mathrm{dist}(\boldsymbol{0},\mathcal{D}_{J}) > 0$ (with $\mathrm{dist}(\boldsymbol{0},\emptyset) = +\infty$). Since $-\bd^{(k)} \in \mathcal{D}_{J^{(k)}}$, every iterate satisfies $\normtwo{\bd^{(k)}} \geqslant \theta_{0}$, so the iteration is infinite. If $V(\bx^{(k)})$ decreased to some $V_{\infty} > -\infty$, Lemma~\ref{lem:tail_extraction} would produce a blocked tail, which Theorem~\ref{thm:no_blocked_tail} excludes. Hence $V(\bx^{(k)}) \to -\infty$.
\end{proof}

\subsubsection{Proof of Theorem~\ref{thm:descent_convergence}}
\label{app:proof_descent_convergence}
Assume throughout that problem~\eqref{eq:intro_main_problem} has an optimal solution. By the first part of the proof of Proposition~\ref{prop:descent_dichotomy}, no iterate has $\Delta^{(k)} = +\infty$, and every step with $\bd^{(k)} \in \ker\matr{Q}$ is blocked. If $V(\bx^{(k)})$ decreased to some $V_{\infty} > V^{s}$, then every $k$ would have $\normtwo{\bd^{(k)}} \geqslant \theta \coloneqq \min\bigl(\varepsilon_{0},\sqrt{(V_{\infty} - V^{s})/\kappa_{0}}\bigr) > 0$: either $\normtwo{\bd^{(k)}} \geqslant \varepsilon_{0}$, or~\eqref{eq:hoffman_error_bound} gives $\kappa_{0}\normsq{\bd^{(k)}} \geqslant V(\bx^{(k)}) - V^{s} \geqslant V_{\infty} - V^{s}$. Lemma~\ref{lem:tail_extraction} would then yield a blocked tail, which Theorem~\ref{thm:no_blocked_tail} excludes. Hence $V(\bx^{(k)}) \downarrow V^{s}$.

\begin{lem}
\label{lem:finite_length}
Every step satisfies $\normtwo{\bx^{(k+1)} - \bx^{(k)}} \leqslant 2\bigl(V(\bx^{(k)}) - V(\bx^{(k+1)})\bigr)/\normtwo{\bd^{(k)}}$. Moreover, if the values $V(\bx^{(k)})$ decrease to $V^{s}$, then every $k$ with $V(\bx^{(k)}) - V^{s} \leqslant \kappa_{0}\varepsilon_{0}^{2}$ satisfies
\begin{equation}\label{eq:finite_length}
\sum_{l \geqslant k}\normtwo{\bx^{(l+1)} - \bx^{(l)}} \;\leqslant\; 4\sqrt{\kappa_{0}}\sqrt{V(\bx^{(k)}) - V^{s}},
\end{equation}
and the iterates converge to a point $\bxsol \in X^{s}$.
\end{lem}
\begin{proof}
On a full step, evaluating~\eqref{eq:obj_fun_after_descent} at $\Delta_{*}^{(k)}$ gives $V(\bx^{(k)}) - V(\bx^{(k+1)}) = \tfrac12\Delta_{*}^{(k)}\normsq{\bd^{(k)}}$. On a blocked step, the proof of Lemma~\ref{lem:sufficient_decrease} gives $V(\bx^{(k)}) - V(\bx^{(k+1)}) \geqslant \tfrac12\Delta^{(k)}\normsq{\bd^{(k)}}$. Thus, since $\normtwo{\bx^{(k+1)} - \bx^{(k)}} = \Delta^{(k)}\normtwo{\bd^{(k)}}$, dividing both by $\normtwo{\bd^{(k)}}/2$ gives the first claim. The numbers $a_{k} \coloneqq V(\bx^{(k)}) - V^{s}$ are nonnegative and nonincreasing, and $a_{k} > 0$ whenever step $k$ is taken, since $\bx^{(k)}$ is then not optimal. Let $a_{k} \leqslant \kappa_{0}\varepsilon_{0}^{2}$, which holds for some $k$ since $a_{k} \to 0$, and let $l \geqslant k$, so $a_{l} \leqslant \kappa_{0}\varepsilon_{0}^{2}$. Then $\normtwo{\bd^{(l)}} \geqslant \sqrt{a_{l}/\kappa_{0}}$, by~\eqref{eq:hoffman_error_bound} when $\normtwo{\bd^{(l)}} < \varepsilon_{0}$, and because $\normtwo{\bd^{(l)}} \geqslant \varepsilon_{0} \geqslant \sqrt{a_{l}/\kappa_{0}}$ otherwise. Hence
\[
\normtwo{\bx^{(l+1)} - \bx^{(l)}} \;\leqslant\; 2\sqrt{\kappa_{0}}\,\frac{a_{l} - a_{l+1}}{\sqrt{a_{l}}} \;\leqslant\; 4\sqrt{\kappa_{0}}\bigl(\sqrt{a_{l}} - \sqrt{a_{l+1}}\bigr),
\]
using $(a_{l} - a_{l+1})/\sqrt{a_{l}} = (\sqrt{a_{l}} - \sqrt{a_{l+1}})(\sqrt{a_{l}} + \sqrt{a_{l+1}})/\sqrt{a_{l}} \leqslant 2(\sqrt{a_{l}} - \sqrt{a_{l+1}})$, since $0 \leqslant a_{l+1} \leqslant a_{l}$, and the sum over $l \geqslant k$ telescopes to~\eqref{eq:finite_length}. The iterates are then Cauchy. Their limit $\bxsol$ lies in the closed set $C$, and $V(\bxsol) = \lim_{k}V(\bx^{(k)}) = V^{s}$ by continuity, so $\bxsol \in X^{s}$.
\end{proof}

Part~(i) of Theorem~\ref{thm:descent_convergence} follows. If $\bd^{(K)} = \boldsymbol{0}$ for some $K$, then $\bx^{(K)} \in X^{s}$ and $\bxsol = \bx^{(K)}$. Otherwise $V(\bx^{(k)}) \downarrow V^{s}$ as shown above, and Lemma~\ref{lem:finite_length} gives $\bx^{(k)} \to \bxsol \in X^{s}$.

For part~(ii), set $J^{s} \coloneqq J(\bxsol) = \eqset(\bxsol)$. By~\eqref{eq:LI}, the vectors $\{\matr{A}_{i}^{\top} : i \in J^{s}\}$ are linearly independent. Hence the dual solution $\bpsol$ is unique, and $[\bpsol]_{i} > 0$ for every $i \in J^{s}$ because $J^{s}$ is the support of $\bpsol$. Write $s_{0} \coloneqq \min_{i \notin J^{s}}(b_{i} - [\matr{A}\bxsol]_{i}) > 0$ for the least inactive slack at $\bxsol$. Let $\bar{\matr{S}} \coloneqq \matr{S}_{J^{s}}$ and $\bar{R} \coloneqq \image{\bar{\matr{S}}}$. When $\bar{R} \neq \{\boldsymbol{0}\}$, let $\bar{\lambda}_{\min} \leqslant \bar{\lambda}_{\max}$ denote the smallest and largest positive eigenvalues of $\bar{\matr{S}}$, so that $\bar{\kappa} = \bar{\lambda}_{\max}/\bar{\lambda}_{\min}$. The next lemma shows that, near $\bxsol$, the direction is linear on the face of $\bxsol$ and bounded away from zero elsewhere.

\begin{lem}
\label{lem:localization}
There exist $\rho_{0}, \delta_{0} > 0$ such that every feasible $\bx$ with $\normtwo{\bx - \bxsol} \leqslant \rho_{0}$ satisfies:
\begin{itemize}
\item[(i)] $\eqset(\bx) \subseteq J^{s}$;

\item[(ii)] if $\eqset(\bx) \neq J^{s}$, then $\normtwo{\bd(\bx)} \geqslant \delta_{0}$;

\item[(iii)] if $\eqset(\bx) = J^{s}$, then $J(\bx) = J^{s}$ and $\bd(\bx) = \bar{\matr{S}}(\bxsol - \bx) \in \bar{R}$, so that $\normtwo{\bd(\bx)} \leqslant \lambda_{\max}\normtwo{\bx - \bxsol}$.
\end{itemize}
\end{lem}
\begin{proof}
For $i \notin J^{s}$ the slack satisfies $b_{i} - [\matr{A}\bx]_{i} \geqslant s_{0} - \normtwo{\matr{A}}\normtwo{\bx - \bxsol} \geqslant s_{0}/2$ once $\rho_{0} \leqslant s_{0}/(2\normtwo{\matr{A}})$. This proves (i). For (ii), by~\eqref{eq:moreau_decomp}, $\normtwo{\bd(\bx)} = \mathrm{dist}(-\nabla_{\bx}f(\bx),N_{C}(\bx))$. By (i), $\eqset(\bx)$ is a proper subset of $J^{s}$. We claim that $\mathrm{dist}(-\nabla_{\bx}f(\bxsol),N_{C}(\bx)) > 0$. Otherwise $-\nabla_{\bx}f(\bxsol) \in N_{C}(\bx)$ has a representation supported on $\eqset(\bx)$, which, compared with $\bpsol$ over the linearly independent $\{\matr{A}_{i}^{\top} : i \in J^{s}\}$, forces $[\bpsol]_{i} = 0$ for $i \in J^{s} \setminus \eqset(\bx) \neq \emptyset$ and contradicts the hypothesis. The cone $N_{C}(\bx)$ depends only on $\eqset(\bx)$, which takes finitely many values, so $\delta_{0} \coloneqq \tfrac12\min\{\mathrm{dist}(-\nabla_{\bx}f(\bxsol),N_{C}(\by)) : \by \in C,\ \eqset(\by) \subsetneq J^{s}\}$ is positive. Hence
\[
\normtwo{\bd(\bx)} \geqslant \mathrm{dist}\bigl(-\nabla_{\bx}f(\bxsol),N_{C}(\bx)\bigr) - \normtwo{\nabla_{\bx}f(\bx) - \nabla_{\bx}f(\bxsol)} \geqslant 2\delta_{0} - \lambda_{\max}\rho_{0},
\]
and $\rho_{0} \leqslant \delta_{0}/\lambda_{\max}$ gives (ii).

The least-squares multiplier $\bp_{J^{s}}(\bx)$ of~\eqref{eq:pJ_flow} is affine in $\bx$ with $\bp_{J^{s}}(\bxsol) = \bpsol > \boldsymbol{0}$, so shrinking $\rho_{0}$ keeps $\bp_{J^{s}}(\bx) > \boldsymbol{0}$ on the ball. Since $\eqset(\bx) = J^{s}$, Lemma~\ref{lem:velocity_identity} with $\mathcal{K} = T_{C}(\bx)$ gives $\bd(\bx) = -\matr{P}^{\perp}_{J^{s}}\nabla_{\bx}f(\bx)$. Every NNLS solution at $\bx$ is supported in $\eqset(\bx) = J^{s}$, whose rows are linearly independent, so $\bp(\bx) = \bp_{J^{s}}(\bx)$ and $J(\bx) = J^{s}$. Finally $\matr{P}^{\perp}_{J^{s}}\nabla_{\bx}f(\bxsol) = -\bd(\bxsol) = \boldsymbol{0}$, since $\bxsol$ is optimal, so $\bd(\bx) = -\matr{P}^{\perp}_{J^{s}}\matr{Q}(\bx - \bxsol) = \bar{\matr{S}}(\bxsol - \bx)$, since $\bx - \bxsol \in \ker\matr{A}_{J^{s}}$. This proves (iii).
\end{proof}

On the affine set $\mathcal{F} \coloneqq \{\bx : \matr{A}_{J^{s}}\bx = \bb_{J^{s}}\}$, write $\be \coloneqq \bx - \bxsol \in \ker\matr{A}_{J^{s}}$ and $\be_{\bar{R}}$ for its component in $\bar{R}$. Since $\nabla_{\bx}f(\bxsol) = -\matr{A}_{J^{s}}^{\top}\bpsol$ is orthogonal to $\ker\matr{A}_{J^{s}}$, every feasible $\bx \in \mathcal{F}$ satisfies $V(\bx) - V^{s} = \tfrac12\langle\be,\matr{Q}\be\rangle = \tfrac12\langle\be_{\bar{R}},\bar{\matr{S}}\be_{\bar{R}}\rangle$, hence
\begin{equation}\label{eq:facegrowth}
\tfrac{\bar{\lambda}_{\min}}{2}\normsq{\be_{\bar{R}}} \;\leqslant\; V(\bx) - V^{s} \;\leqslant\; \tfrac{\bar{\lambda}_{\max}}{2}\normsq{\be_{\bar{R}}}
\qquad \text{when $\bar{R} \neq \{\boldsymbol{0}\}$.}
\end{equation}

\begin{lem}
\label{lem:absorbing}
There exists $\rho_{1} \in (0,\rho_{0}]$ such that if some iterate satisfies $\normtwo{\bx^{(k)} - \bxsol} \leqslant \rho_{1}$ and $\eqset^{(k)} = J^{s}$, then every $k' \geqslant k$ satisfies $\eqset^{(k')} = J^{(k')} = J^{s}$ and $\normtwo{\bx^{(k')} - \bxsol} \leqslant \rho_{0}$, every step $k' \geqslant k$ is full, and~\eqref{eq:descent_rate} holds for every $k' \geqslant k$ at which the iteration has not stopped.
\end{lem}
\begin{proof}
If $\bar{R} = \{\boldsymbol{0}\}$, Lemma~\ref{lem:localization}(iii) gives $\bd^{(k)} = \boldsymbol{0}$, so the iteration stops at $\bx^{(k)}$ and the claims hold with $\rho_{1} \coloneqq \rho_{0}$. Otherwise, choose $\rho_{1} \coloneqq \min\bigl(\rho_{0}/\sqrt{\bar{\kappa}},\ s_{0}\bar{\lambda}_{\min}/(4\normtwo{\matr{A}}\lambda_{\max}\sqrt{\bar{\kappa}})\bigr)$ and write $\be^{(k')} \coloneqq \bx^{(k')} - \bxsol$. We show by induction that every iterate from $\bx^{(k)}$ onward lies in $\mathcal{F} \cap \{\bx : \normtwo{\bx - \bxsol} \leqslant \sqrt{\bar{\kappa}}\rho_{1}\}$ with $\eqset = J = J^{s}$, that its component orthogonal to $\bar{R}$ equals that of $\be^{(k)}$, and that it takes a full step. Given the hypothesis at $k'$, Lemma~\ref{lem:localization}(iii) applies at $\bx^{(k')}$, since $\sqrt{\bar{\kappa}}\rho_{1} \leqslant \rho_{0}$, and gives $J^{(k')} = J^{s}$, $\bd^{(k')} = -\bar{\matr{S}}\be^{(k')} \in \bar{R}$, and
\[
\normtwo{\bd^{(k')}} \leqslant \lambda_{\max}\normtwo{\bx^{(k')} - \bxsol} \leqslant \lambda_{\max}\sqrt{\bar{\kappa}}\rho_{1} \leqslant \frac{s_{0}\bar{\lambda}_{\min}}{4\normtwo{\matr{A}}}.
\]
Since $\bd^{(k')} \in \bar{R}$, we have $\langle\bd^{(k')},\matr{Q}\bd^{(k')}\rangle = \langle\bd^{(k')},\bar{\matr{S}}\bd^{(k')}\rangle \geqslant \bar{\lambda}_{\min}\normsq{\bd^{(k')}}$, so $\Delta_{*}^{(k')} \leqslant 1/\bar{\lambda}_{\min}$. Every constraint outside $J^{s}$ has slack at least $s_{0}/2$ at $\bx^{(k')}$ by the proof of Lemma~\ref{lem:localization}(i), so
\[
\Delta_{\mathrm{feas}}(\bx^{(k')}) \geqslant \frac{s_{0}/2}{\normtwo{\matr{A}}\normtwo{\bd^{(k')}}} \geqslant \frac{2}{\bar{\lambda}_{\min}} > \Delta_{*}^{(k')},
\]
and the step is full. It moves along $\bar{R} \subseteq \ker\matr{A}_{J^{s}}$, so $\bx^{(k'+1)} \in \mathcal{F}$, where the constraints of $J^{s}$ are active and every other is strictly inactive because $\Delta^{(k')} < \Delta_{\mathrm{feas}}(\bx^{(k')})$, giving $\eqset^{(k'+1)} = J^{s}$; and the component of $\be^{(k'+1)}$ orthogonal to $\bar{R}$ equals that of $\be^{(k')}$, hence that of $\be^{(k)}$. Since $V$ is nonincreasing, \eqref{eq:facegrowth} gives $\normsq{\be^{(k'+1)}_{\bar{R}}} \leqslant \bar{\kappa}\normsq{\be^{(k)}_{\bar{R}}}$, so $\normsq{\be^{(k'+1)}} \leqslant \bar{\kappa}\normsq{\be^{(k)}_{\bar{R}}} + \normsq{\be^{(k)} - \be^{(k)}_{\bar{R}}} \leqslant \bar{\kappa}\normsq{\be^{(k)}} \leqslant \bar{\kappa}\rho_{1}^{2}$, closing the induction. On this tail, $\be^{(k'+1)}_{\bar{R}} = \be^{(k')}_{\bar{R}} - \Delta_{*}^{(k')}\bar{\matr{S}}\be^{(k')}_{\bar{R}}$ with $\Delta_{*}^{(k')} = \normsq{\bar{\matr{S}}\be^{(k')}_{\bar{R}}}/\langle\bar{\matr{S}}\be^{(k')}_{\bar{R}},\bar{\matr{S}}^{2}\be^{(k')}_{\bar{R}}\rangle$, which is steepest descent with exact line search for the quadratic $\be \mapsto \tfrac12\langle\be,\bar{\matr{S}}\be\rangle$ on $\bar{R}$, where $\bar{\matr{S}}$ is positive definite with spectrum in $[\bar{\lambda}_{\min},\bar{\lambda}_{\max}]$. The Kantorovich bound~\cite[Theorem~3.3]{nocedal2006numerical}, stated there for the norm $\be \mapsto \langle\be,\bar{\matr{S}}\be\rangle^{1/2}$, gives~\eqref{eq:descent_rate}, because $V(\bx) - V^{s} = \tfrac12\langle\be_{\bar{R}},\bar{\matr{S}}\be_{\bar{R}}\rangle$ at every feasible $\bx \in \mathcal{F}$.
\end{proof}

\begin{prop}
\label{prop:identification}
Under the hypotheses of Theorem~\ref{thm:descent_convergence}(ii), some iterate satisfies $\eqset^{(k)} = J^{s}$ and $\normtwo{\bx^{(k)} - \bxsol} \leqslant \rho_{1}$.
\end{prop}
\begin{proof}
If $\bd^{(K)} = \boldsymbol{0}$ for some $K$, then $\bx^{(K)} = \bxsol$ and the claim holds at $k = K$, so assume the iteration is infinite. By part~(i), all late iterates have $\normtwo{\bx^{(k)} - \bxsol} \leqslant \rho_{1}$. If none of them had $\eqset^{(k)} = J^{s}$, then $\eqset^{(k)} \subsetneq J^{s}$ by Lemma~\ref{lem:localization}(i), and Lemma~\ref{lem:localization}(ii) would give $\normtwo{\bd^{(k)}} \geqslant \delta_{0}$ for every late $k$. Since $V \geqslant V^{s}$, Lemma~\ref{lem:tail_extraction} applied to the late run would then yield a blocked tail, contradicting Theorem~\ref{thm:no_blocked_tail}.
\end{proof}

Let $K'$ be the index of Proposition~\ref{prop:identification}. Lemma~\ref{lem:absorbing} gives $\eqset^{(k)} = J^{(k)} = J^{s}$, a full step, and~\eqref{eq:descent_rate} at every $k \geqslant K'$ at which the iteration has not stopped, which is~\eqref{eq:descent_identification} together with the rate. If $\bar{\matr{S}} = \matr{0}$, then Lemma~\ref{lem:localization}(iii) gives $\bd^{(K')} = \boldsymbol{0}$.
\hfill$\square$

\bibliography{proj_quad_prog}

@article{aps2019mosek,
  title={{MOSEK} optimization toolbox for {MATLAB}},
  author={{MOSEK ApS}},
  journal={User's Guide and Reference Manual, Version},
  volume={4},
  number={1},
  pages={116},
  year={2019}
}

@book{attouch2014variational,
  title={Variational analysis in {S}obolev and {BV} spaces: applications to {PDEs} and optimization},
  author={Attouch, H. and Buttazzo, G. and Michaille, G.},
  year={2014},
  publisher={SIAM}
}

@book{aubin2012differential,
  title={Differential inclusions: {S}et-valued {M}aps and {V}iability {T}heory},
  author={Aubin, J.-P. and Cellina, A.},
  volume={264},
  year={2012},
  publisher={Springer Science \& Business Media}
}

@book{brezis1973ope,
  title={Op\'erateurs maximaux monotones et semi-groupes de contractions dans les espaces de {H}ilbert},
  author={Brezis, H.},
  year={1973},
  publisher={North-Holland Publishing Company},
  address={Amsterdam},
  series={North-Holland Mathematics Studies, No. 5}
}

@article{brogliato2006equivalence,
  title={On the equivalence between complementarity systems, projected systems and differential inclusions},
  author={Brogliato, B. and Daniilidis, A. and Lemar{\'e}chal, C. and Acary, V.},
  journal={Systems \& Control Letters},
  volume={55},
  number={1},
  pages={45--51},
  year={2006},
  publisher={Elsevier}
}

@incollection{brown1951computational,
  title={Computational suggestions for maximizing a linear function subject to linear inequalities},
  author={Brown, G. W. and Koopmans, T. C.},
  booktitle={Activity Analysis of Production and Allocation},
  editor={Koopmans, T. C.},
  pages={377--380},
  year={1951},
  publisher={John Wiley \& Sons},
  address={New York}
}

@article{chambolle2016introduction,
  title={An introduction to continuous optimization for imaging},
  author={Chambolle, A. and Pock, T.},
  journal={Acta Numer.},
  volume={25},
  pages={161--319},
  year={2016},
  publisher={Cambridge University Press}
}

@article{chang1989steepest,
  title={The steepest descent gravitational method for linear programming},
  author={Chang, S. Y. and Murty, K. G.},
  journal={Discrete Appl. Math.},
  volume={25},
  number={3},
  pages={211--239},
  year={1989},
  publisher={Elsevier}
}

@article{charnes1957management,
  title={Management models and industrial applications of linear programming},
  author={Charnes, A. and Cooper, W. W.},
  journal={Manage. Sci.},
  volume={4},
  number={1},
  pages={38--91},
  year={1957},
  publisher={INFORMS}
}

@article{chen2001atomic,
  title={Atomic decomposition by basis pursuit},
  author={Chen, SS and Donoho, DL and Saunders, MA},
  journal={SIAM Rev.},
  volume={43},
  number={1},
  pages={129--159},
  year={2001},
  publisher={SIAM}
}

@incollection{cottle2010harry,
  title={Harry {M}arkowitz and the early history of quadratic programming},
  author={Cottle, R. W. and Infanger, G.},
  booktitle={Handbook of Portfolio Construction},
  pages={179--211},
  year={2010},
  publisher={Springer}
}

@incollection{dantzig1951maximization,
  title={Maximization of a linear function of variables subject to linear inequalities},
  author={Dantzig, G. B.},
  booktitle={Activity Analysis of Production and Allocation},
  editor={Koopmans, T. C.},
  pages={339--347},
  year={1951},
  publisher={John Wiley \& Sons},
  address={New York}
}

@techreport{dimiccoli2016fundamentals,
  title={Fundamentals of cone regression},
  author={Dimiccoli, M.},
  institution={Institute of Robotics and Industrial Informatics (CSIC-UPC)},
  year={2016}
}

@article{dinh2025real,
  title={Real-time {QP} solvers: a concise review and practical guide towards legged robots},
  author={Dinh, V. N.},
  journal={arXiv preprint arXiv:2510.21773},
  year={2025}
}

@article{dontchev1992difference,
  title={Difference methods for differential inclusions: A survey},
  author={Dontchev, Asen and Lempio, Frank},
  journal={SIAM review},
  volume={34},
  number={2},
  pages={263--294},
  year={1992},
  publisher={SIAM}
}

@article{dupuis1991lipschitz,
  title={On {L}ipschitz continuity of the solution mapping to the {S}korokhod problem, with applications},
  author={Dupuis, P. and Ishii, H.},
  journal={Stochastics},
  volume={35},
  number={1},
  pages={31--62},
  year={1991},
  publisher={Taylor \& Francis}
}

@article{dupuis1993dynamical,
  title={Dynamical systems and variational inequalities},
  author={Dupuis, Paul and Nagurney, Anna},
  journal={Annals of Operations Research},
  volume={44},
  number={1},
  pages={7--42},
  year={1993},
  publisher={Springer}
}

@article{ers14,
author    = "Howard Elman and Alison Ramage and David Silvester",
title     = "{IFISS}: A computational laboratory for investigating incompressible flow problems",
journal   = "SIAM Review",
volume    = "56",
pages     = "261--273",
year      = "2014"
}

@article{frank1956,
  author  = {Frank, Marguerite and Wolfe, Philip},
  title   = {An algorithm for quadratic programming},
  journal = {Naval Research Logistics Quarterly},
  volume  = {3},
  number  = {1--2},
  pages   = {95--110},
  year    = {1956}
}

@book{goebel2024set,
	title={Set-valued, convex, and nonsmooth analysis in dynamics and control: An introduction},
	author={Goebel, Rafal K},
	year={2024},
	publisher={SIAM}
}

@misc{gurobi2026,
  author       = {{Gurobi Optimization, LLC}},
  title        = {{G}urobi {O}ptimizer Reference Manual},
  year         = {2026},
  howpublished = {\url{https://www.gurobi.com}}
}

@article{heemels2000linear,
  title={Linear complementarity systems},
  author={Heemels, W. P. M. H. and Schumacher, J. M. and Weiland, S.},
  journal={SIAM J. Appl. Math.},
  volume={60},
  number={4},
  pages={1234--1269},
  year={2000},
  publisher={SIAM}
}

@book{hiriart2013convexI,
  title={Convex analysis and minimization algorithms {I}: Fundamentals},
  author={Hiriart-Urruty, J.-B. and Lemar{\'e}chal, C.},
  volume={305},
  series={Grundlehren Text Editions},
  year={1996},
  publisher={Springer Science \& Business Media}
}

@book{hiriart2013convexII,
  title={Convex analysis and minimization algorithms {II}: Advanced Theory and Bundle Methods},
  author={Hiriart-Urruty, J.-B. and Lemar{\'e}chal, C.},
  volume={306},
  series={Grundlehren Text Editions},
  year={1996},
  publisher={Springer Science \& Business Media}
}

@article{hoffman1952,
  author  = {Hoffman, Alan J.},
  title   = {On approximate solutions of systems of linear inequalities},
  journal = {Journal of Research of the National Bureau of Standards},
  volume  = {49},
  number  = {4},
  pages   = {263--265},
  year    = {1952}
}

@article{huangfu2018parallelizing,
  title={Parallelizing the dual revised simplex method},
  author={Huangfu, Q. and Hall, J. A. J.},
  journal={Math. Program. Comput.},
  volume={10},
  number={1},
  pages={119--142},
  year={2018},
  publisher={Springer}
}

@article{langlois2025fast,
  title={A fast algorithm for solving the lasso problem exactly without homotopy using differential inclusions},
  author={Langlois, G. P. and Darbon, J.},
  journal={arXiv preprint arXiv:2507.05562},
  year={2025}
}

@book{lawson1995solving,
  author={Lawson, C. L. and Hanson, R. J.},
  pages={xii+337},
  publisher={SIAM},
  series={Classics in Applied Mathematics},
  title={Solving least squares problems},
  volume={15},
  year={1995}
}

@article{lemke1961constrained,
  title={The constrained gradient method of linear programming},
  author={Lemke, C. E.},
  journal={J. Soc. Indust. Appl. Math.},
  volume={9},
  number={1},
  pages={1--17},
  year={1961},
  publisher={SIAM}
}

@article{liu2024optimization,
  author  = {Liu, Chen and Buzzard, Gregery T. and Zhang, Xiangxiong},
  title   = {An optimization based limiter for enforcing positivity in a semi-implicit discontinuous {G}alerkin scheme for compressible {N}avier--{S}tokes equations},
  journal = {J. Comput. Phys.},
  volume  = {519},
  pages   = {113440},
  year    = {2024},
  doi     = {10.1016/j.jcp.2024.113440}
}

@article{liu2024simple,
  author  = {Liu, Chen and Riviere, Beatrice and Shen, Jie and Zhang, Xiangxiong},
  title   = {A simple and efficient convex optimization based bound-preserving high order accurate limiter for {C}ahn--{H}illiard--{N}avier--{S}tokes system},
  journal = {SIAM J. Sci. Comput.},
  volume  = {46},
  number  = {3},
  pages   = {A1923--A1948},
  year    = {2024},
  doi     = {10.1137/23M1587853}
}

@article{liu2025optimization,
  author  = {Liu, Chen and Hu, Jingwei and Taitano, William T. and Zhang, Xiangxiong},
  title   = {An optimization-based positivity-preserving limiter in semi-implicit discontinuous {G}alerkin schemes solving {F}okker--{P}lanck equations},
  journal = {Comput. Math. Appl.},
  volume  = {192},
  pages   = {54--71},
  year    = {2025},
  doi     = {10.1016/j.camwa.2025.05.008}
}

@article{markowitz1956optimization,
  title={The optimization of a quadratic function subject to linear constraints},
  author={Markowitz, H.},
  journal={Nav. Res. Logist. Q.},
  volume={3},
  number={1-2},
  pages={111--133},
  year={1956},
  publisher={John Wiley \& Sons}
}

@article{meyer2013simple,
  title={A simple new algorithm for quadratic programming with applications in statistics},
  author={Meyer, M. C.},
  journal={Comm. Statist. Simulation Comput.},
  volume={42},
  number={5},
  pages={1126--1139},
  year={2013},
  publisher={Taylor \& Francis}
}

@article{moreau1962decomposition,
  title={D{\'e}composition orthogonale d'un espace hilbertien selon deux c{\^o}nes mutuellement polaires},
  author={Moreau, J. J.},
  journal={C. R. Acad. Sci. Paris},
  volume={255},
  pages={238--240},
  year={1962}
}

@book{nocedal2006numerical,
  title={Numerical optimization},
  author={Nocedal, J. and Wright, S. J.},
  year={2006},
  publisher={Springer}
}

@article{pearson2012regularization,
  title={Regularization-robust preconditioners for time-dependent {PDE}-constrained optimization problems},
  author={Pearson, J. W. and Stoll, M. and Wathen, A. J.},
  journal={SIAM J. Matrix Anal. Appl.},
  volume={33},
  number={4},
  pages={1126--1152},
  year={2012},
  doi={10.1137/110847949}
}

@article{pearson2017fast,
  title={Fast interior point solution of quadratic programming problems arising from {PDE}-constrained optimization},
  author={Pearson, J. W. and Gondzio, J.},
  journal={Numer. Math.},
  volume={137},
  number={4},
  pages={959--999},
  year={2017},
  publisher={Springer},
  doi={10.1007/s00211-017-0892-8}
}

@misc{qpbenchmark,
    title = {{qpbenchmark: Benchmark for quadratic programming solvers available
  in Python}},
    author = {Caron, St{\'e}phane and Zaki, Akram and Otta, Pavel and
  Arnstr{\"o}m, Daniel and Carpentier, Justin and Yang, Fengyu and Leziart,
  Pierre-Alexandre and Sousa Pinto, Jo{\~a}o},
    url = {https://github.com/qpsolvers/qpbenchmark},
    note = {Version 2.7.1},
    year = {2026}
  }

@book{rockafellar1997convex,
  title={Convex analysis},
  author={Rockafellar, R. T.},
  volume={11},
  year={1997},
  publisher={Princeton University Press}
}

@book{rockafellar2009variational,
  title={Variational analysis},
  author={Rockafellar, R. T. and Wets, R. J.-B.},
  volume={317},
  year={2009},
  publisher={Springer Science \& Business Media}
}

@article{rosen1960gradient,
  title={The gradient projection method for nonlinear programming. {P}art {I}. {L}inear constraints},
  author={Rosen, J. B.},
  journal={J. Soc. Indust. Appl. Math.},
  volume={8},
  number={1},
  pages={181--217},
  year={1960},
  publisher={SIAM}
}

@techreport{stoll2010all,
  title={All-at-once solution of time-dependent {PDE}-constrained optimization problems},
  author={Stoll, M. and Wathen, A.},
  institution={Mathematical Institute, University of Oxford},
  type={Technical Report},
  number={NA-10/13},
  year={2010},
  month={nov},
  url={https://ora.ox.ac.uk/objects/uuid:c94b530e-5a3d-4fb1-9c11-8d50302b9cb3}
}

@article{su2016differential,
  title={A differential equation for modeling Nesterov's accelerated gradient method: Theory and insights},
  author={Su, Weijie and Boyd, Stephen and Candes, Emmanuel J},
  journal={Journal of Machine Learning Research},
  volume={17},
  number={153},
  pages={1--43},
  year={2016}
}

@inproceedings{tassa2014control,
  title={Control-limited differential dynamic programming},
  author={Tassa, Y. and Mansard, N. and Todorov, E.},
  booktitle={2014 {IEEE} International Conference on Robotics and Automation ({ICRA})},
  pages={1168--1175},
  year={2014},
  organization={IEEE}
}

@article{tendero2021algorithm,
  title={An algorithm solving compressive sensing problem based on maximal monotone operators},
  author={Tendero, Y. and Ciril, I. and Darbon, J. and Serna, S.},
  journal={SIAM J. Sci. Comput.},
  volume={43},
  number={6},
  pages={A4067--A4094},
  year={2021},
  publisher={SIAM}
}

@article{tibshirani1996regression,
  title={Regression shrinkage and selection via the lasso},
  author={Tibshirani, R.},
  journal={J. R. Stat. Soc. Ser. B. Stat. Methodol.},
  volume={58},
  number={1},
  pages={267--288},
  year={1996},
  publisher={Oxford University Press}
}

@article{todd2002many,
  title={The many facets of linear programming},
  author={Todd, M. J.},
  journal={Math. Program.},
  volume={91},
  number={3},
  pages={417--436},
  year={2002}
}

@article{wolfe1959simplex,
  title={The simplex method for quadratic programming},
  author={Wolfe, P.},
  journal={Econometrica},
  volume={27},
  number={3},
  pages={382--398},
  year={1959},
  publisher={JSTOR}
}

@article{wright2025optimization,
  title={Optimization in theory and practice},
  author={Wright, S. J.},
  journal={arXiv preprint arXiv:2510.15734},
  year={2025}
}

@article{zhang2010maximum,
  author  = {Zhang, Xiangxiong and Shu, Chi-Wang},
  title   = {On maximum-principle-satisfying high order schemes for scalar conservation laws},
  journal = {J. Comput. Phys.},
  volume  = {229},
  number  = {9},
  pages   = {3091--3120},
  year    = {2010}
}

@article{zhang2010positivity,
  author  = {Zhang, Xiangxiong and Shu, Chi-Wang},
  title   = {On positivity-preserving high order discontinuous {G}alerkin schemes for compressible {E}uler equations on rectangular meshes},
  journal = {J. Comput. Phys.},
  volume  = {229},
  number  = {23},
  pages   = {8918--8934},
  year    = {2010}
}

@article{zhang2011maximum,
  title={Maximum-principle-satisfying and positivity-preserving high-order schemes for conservation laws: survey and new developments},
  author={Zhang, Xiangxiong and Shu, Chi-Wang},
  journal={Proceedings: Mathematical, Physical and Engineering Sciences},
  pages={2752--2776},
  year={2011},
  publisher={JSTOR}
}

@book{zoutendijk1960methods,
  title={Methods of feasible directions},
  author={Zoutendijk, G.},
  volume={960},
  year={1960},
  publisher={Elsevier},
  address={Amsterdam}
}

@article{benner2016block,
  title={Block-diagonal preconditioning for optimal control problems constrained by {PDE}s with uncertain inputs},
  author={Benner, P. and Onwunta, A. and Stoll, M.},
  journal={SIAM Journal on Matrix Analysis and Applications},
  volume={37},
  number={2},
  pages={491--518},
  year={2016},
  publisher={SIAM}
}

@article{BENNER201626,
title = {Low-rank solvers for unsteady {S}tokes–{B}rinkman optimal control problem with random data},
journal = {Computer Methods in Applied Mechanics and Engineering},
volume = {304},
pages = {26-54},
year = {2016},
issn = {0045-7825},
doi = {https://doi.org/10.1016/j.cma.2016.02.004},
url = {https://www.sciencedirect.com/science/article/pii/S0045782516300263},
author = {Peter Benner and Sergey Dolgov and Akwum Onwunta and Martin Stoll}
}

@article{Antil2023,
  title={{TTRISK}: Tensor Train Decomposition Algorithm for Risk-Averse Optimization},
  author={Antil, Harbir and Dolgov, Sergey and Onwunta, Akwum},
  journal={Numerical Linear Algebra with Applications},
  volume={3},
  number={3},
  pages={e2481},
  year={2023},
  url={https://doi.org/10.1002/nla.2481.12}
 }

@article{Antil2025,
  title={Smoothed {M}oreau-{Y}osida Tensor-Train Approximation of State-Constrained Optimization Problems Under Uncertainty},
  author={Antil, Harbir and Dolgov, Sergey and Onwunta, Akwum},
  journal={Numerical Linear Algebra with Applications},
  pages={32:e70028},
  year={2025},
  url={https://doi.org/10.1002/nla.70028}
 }

@article{Antil2025b,
  title={Tensor train solution to uncertain optimization problems with shared sparsity penalty},
  author={Antil, Harbir and Dolgov, Sergey and Onwunta, Akwum},
  journal={Optimization and Engineering},
   year={2025},
  url={https://doi.org/10.1007/s11081-025-10036-7}
 }

@article{BENNER2020,
title = {Low-rank solution of an optimal control problem constrained by random {N}avier-{S}tokes equations},
journal = {International Journal on Numerical Methods for Fluids},
volume = {92},
pages = {1653-1678},
year = {2020},
author = {Peter Benner and Sergey Dolgov and Akwum Onwunta and Martin Stoll}
}

\end{document}